\documentclass{scrartcl}
\usepackage{preamble}

\begin{document}

\maketitle

\begin{abstract}
     \small{\textsc{Abstract.} We study de Rham character sheaves on a commutative connected algebraic group $G$, defined as multiplicative line bundles with integrable connection. We construct a group algebraic space $G^\flat$ representing their moduli problem on seminormal test schemes, and we investigate its functoriality and geometry. The main technical ingredient is a study of extension sheaves on the de Rham space $G_\text{dR}$. An appendix provides self-contained, elementary proofs of basic results on de Rham spaces that may be of independent interest.}
\end{abstract}

\section{Introduction}
\subsection*{Motivation from number theory}
Let $p$ and $\ell$ be distinct prime numbers, and let $G$ be a commutative connected algebraic group over $\mathbb{F}_p$. Denoting by $\operatorname{Fr}_G\colon G\to G$ the arithmetic Frobenius, the \emph{Lang isogeny}
\[\operatorname{L}_G\colon G \to G, \qquad x \mapsto \operatorname{Fr}_G(x)\,x^{-1}.\]
is a finite étale Galois covering with automorphism group $G(\mathbb{F}_p)$. Consequently, $\operatorname{L}_G$ realizes $G(\mathbb{F}_p)$ as a quotient of $\pi_1^\text{ét}(G)$. A character $\chi\colon G(\mathbb{F}_p)\to \overline{\mathbb{Q}}_\ell^\times$ then naturally induces a rank-one $\ell$-adic representation of $\pi_1^\text{ét}(G)$:
\[
\pi_1^\text{ét}(G)\twoheadrightarrow G(\mathbb{F}_p)\xrightarrow{\ \chi^{-1}\ } \overline{\mathbb{Q}}_\ell^\times.
\]
Equivalently, we obtain a rank-one $\ell$-adic local system on $G$ that we denote by $\mathscr{L}_\chi$.

\begin{samepage}
Write $m\colon G\times G\to G$ for the group law. Among rank-one $\ell$-adic local systems on $G$, the local systems above are precisely the \emph{multiplicative} ones; those admitting an isomorphism
\[
m^*\mathscr{L}\simeq \mathscr{L}\boxtimes \mathscr{L},
\]
reflecting the identity $\chi(xy)=\chi(x)\chi(y)$. More precisely, the assignment
\[\begin{array}{c}
	\left\{\begin{array}{c}
	\text{Characters}\\
	\chi\colon G(\mathbb{F}_p)\to \overline{\mathbb{Q}}_\ell^\times
	\end{array}\right\}
	\end{array} \longrightarrow \begin{array}{c}
	\left\{\begin{array}{c}
	\text{Isomorphism classes of rank-one }\ell\text{-adic local}\\
	\text{systems }\mathscr{L}\text{ on }G\text{ with }m^*\mathscr{L}\simeq\mathscr{L}\boxtimes\mathscr{L}
	\end{array}\right\}
	\end{array}\]
is an isomorphism of groups \cite[Lem.\ 2.16]{sawin2021ramanujan}.
\end{samepage}

This geometrization of characters has had far-reaching consequences in number theory. It underlies Deligne's $\ell$-adic Fourier transform, which Laumon used to give a striking simplification of Deligne's proof of Weil~II \cite{weilII,laumontransf}. Since then, Deligne's theory of weights has remained a central input in analytic number theory through trace functions and character sums; see \cite{deligne1971formes,KL,kowalski2017bilinear,forey2025arithmeticfouriertransformsfinite} for some remarkable examples.

\subsection*{Character sheaves}
Henceforth, we assume that the base field $k$ has characteristic zero. Motivated by the discussion above, we introduce the following \emph{de Rham} counterpart of the multiplicative local systems $\mathscr{L}_\chi$.

\begin{definition*}
	Let $S$ be a $k$-scheme. A \emph{character sheaf on} $G$ \emph{relative to $S$} is a line bundle $\mathscr{L}$ on $G_S$ equipped with an integrable connection $\nabla$ relative to $S$ satisfying
    \[m_S^*(\mathscr{L},\nabla)\simeq (\mathscr{L},\nabla)\boxtimes (\mathscr{L},\nabla).\]
    In the absolute case $S=\operatorname{Spec} k$, we call such an object a \emph{character sheaf on} $G$.
\end{definition*}

For reasons that will soon become apparent, we denote by $\mathrm{H}^1_m(G_\text{dR}\times S,\mathbb{G}_m)$ the group of isomorphism classes of character sheaves on $G$ relative to $S$. To get a first grip on these objects, we recall the Barsotti--Chevalley theorem: the algebraic group $G$ fits into a short exact sequence
\begin{equation*}
    0\to T\times U\to G\to A\to 0,
\end{equation*}
where $T$ is a torus, $U$ is a unipotent group, and $A$ is an abelian variety. The groups $\mathrm{H}^1_m(G_\text{dR},\mathbb{G}_m)$ have straightforward descriptions for these building blocks.

\begin{example*}
	In characteristic zero, a unipotent group $U$ is necessarily a vector group. The group $\mathrm{H}^1_m(U_\text{dR},\mathbb{G}_m)$ then identifies with the dual vector space $U^*$. Choosing coordinates $x_1,\dots,x_n$ on $U$, the isomorphism is given by
	\begin{align*}
		k^n &\to \mathrm{H}^1_m(U_\text{dR},\mathbb{G}_m)\\
		(\chi_1,\dots,\chi_n) &\mapsto \left(\mathcal{O}_U,\mathrm{d}-\chi_1\,\mathrm{d}x_1-\dots -\chi_n\,\mathrm{d}x_n\right).
	\end{align*}

    Given a torus $T$ with character group $X$ and Lie algebra $\mathfrak{t}$, the group $\mathrm{H}^1_m(T_\text{dR},\mathbb{G}_m)$ is isomorphic to $\mathfrak{t}^*/X$. After a finite extension of $k$, we may assume that $T$ is split with coordinates $t_1,\dots,t_r$. The isomorphism is then explicitly given by
	\begin{align*}
		(k/\mathbb{Z})^r &\to \mathrm{H}^1_m(T_\text{dR},\mathbb{G}_m)\\
		(\chi_1,\dots,\chi_r) &\mapsto \left(\mathcal{O}_T,\mathrm{d}-\chi_1\,\frac{\mathrm{d}t_1}{t_1}-\dots -\chi_r\,\frac{\mathrm{d}t_r}{t_r}\right).
	\end{align*}
	
	For abelian varieties $A$, every line bundle with integrable connection is a character sheaf. This is because a line bundle equipped with an integrable connection has a vanishing first Chern class, implying that it lies in $\operatorname{Pic}^0(A)$. (See \cite[Lem.~2.1.1]{Lau}.)
\end{example*}

\subsection*{A representability theorem}
In \cite{mazur2006universal}, Mazur and Messing constructed a moduli space of line bundles with integrable connection on an abelian variety. By the discussion above, this is equivalently a moduli space of character sheaves. This construction may be viewed as an early instance of the moduli spaces that later became central in non-abelian Hodge theory \cite{simpson1994moduli}. For affine groups, however, the collection of line bundles with flat connection is typically far too large to be algebro-geometric in any reasonable sense.\footnote{For instance, the "moduli space" of line bundles with integrable connection on $\mathbb{A}^1$ is the ind-scheme $\mathbb{A}^\infty = \operatorname{colim}_n \mathbb{A}^n$, whereas the moduli space of character sheaves on $\mathbb{G}_a$ is $\mathbb{G}_a$ itself.} Imposing the character condition restores finiteness and leads to a well-behaved moduli problem.

\renewcommand{\thedummy}{\Alph{dummy}}
\begin{theorem}[\ref{G flat is an algebraic space}]\label{thm A}
	There exists a smooth commutative connected group algebraic space $G^\flat$ satisfying $\dim G\leq \dim G^\flat\leq 2\dim G$, whose $S$-points parametrize character sheaves on $G$ relative to $S$ for every seminormal $k$-scheme $S$.
\end{theorem}

\begin{proof}[Sketch of the proof]
	The inspiration for this construction is the Barsotti--Weil formula \cite[Cor.~3.6]{ribeiro2025}. For an abelian variety $A$, this formula asserts that the extension sheaf $\underline{\operatorname{Ext}}^1(A,\mathbb{G}_m)$, computed in the category of abelian sheaves on the fppf site $(\textsf{Sch}/k)_\text{fppf}$, is representable by the dual abelian variety $A^\prime$. This is a moduli space of line bundles $\mathscr{L}$ on $A$ satisfying $m^*\mathscr{L}\simeq \mathscr{L}\boxtimes \mathscr{L}$.
	
	The next ingredient is the so-called \emph{de Rham space}. For a smooth algebraic variety $X$, line bundles on the de Rham space $X_\text{dR}$ correspond to line bundles with flat connection on $X$. Consequently, it can be shown that the abelian sheaf $A^\natural\colonequals \underline{\operatorname{Ext}}^1(A_\text{dR},\mathbb{G}_m)$ parametrizes character sheaves on $A$ and is representable by a commutative connected algebraic group. This provides a simpler construction of the Mazur--Messing algebraic group.
	
	For a general commutative connected algebraic group $G$ and a reduced $k$-scheme $S$, Theorem~\ref{global de rham} provides a functorial isomorphism between the $S$-points of the sheaf $G^\natural\colonequals \underline{\operatorname{Ext}}^1(G_\text{dR},\mathbb{G}_m)$ and the group $\mathrm{H}^1_m(G_\text{dR}\times S,\mathbb{G}_m)$. If $G$ is semiabelian (i.e., an extension of an abelian variety by a torus), then this abelian sheaf is representable by a commutative connected group algebraic space, thereby establishing the theorem.

    However, when $G$ contains a unipotent subgroup $U$, the sheaf $G^\natural$ is no longer representable by an algebraic space. Proposition~\ref{computation of ker gamma} shows that the issue comes from a quotient of $G^\natural$ isomorphic to the abelian sheaf $\underline{\operatorname{Ext}}^{1}(U,\mathbb{G}_m)$, studied in~\cite{ribeiro2025}, which is nonzero but vanishes on all seminormal $k$-schemes. Discarding this quotient produces the desired moduli space $G^\flat$.
\end{proof}

For the reader's convenience, we record that every smooth $k$-scheme is seminormal. The need for this regularity hypothesis stems from the fact that $G$ need not be proper. By contrast, when $G$ is an abelian variety, the $S$-points of $G^\flat$ parametrize character sheaves on $G$ relative to $S$ for \emph{every} $k$-scheme $S$. A similar phenomenon occurs in \cite{rosengarten2025restrictedpicardfunctor}, where the author proves the (almost-)representability of the Picard functor for non-proper schemes after restricting to smooth test schemes.

\subsection*{Functorial and geometric properties of the moduli space}
Motivated by the analogy between the moduli space $G^\flat$ and the Pontryagin dual of a locally compact abelian group, one might expect that a short exact sequence
\[0\to G_1\to G_2\to G_3\to 0\]
of commutative connected algebraic groups induces a short exact sequence
\begin{equation}\tag{$\ast$}\label{dual SES}
    0\to G_3^\flat\to G_2^\flat\to G_1^\flat\to 0
\end{equation}
of abelian fppf sheaves. Somewhat surprisingly, this fails in general; see Example~\ref{flats are not exact} for a counterexample in which $G_1$ is an elliptic curve and $G_3=\mathbb{G}_m$. Nevertheless, the following theorem shows that~\eqref{dual SES} is exact in many cases.

\begin{theorem}[\ref{exact sequence of flats}, \ref{exact sequence of flats abelian}, \ref{extensions of linear groups}]
The sequence~\textnormal{(\ref{dual SES})} is always left exact and right exact. Moreover, it is exact in the middle in each of the following cases:
\begin{enumerate}
\item $G_1$ is linear and $G_3$ is an abelian variety;
\item $G_1,G_2,G_3$ are all linear;
\item $G_1,G_2,G_3$ are all abelian varieties.
\end{enumerate}
\end{theorem}

We next turn to the geometry of these moduli spaces. For an abelian variety $A$, certain \emph{big} subsets of $A^\flat$ play a central role in non-abelian Hodge theory \cite{simpson1993subspaces} and in the formulation of generic vanishing theorems for de Rham cohomology \cite{schnell2015holonomic}. For a torus $T$, similar \emph{big} subsets of $T^\flat$ were studied in \cite{sabbah1992lieu}. Here, we extend the definition of these subsets to arbitrary commutative connected algebraic groups.

\begin{definition*}[\ref{def linear subspace}, \ref{def generic subspace}]
For an epimorphism $\rho\colon G\twoheadrightarrow \widetilde{G}$ with connected kernel, the image of $\rho^\flat\colon \widetilde{G}^\flat\hookrightarrow G^\flat$ is said to be a \emph{linear subspace} of $G^\flat$. A \emph{generic subspace} of $G^\flat$ is the complement of a finite union of translates of linear subspaces of $G^\flat$ with positive codimension.
\end{definition*}

In the affine case these notions reduce to familiar linear-algebraic ones, while for abelian varieties they recover the subsets studied by Simpson and Schnell.

\begin{example*}[\ref{examples of linear subvarieties}]
Let $T$ be a torus with character group $X$ and Lie algebra $\mathfrak{t}$. A linear subspace of $T^\flat\simeq \mathfrak{t}^*/X$ is of the form $V/Y$, where $Y$ is a subgroup of $X$ and $V$ is the linear subspace of $\mathfrak{t}^*$ generated by $Y$, in the sense of linear algebra. For a unipotent group $U$, a linear subspace of $U^\flat \simeq U^*$ corresponds to a linear subspace of the underlying vector space of $U^*$, in the sense of linear algebra.

For an abelian variety $A$, linear subspaces of $A^\flat\simeq A^\natural$ were first studied by Simpson, who termed them \emph{triple tori} \cite[p.\ 365]{simpson1993subspaces}. Schnell refers to translates of linear subspaces of $A^\flat$ as \emph{linear subvarieties} \cite[Def.\ 2.3]{schnell2015holonomic}.
\end{example*}

We refer to Subsection~\ref{Subsection on generic subspaces} for a detailed study of linear and generic subspaces. For example, when $k=\mathbb{C}$, we consider the analytification $G^\flat_\text{an}$ of the moduli space $G^\flat$, and we show that if $V$ is a generic subspace of $G^\flat$, then $V_\text{an}\subset G^\flat_\text{an}$ is open and dense with respect to the analytic topology; see Proposition~\ref{generic implies analytic open dense}.

\subsection*{Comparison with Betti character sheaves}
Replacing the coefficients of de Rham cohomology by their Betti counterparts yields another natural geometrization of characters: rank-one local systems of $\mathbb{C}$-vector spaces on the analytification $G_\text{an}$.\footnote{In this setting, multiplicativity is automatic, so these are precisely the multiplicative local systems on $G_\text{an}$.} Equivalently, they correspond to characters of the topological fundamental group $\pi_1(G_\text{an})\to \mathbb{C}^\times$.

These objects have a well-studied moduli space in the literature: the character variety
\[\operatorname{Char}(G)\colonequals \operatorname{Spec}(\mathbb{C}[\pi_1(G_\text{an})]).\]
Its $\mathbb{C}$-points identify with characters $\pi_1(G_\text{an})\to\mathbb{C}^\times$, hence with rank-one local systems on $G_\text{an}$. Moreover, the Hopf algebra structure on $\mathbb{C}[\pi_1(G_\text{an})]$ endows $\operatorname{Char}(G)$ with a natural structure of algebraic group, corresponding to tensor product of local systems.

On the de Rham side, a character sheaf $(\mathscr{L},\nabla)$ on $G$ determines a rank-one local system on $G_\text{an}$ by taking horizontal sections. This construction induces a holomorphic morphism of groups
\[G^\flat_\text{an}\to \operatorname{Char}(G)_\text{an}.\]
Our final theorem, a consequence of the Riemann--Hilbert correspondence, identifies these analytic moduli spaces for semiabelian varieties.

\begin{theorem}[\ref{character sheaves and character variety}]\label{thm C}
The Riemann--Hilbert map $\normalfont G^\flat_\text{an}\to \operatorname{Char}(G)_\text{an}$ is surjective. Moreover, it is an isomorphism if and only if $G$ is a semiabelian variety.
\end{theorem}

For unipotent groups, the Riemann--Hilbert map is far from being an isomorphism: their character varieties are trivial, whereas their nontrivial character sheaves are the abundant exponential connections. These are the prototypical examples of irregular connections. Accessing them on the Betti side would require an enhancement of the category of local systems, such as the irregular constructible sheaves of Kuwagaki \cite{kuwagaki} or the wild Betti sheaves of Scholze \cite{scholze2025wildbettisheaves}. We leave the development of this direction to future work.

\vspace{1em}

\textsc{Acknowledgements.} Most of this work was carried out while I was a doctoral student of Javier Fresán, whose guidance and support are gratefully acknowledged. I would like to thank Thomas Krämer, who suggested that the preceding definition of linear subspaces might be interesting beyond abelian varieties, and Michel Brion, to whom I owe much of what I know about algebraic groups. 

This paper is dedicated to Gérard Laumon. This project grew out of an attempt to understand his paper \cite{Lau}; after kindly listening to my questions, he offered characteristically lucid advice on how to approach them. As a naive doctoral student, I promptly tried a different route---no doubt with Laumon's ideas firmly at work in the background---and most of the results of this paper were proved the following week.

\renewcommand{\thedummy}{\arabic{dummy}}\numberwithin{dummy}{section}

\section{Cartier duality of commutative group stacks}
\subsection{Definitions and first properties}\label{commutative group stacks}

Let $\textsf{Lat}$ be the full subcategory of $\textsf{Ab}$ whose objects are free abelian groups of finite rank. Given an $\infty$-category $\textsf{C}$ with finite products, we define its category of \emph{abelian group objects} $\textsf{Ab}(\textsf{C})$ as the $\infty$-category of functors $\textsf{Lat}^{\text{op}}\to \textsf{C}$ that commute with finite products. (See \cite[\S1.2]{lurie2017elliptic} for more.)

When $\textsf{C}$ is the category of sets, $\textsf{Ab}(\textsf{C})$ coincides with the usual category of abelian groups. If $\textsf{C}$ is the $\infty$-category of anima\footnote{We refer to \cite[\S 5.1]{purity} for the definition of animation and of the $\infty$-category $\mathsf{An}$ of anima (also known as the $\infty$-category of spaces or $\infty$-groupoids).}, $\textsf{Ab}(\textsf{C})$ is the $\infty$-category $\mathsf{An}(\mathsf{Ab})$ of animated abelian groups which, by the Dold--Kan correspondence, is equivalent to $\mathsf{D}^{\leq 0}(\mathsf{Ab})$. More generally, if $\textsf{X}$ is an $\infty$-topos, then $\textsf{Ab}(\textsf{X})$ is the $\infty$-category of sheaves on $\textsf{X}$ valued in animated abelian groups. This motivates the definition below.

\begin{definition}[Commutative group stack]\label{def commutative group stack}
Let $\textsf{C}$ be a 1-site. A \emph{commutative group stack} on $\textsf{C}$ is a sheaf
\[\textsf{C}^\text{op}\to \textsf{An}_{\leq 1}(\textsf{Ab}),\]
where $\textsf{An}_{\leq 1}(\textsf{Ab})$ is the full subcategory of $\textsf{An}(\textsf{Ab})$ consisting of animated abelian groups $M$ with $\pi_i(M)=0$ for $i\neq 0,1$.
\end{definition}

Equivalently, a commutative group stack on $\textsf{C}$ is an abelian group object in the 2-topos $\textsf{Sh}(\textsf{C}, \textsf{Grpd})$. These objects were first suggested by Grothendieck and subsequently studied by Deligne in \cite[Exp.\ XVIII, \S1.4]{SGA43} under the name \emph{champs de Picard strictement commutatifs}.

\begin{remark}
The forgetful functor $\textsf{An}(\textsf{Ab})\to \textsf{An}$ restricts to $\textsf{An}_{\leq 1}(\textsf{Ab})\to \textsf{Grpd}$. Concretely, an object of $\textsf{An}_{\leq 1}(\textsf{Ab})$ is a groupoid $M$ endowed with a symmetric monoidal structure $+$ such that the set of isomorphism classes $\pi_0(M)$ is a group. Additionally, for all $x,y$ in $M$, the symmetry constraints $x+x\to x+x$ and $x+y\to y+x\to x+y$ are supposed to be the identity maps.
\end{remark}

The following proposition, a consequence of the Dold--Kan correspondence, provides a concrete characterization of commutative group stacks. This characterization will underpin virtually every computation involving these objects throughout this paper.

\begin{proposition}[Deligne]\label{DK}
	Let $\normalfont\textsf{C}$ be a 1-site. Denote by $\normalfont\textsf{Ch}^{[-1,0]}(\textsf{Ab}(\textsf{C}))$ the category of complexes of abelian sheaves\footnote{Denoting the category of abelian sheaves on $\textsf{C}$ by $\textsf{Ab}(\textsf{C})$ is an abuse of notation, since this is the category of abelian group objects \emph{on the associated 1-topos}.} on $\normalfont\textsf{C}$ concentrated in degrees $-1$ and $0$. Similarly, denote by $\normalfont\mathsf{D}^{[-1,0]}(\textsf{Ab}(\textsf{C}))$ the full subcategory of $\normalfont\mathsf{D}(\textsf{Ab}(\textsf{C}))$ consisting of the objects whose cohomologies are concentrated in degrees $-1$ and $0$. The functor
	\begin{align*}
	     \normalfont\textsf{Ch}^{[-1,0]}(\textsf{Ab}(\textsf{C})) &\to \normalfont\textsf{Sh}(\textsf{C},\textsf{An}_{\leq 1}(\textsf{Ab})) \\
		[\mathscr{F}\to \mathscr{G}] &\mapsto [\mathscr{G}/\mathscr{F}]
	\end{align*}
	factors through $\normalfont\mathsf{D}^{[-1,0]}(\textsf{Ab}(\textsf{C}))$, and the induced functor $\normalfont\mathsf{D}^{[-1,0]}(\textsf{Ab}(\textsf{C}))\to \textsf{Sh}(\textsf{C},\textsf{An}_{\leq 1}(\textsf{Ab}))$ is an equivalence of $\infty$-categories.
\end{proposition}

\begin{proof}
Let $\textsf{D}$ be an $\infty$-category that admits small limits. Recall that a functor $\mathscr{F}\colon \textsf{C}^\text{op}\to \textsf{D}$ is a sheaf if it preserves finite products and if, for every covering $X\to S$ in $\textsf{C}$, the natural morphism
\[\mathscr{F}(S)\to \lim \Bigg[\mathscr{F}(X) \rightrightarrows \mathscr{F}(X\times_S X)\rightrightrightarrows \mathscr{F}(X\times_S X\times_S X)\rightrightrightrightarrows\cdots \Bigg]\]
is invertible. Functors preserving finite products and satisfying a similar descent condition with respect to \emph{hypercovers} (see \cite[Def.~A.4.19]{heyer20246functorformalismssmoothrepresentations} for a precise definition) are called \emph{hypersheaves}. We denote by $\textsf{HSh}(\textsf{C},\textsf{D})$ the full subcategory of $\textsf{Sh}(\textsf{C},\textsf{D})$ constituted of the hypersheaves.

Consider a sheaf of abelian groups $\mathscr{G}$ on $\textsf{C}$. Using the natural inclusion $\textsf{Ab}\hookrightarrow \mathsf{D}(\textsf{Ab})$, $\mathscr{G}$ can be viewed as a presheaf valued in the derived $\infty$-category $\mathsf{D}(\textsf{Ab})$. While this often fails to satisfy descent, one can sheafify to obtain a functor $\textsf{Ab}(\textsf{C})\to \textsf{Sh}(\textsf{C},\mathsf{D}(\textsf{Ab}))$. This induces a $t$-exact fully faithful functor $\mathsf{D}(\textsf{Ab}(\textsf{C}))\hookrightarrow \textsf{Sh}(\textsf{C},\mathsf{D}(\textsf{Ab}))$, whose essential image is precisely $\textsf{HSh}(\textsf{C},\mathsf{D}(\textsf{Ab}))$ \cite[Cor.\ 2.1.2.3]{lurie2018spectral}.

The distinction between hypersheaves and ordinary sheaves is a phenomenon unique to $\infty$-categories. More precisely, the inclusion $\textsf{HSh}(\textsf{C},\textsf{D})\hookrightarrow \textsf{Sh}(\textsf{C},\textsf{D})$ is an equivalence as soon as $\textsf{D}$ is a $n$-category for some \emph{finite} $n$ \cite[Lem.\ 6.5.2.9]{lurie2009HTT}. Consequently,
\[\mathsf{D}^{[-1,0]}(\textsf{Ab}(\textsf{C}))\to\textsf{Sh}(\textsf{C},\mathsf{D}^{[-1,0]}(\textsf{Ab}))\]
is an equivalence of $\infty$-categories. 

Finally, recall that the Dold--Kan correspondence gives an equivalence of $\infty$-categories $\mathsf{D}^{\leq 0}(\textsf{Ab})\xrightarrow{\ \sim\ } \textsf{An}(\textsf{Ab})$. Here, the $i$-th homotopy group of an object on the right is isomorphic to the $-i$-th cohomology group of the corresponding object on the left. In particular, this restricts to an equivalence $\mathsf{D}^{[-1,0]}(\textsf{Ab})\xrightarrow{\ \sim\ } \textsf{An}_{\leq 1}(\textsf{Ab})$ leading to the desired equivalence $\mathsf{D}^{[-1,0]}(\textsf{Ab}(\textsf{C}))\xrightarrow{\ \sim\ } \textsf{Sh}(\textsf{C},\textsf{An}_{\leq 1}(\textsf{Ab}))$.
\end{proof}

Let us momentarily denote by $\mathcal{G}^\circ$ the object in the derived category of abelian sheaves associated with a commutative group stack $\mathcal{G}$. (We will soon begin identifying these objects.) Here, we outline some consequences of the proof of Proposition~\ref{DK}.

\begin{remark}$ $
	\begin{enumerate}
		\item Let $\mathscr{G}$ be an abelian sheaf. Composing with the natural inclusion $\textsf{Ab}\hookrightarrow \textsf{An}_{\leq 1}(\textsf{Ab})$, we obtain a commutative group stack denoted by the same symbol. Then, $\mathscr{G}^\circ$ is isomorphic to the complex $[0\to \mathscr{G}]$. Similarly, $(\mathsf{B}\mathscr{G})^\circ$ is isomorphic to $[\mathscr{G}\to 0]$.
		\item Given a commutative group stack $\mathcal{G}$, the abelian sheaf $\mathscr{H}^0(\mathcal{G}^\circ)$ is isomorphic to the coarse moduli sheaf of $\mathcal{G}$. This is the sheafification of the presheaf sending an object $X$ of $\textsf{C}$ to the group of isomorphism classes $\pi_0(\mathcal{G}(X))$ of $\mathcal{G}(X)$. Similarly, $\mathscr{H}^{-1}(\mathcal{G}^\circ)$ is isomorphic to the automorphism sheaf of a zero section of $\mathcal{G}$.
		\item Let $\mathcal{G}$ and $\mathcal{H}$ be commutative group stacks. We denote by $\underline{\operatorname{Hom}}(\mathcal{G},\mathcal{H})$ the internal Hom of commutative group stacks. This is another commutative group stack satisfying $\underline{\operatorname{Hom}}(\mathcal{G},\mathcal{H})^\circ\simeq \tau_{\leq 0}\mathsf{R}\underline{\operatorname{Hom}}(\mathcal{G}^\circ,\mathcal{H}^\circ)$.\vspace{-2em}
	\end{enumerate}
\end{remark}

Henceforth, we will focus on the big fppf site $\textsf{C}=(\textsf{Sch}/S)_\text{fppf}$ associated with a base scheme $S$. There is a distinguished abelian sheaf on this site: the multiplicative group $\mathbb{G}_m$, sending a $S$-scheme $T$ to the group of units $\Gamma(T,\mathcal{O}_T)^\times$. Its classifying stack, $\textsf{B}\mathbb{G}_m$, will play a role in this story similar to that of the circle group in the Pontryagin duality of locally compact abelian groups.

\begin{definition}\label{def stacky cartier dual}
Let $\mathcal{G}$ be a commutative group stack on $(\textsf{Sch}/S)_\text{fppf}$. We define its \emph{Cartier dual} $\mathcal{G}^D$ as $\underline{\operatorname{Hom}}(\mathcal{G},\mathbb{G}_m)$ and its \emph{stacky Cartier dual} $\mathcal{G}^\vee$ as $\underline{\operatorname{Hom}}(\mathcal{G},\mathsf{B}\mathbb{G}_m)$.
\end{definition}

From our perspective, the stacky Cartier dual $\mathcal{G}^\vee$ is the \emph{true} dual of a commutative group stack while $\mathcal{G}^D$ is a good enough approximation in some situations. The following remark explains our motivation for considering $\mathcal{G}^\vee$.

\begin{remark}\label{objects of stacky Cartier dual}
	Let $\mathcal{G}$ be a commutative group stack on $(\textsf{Sch}/S)_\text{fppf}$, and denote is group law by $m\colon \mathcal{G}\times\mathcal{G}\to \mathcal{G}$. Given an $S$-scheme $T$, the groupoid $\mathcal{G}^\vee(T)$ can be described as follows: its objects are pairs $(\mathscr{L},\alpha)$, where $\mathscr{L}$ is a line bundle on $\mathcal{G}\times_S T$ and $\alpha$ is an isomorphism $m^*\mathscr{L}\to \mathscr{L}\boxtimes\mathscr{L}$ making diagrams (A) and (B) just above \cite[Rem.\ 3.13]{brochard2021duality} commute. A morphism from $(\mathscr{L},\alpha)$ to $(\mathscr{L}^\prime, \alpha^\prime)$ is an isomorphism of line bundles $\varphi\colon \mathscr{L}\to \mathscr{L}^\prime$ making the diagram
\[\begin{tikzcd}[ampersand replacement=\&]
	{m^*\mathscr{L}} \& {\mathscr{L}\boxtimes\mathscr{L}} \\
	{m^*\mathscr{L}^\prime} \& {\mathscr{L}^\prime\boxtimes\mathscr{L}^\prime}
	\arrow["\alpha", from=1-1, to=1-2]
	\arrow["{m^*\varphi}"', from=1-1, to=2-1]
	\arrow["\varphi\boxtimes\varphi", from=1-2, to=2-2]
	\arrow["{\alpha^\prime}", from=2-1, to=2-2]
\end{tikzcd}\]
commute. The symmetric monoidal structure on $\mathcal{G}^\vee(T)$ sends $(\mathscr{L},\alpha)$ and $(\mathscr{L}^\prime,\alpha^\prime)$ to the pair $(\mathscr{L}\otimes\mathscr{L}^\prime,\alpha\cdot \alpha^\prime)$, where $\alpha\cdot\alpha^\prime$ is defined as the composition
\[m^*(\mathscr{L}\otimes\mathscr{L}^\prime)\simeq m^*\mathscr{L}\otimes m^*\mathscr{L}^\prime \xrightarrow{\alpha\otimes\alpha^\prime} (\mathscr{L}\boxtimes \mathscr{L})\otimes (\mathscr{L}^\prime\boxtimes \mathscr{L}^\prime)\simeq (\mathscr{L}\otimes \mathscr{L}^\prime)\boxtimes (\mathscr{L}\otimes \mathscr{L}^\prime).\]
The reader might also be interested in comparing this description to \cite[Exp.\ VII, \S 1.1.6]{raynaud2006groupes} or \cite[\S I.2.3]{moret1985pinceaux}.
\end{remark} 

In most cases of interest in this paper, we will work with commutative group stacks $\mathcal{G}$ defined by complexes $[\mathscr{F}\to \mathscr{G}]$, where the map $d\colon \mathscr{F}\to \mathscr{G}$ is a monomorphism. In this case, $\mathcal{G}\simeq \operatorname{coker} d$ takes values in the full subcategory $\textsf{Ab}$ of $\textsf{An}_{\leq 1}(\textsf{Ab})$. The following proposition gives a convenient criterion for computing the stacky Cartier duals of these objects.

\begin{proposition}\label{computation of stacky cartier duals}
Let $\mathscr{G}$ be an abelian sheaf on $(\normalfont\textsf{Sch}/S)_{\normalfont\text{fppf}}$. Then the stacky Cartier dual of the classifying stack $\mathsf{B}\mathscr{G}$ is isomorphic to $\mathscr{G}^D$. If $\underline{\operatorname{Ext}}^1(\mathscr{G},\mathbb{G}_m)=0$, then $\mathscr{G}^\vee\simeq \mathsf{B}\mathscr{G}^D$. Similarly, if $\mathscr{G}^D=0$, then $\mathscr{G}^\vee\simeq \underline{\operatorname{Ext}}^1(\mathscr{G},\mathbb{G}_m)$.
\end{proposition}

\begin{proof}
	The first isomorphism is simply the fact that $((\mathsf{B}\mathscr{G})^\vee)^\circ$ is given by
	\[\tau_{\leq 0}\mathsf{R}\underline{\operatorname{Hom}}(\mathscr{G}[1],\mathbb{G}_m[1])\simeq \underline{\operatorname{Hom}}(\mathscr{G},\mathbb{G}_m).\]
	Similarly, using the explicit description of $(\mathscr{G}^\vee)^\circ$ as $\tau_{\leq 0}\mathsf{R}\underline{\operatorname{Hom}}(\mathscr{G},\mathbb{G}_m[1])$, the other isomorphisms follow from the fact that an object of a derived category, whose cohomology is concentrated in a single degree, is isomorphic to that cohomology in the corresponding degree.
\end{proof}

\subsection{Computations of Cartier duals}\label{section computation of cartier duals}

Consider a base scheme $S$, and let $\mathscr{G}$ be an abelian sheaf on $(\textsf{Sch}/S)_\text{fppf}$, regarded as a commutative group stack. Our approach to computing the stacky Cartier dual $\mathscr{G}^\vee$ predominantly follows the method outlined in Proposition~\ref{computation of stacky cartier duals}. We typically either prove the vanishing of $\underline{\operatorname{Ext}}^1(\mathscr{G}, \mathbb{G}_m)$ and calculate $\mathscr{G}^D$, or we establish that $\mathscr{G}^D$ vanishes and compute $\underline{\operatorname{Ext}}^1(\mathscr{G}, \mathbb{G}_m)$. This section focuses on computing the Cartier duals $\mathscr{G}^D$ for some abelian sheaves $\mathscr{G}$ of interest.

\begin{proposition}\label{cartier dual of abelian schemes and tori}
	Let $A$ be an abelian scheme and $T$ be a torus over a base scheme $S$. Then the Cartier dual $A^D$ of $A$ vanishes, and the Cartier dual $X$ of $T$ is representable by a group scheme that is étale locally isomorphic to a finite product of copies of the constant group scheme $\mathbb{Z}$. Moreover, the Cartier dual of $X$ is isomorphic to $T$.
\end{proposition}

\begin{proof}
Denote by $p\colon A\to S$ the structure map, and let $X$ be an $S$-scheme. By the universal property of the global spectrum, a morphism of schemes $A_X\to \mathbb{G}_{m,X}$ over $X$ is equivalent to a morphism of $\mathcal{O}_X$-algebras 
\[\mathcal{O}_X[x,x^{-1}]\to p_{X,*}\mathcal{O}_{A_X}\simeq  \mathcal{O}_X,\]
where $p_X\colon A_X\to X$ is the base change of $p$ \cite[Tag \href{https://stacks.math.columbia.edu/tag/0E0L}{0E0L}]{Stacks}. It follows that every morphism of schemes $A_X\to \mathbb{G}_{m,X}$ over $X$ must be constant. If it is a morphism of groups, it has to be trivial. This proves that $A^D$ vanishes. The statements about tori are proven in \cite[Exp.\ X, Cor.\ 5.7]{sga32}.
\end{proof}

Moving forward, we focus on the case where the base scheme $S$ is the spectrum of a characteristic zero field $k$. The following proposition recalls the standard setting for Cartier duality \cite[Exp.\ $\text{VII}_\text{B}$, Prop.\ 2.2.2]{sga31}. (Proposition~\ref{affine cartier duality} is independent of the characteristic of $k$, but every other result below requires it to be zero.)

\begin{proposition}\label{affine cartier duality}
	Let $G=\operatorname{Spec}R$ be an affine commutative group scheme over $k$. Its Cartier dual $G^D$ is represented by the formal group $\operatorname{Spf}R^*$, where $R^*$ is the dual Hopf algebra. Moreover, the double dual $(G^D)^D$ is naturally isomorphic to $G$.
\end{proposition}

Let $U$ be a commutative unipotent algebraic group over $k$. Since the exponential map gives an isomorphism between $U$ and its Lie algebra, $U$ is necessarily a vector group \cite[Prop.\ 14.32]{M} and we denote by $U^*$ its vector space dual.

\begin{proposition}\label{computation of cartier duals}
Let $U$ be a commutative unipotent algebraic group with dual $U^*$. Denote by $\widehat{U}$ and $\widehat{U^*}$ their formal completions along the zero sections. Then $U^D\simeq \widehat{U^*}$ and $\widehat{U}^D\simeq U^*$.
\end{proposition}

\begin{proof}
The computation $U^D\simeq  \widehat{U^*}$ follows from the fact that the dual of $\operatorname{Sym}(U^*)$ is the completion of $\operatorname{Sym}(U)$ at the ideal of degree one elements. (Upon a choice of basis, this is nothing but the isomorphism $k[x_1,\dotsc,x_n]^*\simeq  k\llbracket x_1,\dotsc,x_n\rrbracket$.) The other computation follows from this one by duality.
\end{proof}

Given a $k$-algebra $R$, we have that $\widehat{\mathbb{G}}_a(R)$ is the group of nilpotent elements in $R$ and $\widehat{\mathbb{G}}_m(R)$ is that of unipotent elements. We remark that the formal groups $\widehat{\mathbb{G}}_a$ and $\widehat{\mathbb{G}}_m$ are isomorphic via the map
\begin{align*}\widehat{\mathbb{G}}_m(R) &\to \widehat{\mathbb{G}}_a(R) \\ 1+x &\mapsto \log(1+x).\end{align*}
This phenomenon is a general property of formal groups in characteristic zero, and it simplifies their study.

\begin{proposition}[Cartier]\label{Cartier lemma}
Let $G$ be a commutative algebraic group over $k$ and let $\mathfrak{g}$ be its Lie algebra, seen as a vector group. The formal completions of $G$ and of $\mathfrak{g}$ along their zero sections coincide.
\end{proposition}

\begin{proof}
Since an algebraic group and its formal completion share the same Lie algebra, the composition
\[\begin{tikzcd}
	{\left\{\begin{array}{c}       \text{Algebraic}\\       \text{groups over }k     \end{array}\right\}} & {\left\{\begin{array}{c}       \text{Infinitesimal formal}\\      \text{groups over }k     \end{array}\right\}} & {\left\{\begin{array}{c}       \text{Lie algebras}\\       \text{over }k     \end{array}\right\}}
	\arrow["{\widehat{(-)}}", from=1-1, to=1-2]
	\arrow["{\operatorname{Lie}(-)}", from=1-2, to=1-3]
\end{tikzcd}\]
is the functor associating an algebraic group to its Lie algebra. In particular, $G$ and $\mathfrak{g}$, seen as a vector group, have the same image by the composition above. Now, by \cite[Exp.\ \texorpdfstring{VII\textsubscript{B}}{VIIB}, Cor.\ 3.3.2]{sga31}, the functor on the right is an equivalence of categories. In particular, $G$ and $\mathfrak{g}$ have isomorphic formal completions.   
\end{proof}

The preceding propositions enable us to compute the Cartier dual of the formal completions of commutative algebraic groups. This result, with a slightly different proof, also appears in \cite[Lem.\ A.3.1]{barbieri2009sharp}.

\begin{corollary}\label{cartier dual of formal completion}
Let $G$ be a commutative algebraic group over $k$ and denote by $\mathfrak{g}$ its Lie algebra. The Cartier dual of the formal completion $\widehat{G}$ is naturally isomorphic to $\mathfrak{g}^*$. This also coincides with the \emph{invariant} differentials $\Omega_G$ of $G$.
\end{corollary}

We are now in position to compute the Cartier dual of the de Rham space $G_\text{dR}\simeq G/\widehat{G}$ of a commutative connected algebraic group $G$.

\begin{proposition}\label{cartier dual of dR}
For a commutative connected algebraic group $G$ over $k$, the Cartier dual of $G_{\normalfont\text{dR}}$ vanishes.
\end{proposition}

\begin{proof}
Recall from the Barsotti--Chevalley theorem \cite[Thm.\ 8.28 and Cor.\ 16.15]{M} that $G$ is an extension of an abelian variety $A$ by a product of a torus $T$ and a unipotent group $U$. Since the de Rham functor $(-)_\text{dR}$ is exact, the de Rham space $G_\text{dR}$ is an extension of $A_\text{dR}$ by $T_\text{dR}\times U_\text{dR}$. Thus, we have an induced exact sequence
\[0\to \underline{\operatorname{Hom}}(A_\text{dR},\mathbb{G}_m)\to \underline{\operatorname{Hom}}(G_\text{dR},\mathbb{G}_m)\to \underline{\operatorname{Hom}}(T_\text{dR},\mathbb{G}_m)\times \underline{\operatorname{Hom}}(U_\text{dR},\mathbb{G}_m),\]
and so it suffices to prove the result when $G$ is an abelian variety, a torus, or a unipotent group.

Because the functor $\underline{\operatorname{Hom}}(-,\mathbb{G}_m)$ is left-exact, the vanishing of $G_\text{dR}^D$ is equivalent to the morphism $\underline{\operatorname{Hom}}(G,\mathbb{G}_m)\to \underline{\operatorname{Hom}}(\widehat{G},\mathbb{G}_m)$ being a monomorphism. We verify this in the relevant particular cases. For abelian varieties, their vanishing Cartier dual ensures this property. For a unipotent group $U$, the relevant morphism is isomorphic to $\widehat{U^*}\to U^*$, which is likewise monic per Corollary~\ref{inclusion of formal completion}.

For a torus $T$, the question of whether $\underline{\operatorname{Hom}}(T,\mathbb{G}_m)\to \underline{\operatorname{Hom}}(\widehat{T},\mathbb{G}_m)$ is a monomorphism is local on the base, allowing us to assume that $T=\mathbb{G}_m$. Then the statement boils down to the following: given a connected $k$-algebra $R$, if
\begin{alignat*}{2}
	\operatorname{Uni}(B)&\hookrightarrow B^\times &&\to B^\times\\
	x\hspace{1.1em} &\mapsto \hspace{0.4em} x &&\mapsto \hspace{0.07em} x^n
\end{alignat*}
is the unit map for every $R$-algebra $B$, then $n=0$. Here, $\operatorname{Uni}(B)$ denotes the group of unipotent elements in $B$. This can be proven by choosing $B=R[z]/(z-1)^{r}$ for some $r>n$, leading to the desired result.
\end{proof}

The hypothesis that $G$ is connected in the Proposition~\ref{cartier dual of dR} is essential. In general, a commutative algebraic group $G$ over $k$ fits into a short exact sequence
\[0\to G^0\to G\to \pi_0(G)\to 0,\]
where $G^0$ is the connected component of $G$ containing the identity, and $\pi_0(G)$ is the finite étale group of connected components. Proposition~\ref{de rham is an epimorphism of presheaves} states that the natural map $\pi_0(G)\to \pi_0(G)_\text{dR}$ is an isomorphism, and the preceding result then implies that the Cartier duals of $G_\text{dR}$ and $\pi_0(G)$ coincide.

\subsection{Extensions by the multiplicative group}\label{extensions section}

As in the previous section, consider a base scheme $S$ and an abelian sheaf $\mathscr{G}$ on $(\textsf{Sch}/S)_\text{fppf}$. To continue our strategy outlined in the beginning of Section~\ref{section computation of cartier duals}, we now focus on computing the extension sheaves $\underline{\operatorname{Ext}}^1(\mathscr{G},\mathbb{G}_m)$. We start by recalling some fundamental computations.

\begin{proposition}\label{T'}
	Let $A$ be an abelian scheme and $T$ be a torus over $S$ with Cartier dual $X$. Then $\underline{\operatorname{Ext}}^1(A,\mathbb{G}_m)$ is isomorphic to the dual abelian scheme of $A$, while both $\underline{\operatorname{Ext}}^1(T,\mathbb{G}_m)$ and $\underline{\operatorname{Ext}}^1(X,\mathbb{G}_m)$ vanish.
\end{proposition}

\begin{proof}
	The identification of $\underline{\operatorname{Ext}}^1(A,\mathbb{G}_m)$ with the dual abelian scheme is commonly referred to as the \emph{Barsotti--Weil formula}. A complete proof may be found in \cite[Cor.~3.6]{ribeiro2025}. (Alternatively, this identification also follows directly from Lemma~\ref{vanishing lemma}.) The vanishing of $\underline{\operatorname{Ext}}^1(T,\mathbb{G}_m)$ was established in \cite[Exp.\ VIII, Prop.\ 3.3.1]{raynaud2006groupes}. Finally, the vanishing of $\underline{\operatorname{Ext}}^1(X,\mathbb{G}_m)$ is local on $S$, so we may reduce to the case $X \simeq \mathbb{Z}$. In this case, the claim follows from the exactness of the functor $\underline{\operatorname{Hom}}(\mathbb{Z},-)$, which is isomorphic to the identity functor.
\end{proof}

We now describe the strategy of \cite[\S 2]{ribeiro2025} for computing extension sheaves, specialized to our setting. Let $\mathscr{G}$ and $\mathscr{A}$ be abelian sheaves on $(\textsf{Sch}/S)_\text{fppf}$. As $\underline{\operatorname{Ext}}^i(\mathscr{G},\mathscr{A})$ is the sheafification of the presheaf $X\mapsto \operatorname{Ext}^i_X(\mathscr{G},\mathscr{A})$ \cite[Prop.\ V.6.1]{SGA41}, there is a natural morphism of groups
\[\operatorname{Ext}^i_X(\mathscr{G},\mathscr{A})\to \underline{\operatorname{Ext}}^i(\mathscr{G},\mathscr{A})(X),\]
functorial on $\mathscr{G}$, $\mathscr{A}$, and $X$. The following simple result, also contained in \cite[Prop.\ V.6.1]{SGA41}, gives a relation between these objects

\begin{proposition}\label{sheafification map}
	Let $X$ be a scheme over $S$ and $\mathscr{G},\mathscr{A}$ be abelian sheaves on $\normalfont(\textsf{Sch}/S)_\text{fppf}$. There exists an exact sequence
    \[0\to \mathrm{H}^1(X,\underline{\operatorname{Hom}}(\mathscr{G},\mathscr{A}))\to \operatorname{Ext}^1_X(\mathscr{G},\mathscr{A})\to \underline{\operatorname{Ext}}^1(\mathscr{G},\mathscr{A})(X)\to \mathrm{H}^2(X,\underline{\operatorname{Hom}}(\mathscr{G},\mathscr{A})),\]
    that is functorial on $\mathscr{G}$, $\mathscr{A}$, and $X$.
\end{proposition}

\begin{proof}
	The group of sections of $\underline{\operatorname{Hom}}(\mathscr{G},\mathscr{A})$ over $X$ is, by definition, $\operatorname{Hom}_X(\mathscr{G},\mathscr{A})$. In particular, there is a Grothendieck spectral sequence, often called \emph{local-to-global spectral sequence}, whose five-term exact sequence yields the desired result.
\end{proof}

Now that we have established a connection between extension \emph{sheaves} and extension \emph{groups}, we will study a method for computing the latter. The following proposition, initially suggested by Grothendieck in \cite[Exp.~VII, Rem.~3.5.4]{raynaud2006groupes} and partially developed by Breen in \cite{breen1969extensions}, has been independently proven by Deligne (in a letter to Breen available in \cite[App.~B]{ribeiro}) and by Clausen--Scholze \cite[Thm.~4.10]{scholze2019condensed}.

\begin{proposition}[Breen--Deligne resolution]\label{BD resolution}
Let $\mathscr{G}$ be an abelian sheaf on $\normalfont(\textsf{Sch}/S)_\text{fppf}$. There exists a functorial resolution of the form
\[\cdots \to \bigoplus_{j=1}^{n_i}\mathbb{Z}[\mathscr{G}^{r_{i,j}}]\to \cdots \to \mathbb{Z}[\mathscr{G}^3]\oplus \mathbb{Z}[\mathscr{G}^2]\to \mathbb{Z}[\mathscr{G}^2]\to \mathbb{Z}[\mathscr{G}]\to \mathscr{G},\]
where the $n_i$ and $r_{i,j}$ are all positive integers.
\end{proposition}

Clausen and Scholze's proof shows that the first terms of the resolution can be chosen in the following way:
\[\mathbb{Z}[\mathscr{G}^4]\oplus \mathbb{Z}[\mathscr{G}^3]\oplus \mathbb{Z}[\mathscr{G}^3]\oplus \mathbb{Z}[\mathscr{G}^2]\oplus \mathbb{Z}[\mathscr{G}]\xrightarrow{d_3} \mathbb{Z}[\mathscr{G}^3]\oplus \mathbb{Z}[\mathscr{G}^2]\xrightarrow{d_2} \mathbb{Z}[\mathscr{G}^2]\xrightarrow{d_1} \mathbb{Z}[\mathscr{G}]\xrightarrow{d_0} \mathscr{G},\]
where the differentials are given by
\begin{align*}
    d_3([x,y,z,t]) &= ([x+y,z,t] - [x,y+z,t] + [x,y,z+t]- [x,y,z] - [y,z,t],0)\\
    d_3([x,y,z]) &= (-[x,y,z]+[x,z,y]-[z,x,y],[x+y,z]-[x,z]-[y,z])\\
    d_3([x,y,z]) &= ([x,y,z]-[y,x,z]+[y,z,x],[x,y+z]-[x,y]-[x,z])\\
    d_3([x,y]) &= (0,[x,y]+[y,x])\\
    d_3([x]) &= (0,[x,x]) \\
    d_2([x,y,z]) &= [x+y,z]-[x,y+z]+[x,y]-[y,z] \\
	d_2([x,y]) &= [x,y]-[y,x] \\
    d_1([x,y]) &= [x+y]-[x]-[y]\\
	d_0([x]) &= x.
\end{align*}
Here, the top $d_3([x,y,z])$ acts on the first factor of $\mathbb{Z}[\mathscr{G}^3]$, while the bottom $d_3([x,y,z])$ acts on the second factor. Henceforth, we fix a Breen--Deligne resolution that begins with these terms. In particular, this explicit description allows us to define two important invariants.

\begin{definition}
	Let $\mathscr{G}$ and $\mathscr{A}$ be abelian sheaves on $(\textsf{Sch}/S)_\text{fppf}$. Applying the functor $\operatorname{Hom}(-,\mathscr{A})$ to the Breen--Deligne resolution of $\mathscr{G}$, we obtain the complex
\[\begin{tikzcd}
	{\Gamma(\mathscr{G},\mathscr{A})\to  \Gamma(\mathscr{G}^2,\mathscr{A}) \to \Gamma(\mathscr{G}^3,\mathscr{A})\oplus \Gamma(\mathscr{G}^2,\mathscr{A})} \ar[draw=none]{d}[name=X, anchor=center]{}\ar[rounded corners,
            to path={ -- ([xshift=2ex]\tikztostart.east)
                      |- (X.center) \tikztonodes
                      -| ([xshift=-2ex]\tikztotarget.west)
                      -- (\tikztotarget)}]{d}[at end]{} \\
	{\Gamma(\mathscr{G}^4,\mathscr{A})\oplus \Gamma(\mathscr{G}^3,\mathscr{A})\oplus \Gamma(\mathscr{G}^3,\mathscr{A})\oplus \Gamma(\mathscr{G}^2,\mathscr{A})\oplus \Gamma(\mathscr{G},\mathscr{A}).}
\end{tikzcd}\]
We denote the first cohomology of this complex by $\mathrm{H}^2_s(\mathscr{G},\mathscr{A})$ and the second cohomology by $\mathrm{H}^3_s(\mathscr{G},\mathscr{A})$.
\end{definition}

Let $\mathscr{X}$ be a sheaf of sets and let $\mathscr{A}$ be an abelian sheaf on $(\textsf{Sch}/S)_\text{fppf}$. We define the cohomology group $\mathrm{H}^i(\mathscr{X},\mathscr{A})$ as the value at $\mathscr{A}$ of the $i$-th right derived functor of
\[\operatorname{Mor}(\mathscr{X},-)\colon \mathsf{Ab}((\textsf{Sch}/S)_\text{fppf}) \to \mathsf{Ab}.\]
If $\mathscr{X}$ is representable by an $S$-scheme $X$, the Yoneda lemma implies that $\mathrm{H}^i(\mathscr{X},\mathscr{A})$ coincides with the usual cohomology group $\mathrm{H}^i(X,\mathscr{A})$.

As usual, the group $\mathrm{H}^1(\mathscr{X},\mathscr{A})$ classifies $\mathscr{A}$-torsors over $\mathscr{X}$, where the group law is given by the contracted product. Moreover, any morphism $f\colon \mathscr{X} \to \mathscr{Y}$ induces a homomorphism
\[f^*\colon \mathrm{H}^1(\mathscr{Y},\mathscr{A}) \to \mathrm{H}^1(\mathscr{X},\mathscr{A}),\]
which sends an $\mathscr{A}$-torsor $\mathscr{P} \to \mathscr{Y}$ to its pullback $\mathscr{P} \times_{\mathscr{Y}} \mathscr{X} \to \mathscr{X}$. When $\mathscr{X} = \mathscr{G}$ is itself an abelian sheaf, we define $\mathrm{H}_m^1(\mathscr{G},\mathscr{A})$ to be the subgroup of $\mathrm{H}^1(\mathscr{G},\mathscr{A})$ consisting of those $\mathscr{A}$-torsors over $\mathscr{G}$ that are compatible with the group structure on $\mathscr{G}$.

\begin{definition}\label{def Hm}
	Let $\mathscr{G}$ and $\mathscr{A}$ be abelian sheaves on $(\textsf{Sch}/S)_\text{fppf}$. Denote by $m\colon\mathscr{G}\times \mathscr{G}\to \mathscr{G}$ the group operation of $\mathscr{G}$, and by $\operatorname{pr}_1,\operatorname{pr}_2\colon \mathscr{G}\times \mathscr{G}\to \mathscr{G}$ the natural projections. We define $\mathrm{H}_m^1(\mathscr{G},\mathscr{A})$ as the kernel of the morphism $m^*-\operatorname{pr}_1^*-\operatorname{pr}_2^*$.
\end{definition}

Put simply, $\mathrm{H}^1_m(\mathscr{G},\mathscr{A})$ is the group of isomorphism classes of $\mathscr{A}$-torsors $\mathscr{P}$ over $\mathscr{G}$ satisfying $m^*\mathscr{P}\simeq\operatorname{pr}_1^*\mathscr{P}\wedge\operatorname{pr}_2^*\mathscr{P}$. Often, we say that these $\mathscr{A}$-torsors are \emph{multiplicative}. With this terminology established, we may now explain the computation of the extension groups.

\begin{proposition}[Breen]\label{short exact sequence computing extensions}
	Let $\mathscr{G}$ and $\mathscr{A}$ be abelian sheaves on $\normalfont(\textsf{Sch}/S)_\text{fppf}$. There exists an exact sequence
\[0\to \mathrm{H}^2_s(\mathscr{G},\mathscr{A})\to \operatorname{Ext}^1(\mathscr{G},\mathscr{A})\to \mathrm{H}^1_m(\mathscr{G},\mathscr{A})\to \mathrm{H}^3_s(\mathscr{G},\mathscr{A})\to\operatorname{Ext}^2(\mathscr{G},\mathscr{A}),\]
that is functorial in $\mathscr{G}$ and $\mathscr{A}$.
\end{proposition}

Before diving into the proof, let us explain the morphism $\operatorname{Ext}^1(\mathscr{G},\mathscr{A})\to \mathrm{H}^1_m(\mathscr{G},\mathscr{A})$. Since $\operatorname{Ext}^1(-,\mathscr{A})$ is an additive functor, an extension $\mathscr{E}$ of $\mathscr{G}$ by $\mathscr{A}$ always satisfies
\[m^*\mathscr{E}=(\operatorname{pr}_1+\operatorname{pr}_2)^*\mathscr{E}\simeq \operatorname{pr}_1^*\mathscr{E}+\operatorname{pr}_2^*\mathscr{E},\]
where the sum on the right is the Baer sum of extensions. Such an extension defines an $\mathscr{A}$-torsor\footnote{We refer the reader to Remark~\ref{remark on descent} for more explanations.} $\mathscr{P}$ over $\mathscr{G}$, which satisfies $m^*\mathscr{P}\simeq\operatorname{pr}_1^*\mathscr{P}\wedge\operatorname{pr}_2^*\mathscr{P}$.

\begin{proof}[Proof of Proposition~\ref{short exact sequence computing extensions}]
	The universal property of free objects gives that $\mathrm{H}^i(\mathscr{G}^n,\mathscr{A})$ is isomorphic to $\operatorname{Ext}^i(\mathbb{Z}[\mathscr{G}^n],\mathscr{A})$ for all $n$ and $i$. Then, the Breen--Deligne resolution yields a spectral sequence
	\[E_1^{i,j}\colon\prod_{r=1}^{n_i}\mathrm{H}^j(\mathscr{G}^{s_{i,r}},\mathscr{A})\implies \operatorname{Ext}^{i+j}(\mathscr{G},\mathscr{A}),\]
	whose five-term exact sequence is precisely the one in the statement.
\end{proof}

\begin{remark}\label{Grothendieck's remark}
	In \cite[Exp.\ VII, \S 1.2]{raynaud2006groupes}, Grothendieck proved that $\operatorname{Ext}^1(\mathscr{G},\mathscr{A})$ is isomorphic to the group of isomorphism classes of pairs $(\mathscr{P},\alpha)$, where $\mathscr{P}$ is a $\mathscr{A}$-torsor over $\mathscr{G}$ and $\alpha\colon m^*\mathscr{P}\to\operatorname{pr}_1^*\mathscr{P}\wedge\operatorname{pr}_2^*\mathscr{P}$ is an isomorphism of $\mathscr{A}$-torsors over $\mathscr{G}\times\mathscr{G}$ making two diagrams (imposing that $\mathscr{P}$ admits an associative and commutative group law) commute. In particular, our invariants $\mathrm{H}^2_s(\mathscr{G},\mathscr{A})$ and $\mathrm{H}^3_s(\mathscr{G},\mathscr{A})$ govern how far the map
	\begin{align*}
		\operatorname{Ext}^1(\mathscr{G},\mathscr{A}) &\to \mathrm{H}^1_m(\mathscr{G},\mathscr{A})\\
		[\mathscr{P},\alpha] &\mapsto [\mathscr{P}]
	\end{align*}
	is from being an isomorphism.
\end{remark}

Even though the first terms of the Breen--Deligne resolution are explicit, the invariants $\mathrm{H}^2_s(\mathscr{G},\mathscr{A})$ and $\mathrm{H}^3_s(\mathscr{G},\mathscr{A})$ are usually quite hard to compute. The following observation will suffice for their computations in many interesting cases.

\begin{lemma}\label{vanishing lemma}
Let $\mathscr{G}$ and $\mathscr{A}$ be abelian sheaves on $\normalfont(\textsf{Sch}/S)_\text{fppf}$. If every morphism $\mathscr{G}^n\to \mathscr{A}$ of sheaves of \emph{sets} (for $n=2,3$) can be expressed as a sum of maps $\mathscr{G}\to \mathscr{A}$, then both $\mathrm{H}^2_s(\mathscr{G},\mathscr{A})$ and $\mathrm{H}^3_s(\mathscr{G},\mathscr{A})$ vanish.
\end{lemma}

\begin{proof}
	First, we demonstrate that $\mathrm{H}^2_s(\mathscr{G},\mathscr{A})$ vanishes. The kernel of 
\[\Gamma(\mathscr{G}^2,\mathscr{A})\to \Gamma(\mathscr{G}^3,\mathscr{A})\oplus \Gamma(\mathscr{G}^2,\mathscr{A})\]
is composed of the maps $f\colon\mathscr{G}^2\to \mathscr{A}$ that satisfy $f(x+y,z)-f(y,z)=f(x,y+z)-f(x,y)$ and $f(x,y)=f(y,x)$. By applying the given hypothesis, we find morphisms $f_1,f_2\colon \mathscr{G}\to\mathscr{A}$ such that $f(x,y)=f_1(x)+f_2(y)$. This simplifies the first equation to
\[f_1(x+y)-f_1(y)=f_2(y+z)-f_2(y).\]
By setting $x=y=0$ and $y=z=0$, we observe that $f$ must be constant. Next, the image of
\[\Gamma(\mathscr{G},\mathscr{A})\to \Gamma(\mathscr{G}^2,\mathscr{A})\]
consists of maps of the form $(x,y)\mapsto g(x+y)-g(x)-g(y)$, for some $g\colon\mathscr{G}\to\mathscr{A}$. Since every constant map is of this form, it follows that the cohomology vanishes.

We will use the same strategy to show that $\mathrm{H}^3_s(\mathscr{G},\mathscr{A})$ also vanishes. The kernel of the morphism
\[\Gamma(\mathscr{G}^3,\mathscr{A})\oplus \Gamma(\mathscr{G}^2,\mathscr{A})\to \Gamma(\mathscr{G}^4,\mathscr{A})\oplus \Gamma(\mathscr{G}^3,\mathscr{A})\oplus \Gamma(\mathscr{G}^3,\mathscr{A})\oplus \Gamma(\mathscr{G}^2,\mathscr{A})\oplus \Gamma(\mathscr{G},\mathscr{A})\]
is composed of the maps $p\colon \mathscr{G}^3\to \mathscr{A}$ and $q\colon \mathscr{G}^2\to \mathscr{A}$ satisfying
\begin{align}
	p(x,y,z)+p(x,y+z,t)+p(y,z,t) &= p(x+y,z,t)+p(x,y,z+t) \tag{1.1}\\
	p(x,y,z)+p(z,x,y)+q(x,z)+q(y,z) &= p(x,z,y)+q(x+y,z) \tag{2.1}\\
	p(x,y,z)+p(y,z,x)+q(x,y+z) &= p(y,x,z)+q(x,y)+q(x,z) \tag{3.1}\\
	q(x,y)+q(y,x) &= 0 \tag{4.1}\\
	q(x,x) &= 0. \tag{5.1}
\end{align}
To organize the remainder of the proof, we will denote a simplified version of the Equation ($n$.$i$) as ($n$.$i+1$). Once again, we write $p(x,y,z)=p_1(x)+p_2(y)+p_3(z)$ for some $p_i\colon \mathscr{G}\to\mathscr{A}$, and $q(x,y)=q_1(x)+q_2(y)$ for some $q_i\colon \mathscr{G}\to\mathscr{A}$. Using Equation (5.1), we replace every instance of $q_2$ by $-q_1$. This yields the following relations.
\begin{align}
	p_3(z)+p_1(x)+p_2(y+z)+p_1(y)+p_3(t) &= p_1(x+y)+p_3(z+t) \tag{1.2}\\
	p_2(y)+p_3(z)+p_1(z)+p_2(x)+q_1(x)+q_1(y) &= p_2(z)+q_1(x+y)+q_1(z) \tag{2.2}\\
	p_1(x)+p_2(y)+p_2(z)+p_3(x)+q_1(y)+q_1(z) &= p_2(x)+q_1(y+z)+q_1(x) \tag{3.2}
\end{align}
By setting $x=y=0$ in Equations (2.2) and (3.2), we obtain
\begin{align*}
	p_2(0)+p_3(z)+p_1(z)+q_1(0) &= q_1(z) \tag{2.3}\\
	p_1(0)+p_2(z)+p_3(0) &= 0 \tag{3.3}.
\end{align*}
With this information, we can simplify Equation (1.2):
\begin{align}
	p_3(z)+p_1(x)+p_2(0)+p_1(y)+p_3(t) &= p_1(x+y)+p_3(z+t). \tag{1.3}
\end{align}
Restricting this equation to $x = 0$ and $t = 0$ gives two new relations, implying that $p_1$ and $p_3$ are essentially morphisms of groups up to constants.
\begin{align}
    p_1(x)+p_1(y)-p_1(0) &= p_1(x+y) \tag{1.4}\\
	p_3(z)+p_3(t)-p_3(0) &= p_3(z+t) \tag{1.5}
\end{align}
In summary, we have obtained the following equations:
\begin{align}
    p_1(x)+p_1(y)-p_1(0) &= p_1(x+y) \tag{1.4}\\
	p_3(z)+p_3(t)-p_3(0) &= p_3(z+t) \tag{1.5}\\
	p_1(z)-p_1(0)+p_3(z)-p_3(0) &= q_1(z)-q_1(0) \tag{2.3}\\
	p_1(0)+p_2(z)+p_3(0) &= 0. \tag{3.3}
\end{align}

Finally, to prove that $\mathrm{H}^3_s(\mathscr{G},\mathscr{A})$ vanishes, we need to find a morphism $h\colon \mathscr{G}^2\to \mathscr{A}$ such that
\begin{align*}
	p(x,y,z) &= h(x+y,z)-h(x,y+z)+h(x,y)-h(y,z) \\
	q(x,y) &= h(x,y) - h(y,x).
\end{align*}
A quick verification using the previously established equations shows that the morphism $h(x, y) \colonequals p_1(x) - p_3(y)$ satisfies this criterion.
\end{proof}

The following proposition extends a result by Colliot-Thélène and Gabber \cite[Prop.\ 3.2 and Thm.\ 5.6]{colliotthelene2006resolutions}. Their method of proof can also be applied to our setting, but this result follows directly from our machinery.

\begin{theorem}\label{injectivity}
	Let $G$ be a commutative group scheme over a reduced base $S$. Assume that the morphism $G\to S$ is smooth and has connected geometric fibers. Then the natural map
	\[\operatorname{Ext}^1_S(G,\mathbb{G}_m) \to \mathrm{H}^1_m(G,\mathbb{G}_{m})\]
	is an isomorphism.
\end{theorem}

\begin{proof}
This result is a direct application of Lemma~\ref{vanishing lemma}, using González-Avilés' generalization of Rosenlicht's lemma \cite[Thm.\ 1.1]{gonzalez2017group}. For the reader's convenience, we note that this result holds more generally in the following context: $G$ is a commutative group scheme over a reduced base $S$ such that $G \to S$ is syntomic and has reduced, connected maximal geometric fibers. Here, a maximal geometric fiber of $G \to S$ is a scheme $G \times_S \operatorname{Spec}\overline{\kappa(\eta)}$, for some generic point $\eta$ of an irreducible component of $S$.
\end{proof}

\begin{remark}
The hypothesis that the base $S$ be reduced is essential in the theorem above. Indeed, let $R$ be a reduced ring and set $R' \colonequals R[\varepsilon]/(\varepsilon^2)$. As explained in \cite[Rem.~7.3]{ribeiro2025}, there is a natural isomorphism
\[\mathrm{H}^2_s(\mathbb{G}_{a,R'},\mathbb{G}_{m,R'}) \simeq \mathrm{H}^2_s(\mathbb{G}_{a,R},\mathbb{G}_{a,R}).\]
The right-hand side was explicitly computed in \cite[Prop.~7.1]{ribeiro2025} and is shown to be a free $R$-module whose rank equals the number of prime numbers that are not invertible in $R$. It follows that, unless $R$ is a $\mathbb{Q}$-algebra, the natural map
\[\operatorname{Ext}^1_{R'}(\mathbb{G}_a,\mathbb{G}_m) \to \mathrm{H}^1_m(\mathbb{G}_a,\mathbb{G}_m)\]
is not injective.
\end{remark}

Henceforth, we concentrate on the case where the base scheme is the spectrum of a characteristic zero field. We start by using Proposition~\ref{sheafification map} to prove that the sheafification map is an isomorphism in certain cases.

\begin{proposition}\label{local-to-global}
	Let $T$ be a torus and $U$ be a unipotent group over a characteristic zero field $k$. The sheafification map $\operatorname{Ext}^1_S(T,\mathbb{G}_m)\to \underline{\operatorname{Ext}}^1(T,\mathbb{G}_m)(S)$ is an isomorphism for an irreducible geometrically unibranch $k$-scheme $S$. Similarly, the map $\operatorname{Ext}^1_S(U,\mathbb{G}_m)\to \underline{\operatorname{Ext}}^1(U,\mathbb{G}_m)(S)$ is an isomorphism for a $k$-scheme $S$ that is either reduced or affine.
\end{proposition}

\begin{proof}
Let $S$ be an irreducible geometrically unibranch $k$-scheme. Proposition~\ref{T'} states that the sheaf $\underline{\operatorname{Ext}}^1(T,\mathbb{G}_m)$ vanishes, so the desired result is equivalent to the vanishing of the cohomology group $\mathrm{H}^1(S,X)=0$, where $X\colonequals T^D$ is the Cartier dual of $T$. This holds due to \cite[Exp.\ VIII, Prop.\ 5.1]{raynaud2006groupes}.

For the unipotent case, we may assume that $U = \mathbb{G}_a$. If $S$ is an affine scheme, \cite[Rem.\ 2.2.18]{bhatt2022prismatic} implies that $\mathrm{H}^i(S,\widehat{\mathbb{G}}_a)$ vanishes for all $i\neq 0$. We claim that, for a reduced $k$-scheme $S$, the cohomology groups $\mathrm{H}^i(S,\widehat{\mathbb{G}}_a)$ vanish for all $i$. First, we prove that the \emph{étale} cohomology groups $\mathrm{H}^i_\text{ét}(S, \widehat{\mathbb{G}}_a)$ vanish, and then we identify these groups with their fppf analogues.

Recall that $\widehat{\mathbb{G}}_a(\operatorname{Spec}R)$ is the nilradical of a $k$-algebra $R$. Thus, the sheaf condition implies that $\widehat{\mathbb{G}}_a(S)$ vanishes for reduced $S$. Next, the cohomology $\mathrm{H}_\text{ét}^i(S, \widehat{\mathbb{G}}_a)$ can be computed on the small étale site of $S$ and, according to \cite[Tag \href{https://stacks.math.columbia.edu/tag/03PC}{03PC}.(8)]{Stacks}, the restriction of $\widehat{\mathbb{G}}_a$ to this site vanishes.

Since $\mathbb{G}_{a,\text{dR}}$ is the étale quotient of $\mathbb{G}_a$ by $\widehat{\mathbb{G}}_a$, we obtain the following morphism of long exact sequences:
\[\begin{tikzcd}[column sep=small]
	0 & {\Gamma(S,\widehat{\mathbb{G}}_a)} & {\Gamma(S,\mathbb{G}_a)} & {\Gamma(S,\mathbb{G}_{a,\text{dR}})} & {\mathrm{H}^1(S,\widehat{\mathbb{G}}_a)} & {\mathrm{H}^1(S,\mathbb{G}_a)} & \cdots \\
	0 & {\Gamma(S,\widehat{\mathbb{G}}_a)} & {\Gamma(S,\mathbb{G}_a)} & {\Gamma(S,\mathbb{G}_{a,\text{dR}})} & {\mathrm{H}^1_\text{ét}(S,\widehat{\mathbb{G}}_a)} & {\mathrm{H}^1_\text{ét}(S,\mathbb{G}_a)} & \cdots.
	\arrow[from=2-1, to=2-2]
	\arrow[from=2-2, to=2-3]
	\arrow[from=2-3, to=2-4]
	\arrow[from=2-4, to=2-5]
	\arrow[from=1-1, to=1-2]
	\arrow[from=1-2, to=1-3]
	\arrow[from=1-3, to=1-4]
	\arrow[equals, from=1-2, to=2-2]
	\arrow[equals, from=1-3, to=2-3]
	\arrow[equals, from=1-4, to=2-4]
	\arrow[from=1-4, to=1-5]
	\arrow[from=1-5, to=2-5]
	\arrow[from=1-5, to=1-6]
	\arrow[from=2-5, to=2-6]
	\arrow[from=1-6, to=2-6]
	\arrow[from=2-6, to=2-7]
	\arrow[from=1-6, to=1-7]
\end{tikzcd}\]
Given that $S$ is reduced, we have shown that $\mathrm{H}_\text{ét}^i(S,\widehat{\mathbb{G}}_a)=0$ for all $i$. Moreover, since $\mathbb{G}_a$ is smooth, the natural map $\mathrm{H}^i(S,\mathbb{G}_a)\to \mathrm{H}_\text{ét}^i(S,\mathbb{G}_a)$ is an isomorphism. Using these facts, a diagram chase gives that $\mathrm{H}^i(S,\widehat{\mathbb{G}}_a)$ vanishes for all $i$.
\end{proof}

In order to analyze the sheafification map $\operatorname{Ext}^1_S(G, \mathbb{G}_m) \to \underline{\operatorname{Ext}}^1(G, \mathbb{G}_m)(S)$ for a connected commutative algebraic group $G$ over a field of characteristic zero, we require the following vanishing result. We refer the reader to \cite[\S 6]{ribeiro2025} for further discussion of the elusive sheaf $\underline{\operatorname{Ext}}^1(\mathbb{G}_a,\mathbb{G}_m)$; in particular, we note that it does not vanish.

\begin{proposition}\label{unipotent extension}
	Let $U$ be a unipotent group over a characteristic zero field $k$. Then the group $\underline{\operatorname{Ext}}^1(U,\mathbb{G}_m)(S)$ vanishes for all seminormal $k$-schemes $S$.
\end{proposition}

\begin{proof}
	Since seminormal schemes are reduced \cite[Tag \href{https://stacks.math.columbia.edu/tag/0EUQ}{0EUQ}]{Stacks}, Theorem~\ref{injectivity} and Proposition~\ref{local-to-global} imply that $\underline{\operatorname{Ext}}^1(U,\mathbb{G}_m)(S)\simeq \mathrm{H}^1_m(U_S,\mathbb{G}_{m,S})$. Denoting by $p\colon U_S\to S$ the structure map, Traverso's theorem \cite[Lem.\ 4.3]{sadhu2021equivariant} gives that
\[p^*\colon \mathrm{H}^1(S,\mathbb{G}_{m})\to \mathrm{H}^1(U_S,\mathbb{G}_{m})\]
is an isomorphism. In particular, $\mathrm{H}^1_m(U_S,\mathbb{G}_{m})$ is isomorphic to the subgroup of $\mathrm{H}^1(S,\mathbb{G}_{m})$ consisting of elements $x\in \mathrm{H}^1(S,\mathbb{G}_{m})$ that satisfy $p^*x\in \mathrm{H}^1_m(U_S,\mathbb{G}_{m})$. Next, denote by $m\colon U_S\times_S U_S\to U_S$ the group operation and by $\operatorname{pr}_i\colon U_S\times_S U_S\to U_S$ the natural projections. Then, $p^*x$ lies in $\mathrm{H}^1_m(U_S,\mathbb{G}_{m})$ if and only if
\[m^*p^*x=\operatorname{pr}_1^*p^*x+\operatorname{pr}_2^*p^*x.\]
However, the morphisms $p\circ m$, $p\circ\operatorname{pr}_1$, and $p\circ\operatorname{pr}_2$ are all equal to the structure map of $U^2_S$, which has a section $S\to U^2_S$. Thus, $m^*p^*x=\operatorname{pr}_1^*p^*x+\operatorname{pr}_2^*p^*x$ holds if and only if $x=0$, completing the proof.
\end{proof}

We are now in position to prove an enhancement of Theorem~\ref{injectivity} for commutative connected algebraic groups over a characteristic zero field.

\begin{theorem}\label{local to global for G}
	Let $G$ be a commutative connected algebraic group over a characteristic zero field $k$. For a regular $k$-scheme $S$, the natural maps
	\[\underline{\operatorname{Ext}}^1(G,\mathbb{G}_m)(S)\leftarrow\operatorname{Ext}^1_S(G,\mathbb{G}_m)\to \mathrm{H}^1_m(G_S,\mathbb{G}_m)\]
    are isomorphisms.
\end{theorem}

\begin{proof}
	That the arrow on the right is an isomorphism was already proven in Theorem~\ref{injectivity}. This result holds even if $S$ is merely assumed to be reduced. For the reader's convenience, we note that the hypotheses of Theorem~\ref{injectivity} are satisfied due to \cite[Cor.\ 1.32 and 8.39]{M}.
	
	We now turn our attention to the sheafification map on the left. As in the proof of Proposition~\ref{cartier dual of dR}, we know that $G$ is an extension of an abelian variety $A$ by a linear group $L$, that is a product of a torus $T$ and a unipotent group $U$. This leads to the commutative diagram
\[\begin{tikzcd}[column sep=small]
	{\operatorname{Hom}_S(L,\mathbb{G}_m)} & {\operatorname{Ext}^1_S(A,\mathbb{G}_m)} & {\operatorname{Ext}^1_S(G,\mathbb{G}_m)} & {\operatorname{Ext}^1_S(L,\mathbb{G}_m)} \\
	{\underline{\operatorname{Hom}}(L,\mathbb{G}_m)(S)} & {\underline{\operatorname{Ext}}^1(A,\mathbb{G}_m)(S)} & {\underline{\operatorname{Ext}}^1(G,\mathbb{G}_m)(S)} & {\underline{\operatorname{Ext}}^1(L,\mathbb{G}_m)(S),}
	\arrow[from=1-1, to=1-2]
	\arrow[equals, from=1-1, to=2-1]
	\arrow[from=2-1, to=2-2]
	\arrow[from=2-2, to=2-3]
	\arrow[from=1-2, to=1-3]
	\arrow[from=1-3, to=1-4]
	\arrow[from=2-3, to=2-4]
	\arrow[r, "\sim" labl, from=1-2, to=2-2]
	\arrow[r, "\sim" labl, from=1-4, to=2-4]
	\arrow[from=1-3, to=2-3]
\end{tikzcd}\]
whose rows are exact. Propositions~\ref{T'} and \ref{unipotent extension} imply that $\underline{\operatorname{Ext}}^1(L,\mathbb{G}_m)(S)$ vanishes, and Proposition~\ref{local-to-global} further ensures that $\operatorname{Ext}^1_S(L,\mathbb{G}_m)=0$. The desired result then follows from a diagram chase.
\end{proof}

The following example demonstrates that the map $\operatorname{Ext}^1_S(G, \mathbb{G}_m) \to \underline{\operatorname{Ext}}^1(G, \mathbb{G}_m)(S)$ may not be an isomorphism if $S$ is not regular.

\begin{example}\label{nodal example}
	Let $k$ be a field, and consider the nodal curve $S=\operatorname{Spec}k[x,y]/(y^2-xy-x^3)$. The proof of Proposition~\ref{short exact sequence computing extensions} indicates that $\operatorname{Ext}^1_S(\mathbb{G}_m,\mathbb{G}_m)$ can be computed in the étale topology. Since Grothendieck's proof of Proposition~\ref{T'} also shows that $\underline{\operatorname{Ext}}^1(\mathbb{G}_m,\mathbb{G}_m)$ vanishes on the étale site, Proposition~\ref{sheafification map} gives that
	\[\operatorname{Ext}^1_S(\mathbb{G}_m,\mathbb{G}_m)\simeq\mathrm{H}^1_\text{ét}(S,\mathbb{Z}).\]
	
	As observed in \cite[Rem.~5.5.2]{Weibel}, the étale cohomology group $\mathrm{H}^1_\text{ét}(S,\mathbb{Z})$ is isomorphic to $\mathbb{Z}$; implying that the sheafification map $\operatorname{Ext}^1_S(\mathbb{G}_m, \mathbb{G}_m) \to \underline{\operatorname{Ext}}^1(\mathbb{G}_m, \mathbb{G}_m)(S)$ is not an isomorphism.
\end{example}

We also prove an analogue of Theorem~\ref{injectivity} for de Rham spaces.

\begin{theorem}\label{global de rham}
	Let $G$ be a commutative connected algebraic group over a characteristic zero field $k$. For a reduced $k$-scheme $S$, the natural maps
	\[\underline{\operatorname{Ext}}^1(G_{\normalfont\text{dR}},\mathbb{G}_m)(S)\leftarrow\operatorname{Ext}^1_S(G_{\normalfont\text{dR}},\mathbb{G}_m)\to \mathrm{H}^1_m(G_{\normalfont\text{dR}}\times S,\mathbb{G}_m)\]
    are isomorphisms.
\end{theorem}

\begin{proof}
Proposition~\ref{cartier dual of dR} asserts that the Cartier dual of $G_\text{dR}$ vanishes. Consequently, Proposition~\ref{sheafification map} implies that the map on the left is an isomorphism for all $k$-schemes $S$. To show that the map on the right is also an isomorphism for reduced $k$-schemes $S$, we use Lemma~\ref{vanishing lemma}.

Consider a morphism of $S$-schemes $f\colon G^n_\text{dR}\times S\to \mathbb{G}_{m,S}$, where $n$ is either $2$ or $3$. As in the proof of Theorem~\ref{injectivity}, there exist morphisms $f_i\colon G_S\to \mathbb{G}_{m,S}$ for $i=1,\dots,n$, such that their product equals the composition
\[G_S^n\to G^n_\text{dR}\times S\to \mathbb{G}_{m,S}.\]
Since epimorphisms in topoi are stable under base change, Proposition~\ref{GdR is reasonable} implies that the map $G_S^n\to G^n_\text{dR}\times S$ is an epimorphism. Therefore, the morphisms $f_i$ factor through the quotient, yielding maps
\[\overline{f}_i\colon G_\text{dR}\times S\to \mathbb{G}_{m,S},\]
whose product is equal to $f$.
\end{proof}

We will require two further vanishing results. The first concerns extension sheaves of commutative formal groups and is proved in \cite[Lem.\ 1.14]{russell2013albanese}. This result will be applied primarily to the formal completion $\widehat{G}$ of a commutative algebraic group $G$ over a field of characteristic zero, a context in which it first appeared in \cite[Lem.\ A.4.6]{barbieri2009sharp}.

\begin{proposition}\label{ext G^}
Let $\mathscr{G}$ be a commutative formal group over a field whose Cartier dual is of finite type. Then $\underline{\operatorname{Ext}}^1(\mathscr{G},\mathbb{G}_m)$ vanishes.
\end{proposition}

For the reader's convenience, we provide a proof of this proposition, based on the methods developed in this section, under the assumption that the base field has characteristic zero.

\begin{proof}
    Let $k$ be the base field of characteristic $0$, and set $L \colonequals \mathscr{G}^D$. By \cite[Prop.~2.37 and Cor.~16.15]{M}, there exist a finite étale group scheme $F$, a torus $T$, and a vector group $U$, fitting into a short exact sequence
    \[0 \to T \times U \to L \to F \to 0.\]
    Dualizing this extension (see \cite[Exp.\ VIII, Prop.\ 3.3.1]{raynaud2006groupes} and \cite[Tag~\href{https://stacks.math.columbia.edu/tag/02KB}{02KB}]{Stacks}) yields a short exact sequence
    \[0 \to F^D \to \mathscr{G} \to T^D \times U^D \to 0.\]
    Taking the long exact sequence in cohomology associated to $(-)^D=\underline{\operatorname{Hom}}(-,\mathbb{G}_m)$, we reduce to showing that
    \[\underline{\operatorname{Ext}}^1(F^D,\mathbb{G}_m)=\underline{\operatorname{Ext}}^1(T^D,\mathbb{G}_m)=\underline{\operatorname{Ext}}^1(U^D,\mathbb{G}_m)=0.\]
    The first vanishing follows from another application of \cite[Exp.\ VIII, Prop.\ 3.3.1]{raynaud2006groupes}, whereas the second was dealt with in Proposition~\ref{T'}. Hence it remains to prove the vanishing for $U^D$.

    Since $U$ is a vector group, after choosing coordinates in $U$ we have that $U^D$ is a power of $\widehat{\mathbb{G}}_a$, so it is enough to show that
    \[\operatorname{Ext}^1_R(\widehat{\mathbb{G}}_{a},\mathbb{G}_{m})=0\]
    for every $k$-algebra $R$. By Proposition~\ref{short exact sequence computing extensions}, this in turn reduces to proving that $\mathrm{H}^1_m(\widehat{\mathbb{G}}_{a,R},\mathbb{G}_m)$ and $\mathrm{H}^2_s(\widehat{\mathbb{G}}_{a,R},\mathbb{G}_m)$ are trivial. We now establish these two vanishings.

    For $n\geq 1$, let $\mathbb{G}_{a,R}^{(n)}$ denote the abelian sheaf represented by $\operatorname{Spec}R[t]/(t^n)$. By definition, the formal additive group $\widehat{\mathbb{G}}_{a,R}$ is the filtered colimit of the $\mathbb{G}_{a,R}^{(n)}$ in the category of fppf sheaves, and the same colimit computes $\widehat{\mathbb{G}}_{a,R}$ in the $(2,1)$-category of fppf stacks. Consequently, we have a natural equivalence of groupoids
    \[\operatorname{Map}(\widehat{\mathbb{G}}_{a,R},\mathsf{B}\mathbb{G}_m)\simeq \operatorname{Map}(\operatorname{colim}_{n} \mathbb{G}_{a,R}^{(n)},\mathsf{B}\mathbb{G}_m)\simeq \operatorname{lim}_{n}\operatorname{Map}(\mathbb{G}_{a,R}^{(n)},\mathsf{B}\mathbb{G}_m).\]

    A $\mathbb{G}_m$-torsor on $\mathbb{G}_{a,R}^{(n)}$ is the same as a line bundle on $\operatorname{Spec}R[t]/(t^n)$; i.e.\ an element of $\operatorname{Pic}(R[t]/(t^n))$. Since the Picard group is invariant under quotienting by nilpotent ideals \cite[Tag~\href{https://stacks.math.columbia.edu/tag/0C6R}{0C6R}]{Stacks}, we have a natural isomorphism $\operatorname{Pic}(R[t]/(t^n))\simeq \operatorname{Pic}(R)$. Thus every $\mathbb{G}_m$-torsor on $\widehat{\mathbb{G}}_{a,R}$ is pulled back from $\operatorname{Spec}R$. Finally, as in the proof of Proposition~\ref{unipotent extension}, such a torsor is multiplicative precisely if it is trivial. In other words, the group $\mathrm{H}^1_m(\widehat{\mathbb{G}}_{a,R},\mathbb{G}_m)$ vanishes.

    Unwinding the definitions, we see that the group $\mathrm{H}^2_s(\widehat{\mathbb{G}}_{a,R},\mathbb{G}_m)$ is the middle cohomology of the complex
    \[R\llbracket x\rrbracket^\times \xrightarrow{\ \alpha\ } R\llbracket x,y\rrbracket^\times\xrightarrow{\ \beta\ } R\llbracket x,y,z\rrbracket^\times \oplus R\llbracket x,y\rrbracket^\times,\]
    whose differentials are given by
    \begin{align*}
        \alpha(p(x)) &= \frac{p(x+y)}{p(x)p(y)} \\
        \beta(q(x,y)) &= \mleft(\frac{q(x+y,z)q(x,y)}{q(x,y+z)q(y,z)},\frac{q(x,y)}{q(y,z)}\mright).
    \end{align*}
    Let $q(x,y)\in R\llbracket x,y\rrbracket^\times$ be a formal power series in the kernel of $\beta$. Without loss of generality, we may assume that $q(0,0)=1$. Since $R$ is a $\mathbb{Q}$-algebra, the usual power series for the exponential and logarithm define mutually inverse group isomorphisms
    \[\exp\colon (x_1,\dots,x_r)R\llbracket x_1,\dots,x_r\rrbracket \rightleftarrows 1+ (x_1,\dots,x_r)R\llbracket x_1,\dots,x_r\rrbracket \mathrel{:}\!\log.\]
    Thus, $g(x,y)\colonequals \log(q(x,y))\in (x,y)R\llbracket x,y\rrbracket$ satisfies the additive relations
    \begin{equation}\tag{$\ast$}\label{power series relations}
        g(x+y,z) + g(x,y)=g(x,y+z)+g(y,z)\quad \text{and}\quad g(x,y)=g(y,x).
    \end{equation}

    First, setting $y=z=0$ in (\ref{power series relations}) yields $g(x,0)=g(0,x)=0$. Next, differentiate the first identity in (\ref{power series relations}) with respect to $z$. Writing $g_i$ for the partial derivative of $g$ with respect to its $i$-th variable, we obtain
    \[g_2(x+y,z)=g_2(x,y+z)+g_2(y,z).\]
    Putting $z=0$ and using the second relation in (\ref{power series relations}), we get the relations
    \[g_1(x,y)=\varphi(x+y)-\varphi(x) \quad \text{and}\quad g_2(x,y)=\varphi(x+y)-\varphi(y),\]
    where $\varphi(x)\colonequals g_1(0,x) = g_2(x,0)$.

    Formally integrating $\varphi$, we obtain an element $f(x)\in xR\llbracket x\rrbracket$ satisfying $f^\prime(x)=\varphi(x)$. We claim that
    \[g(x,y)=f(x+y)-f(x)-f(y).\]
    Indeed, the difference $g(x,y)-f(x+y)+f(x)+f(y)$ vanishes at $(0,0)$, and its partial derivatives with respect to both variables are zero. Consequently, the formal power series $p(x)\colonequals \exp(f(x))\in 1+ xR\llbracket x\rrbracket$ is such that $\alpha(p(x))=q(x,y)$, finishing the proof.
\end{proof}

\begin{proposition}\label{vanishing ext2}
    Let $A$ be an abelian variety over a characteristic zero field. Then the abelian sheaves $\underline{\operatorname{Ext}}^2(A,\mathbb{G}_m)$ and $\normalfont\underline{\operatorname{Ext}}^2(A_\text{dR},\mathbb{G}_m)$ vanish.
\end{proposition}

\begin{proof}
    The vanishing of $\underline{\operatorname{Ext}}^2(A,\mathbb{G}_m)$ for abelian schemes over an arbitrary base was established in \cite[Thm.~B]{ribeiro2025}. The abelian sheaf $\normalfont\underline{\operatorname{Ext}}^2(A_\text{dR},\mathbb{G}_m)$ sits in the exact sequence
    \[\underline{\operatorname{Ext}}^1(\widehat{A},\mathbb{G}_m)\to \underline{\operatorname{Ext}}^2(A_\text{dR},\mathbb{G}_m)\to \underline{\operatorname{Ext}}^2(A,\mathbb{G}_m),\]
    and the left-hand term vanishes by the preceding proposition.
\end{proof}

In order to have a bird's-eye view of this section, consider the following definition.

\begin{definition}\label{def of G natural}
	Let $G$ be a commutative connected algebraic group over a characteristic zero field $k$. We denote the abelian sheaf $\underline{\operatorname{Ext}}^1(G_\text{dR},\mathbb{G}_m)$ by $G^\natural$, and the abelian sheaf $\underline{\operatorname{Ext}}^1(G,\mathbb{G}_m)$ by $G^\prime$.
	\end{definition}

Denote by $m\colon G\times G\to G$ the group law of $G$. For a reduced $k$-scheme $S$, Theorem~\ref{global de rham} gives isomorphisms
\[G^\natural (S)\colonequals \underline{\operatorname{Ext}}^1(G_\text{dR},\mathbb{G}_m)(S) \xleftarrow{\ \sim\ } \operatorname{Ext}^1_S(G_\text{dR},\mathbb{G}_m) \xrightarrow{\ \sim\ } \mathrm{H}^1_m(G_\text{dR}\times S,\mathbb{G}_{m})\]
that are functorial on $G$ and $S$. Corollary~\ref{fundamentalresultonderhamspaces} then implies that $G^\natural(S)$ is isomorphic to the set of isomorphism classes of line bundles $\mathscr{L}$ on $G_S$ with integrable connection $\nabla$ relative to $S$ satisfying $m^*(\mathscr{L},\nabla)\simeq (\mathscr{L},\nabla)\boxtimes (\mathscr{L},\nabla)$. Moreover, the group structure of $G^\natural(S)$ corresponds to the tensor products of connections.

Similarly, for a regular $k$-scheme $S$, Theorem~\ref{injectivity} yields isomorphisms
\[G^\prime (S)\colonequals \underline{\operatorname{Ext}}^1(G,\mathbb{G}_m)(S) \xleftarrow{\ \sim\ } \operatorname{Ext}^1_S(G,\mathbb{G}_m) \xrightarrow{\ \sim\ } \mathrm{H}^1_m(G_S,\mathbb{G}_{m})\]
that are functorial on $G$ and $S$. As above, this implies that $G^\prime(S)$ can be identified to the group of isomorphism classes of line bundles $\mathscr{L}$ on $G_S$ satisfying $m^*\mathscr{L}\simeq \mathscr{L}\boxtimes \mathscr{L}$.

The following remark explains our choice of notation in Definition~\ref{def of G natural}.

\begin{remark}\label{remark 3}
Let $A$ be an abelian variety over a characteristic zero field $k$. According to Proposition~\ref{T'}, the abelian sheaf $A^\prime$ in Definition~\ref{def of G natural} is represented by the dual abelian variety. Thus, our notation is not overloaded. 

By considering the long exact sequence in cohomology associated with the Cartier duality functor $(-)^D\colonequals \underline{\operatorname{Hom}}(-,\mathbb{G}_m)$ and the short exact sequence
\[0\to \widehat{A}\to A\to A_\text{dR}\to 0,\]
we deduce that $A^\natural$ is an extension of $A^\prime$ by $\Omega_A$, the vector group of the invariant differentials of $A$. (Consequently, $A^\natural$ is representable by an algebraic group.\footnote{We refer the reader to Remark~\ref{remark on descent} for more explanations.}) We affirm that $A^\natural$ is the universal vector extension of $A^\prime$ as defined in \cite[\S I.1]{mazur2006universal}.
	 
The proof of Theorem~\ref{global de rham} provides natural isomorphisms
\[A^\natural (S)\colonequals \underline{\operatorname{Ext}}^1(A_\text{dR},\mathbb{G}_m)(S) \xleftarrow{\ \sim\ } \operatorname{Ext}^1_S(A_\text{dR},\mathbb{G}_m) \xrightarrow{\ \sim\ } \mathrm{H}^1_m(A_\text{dR}\times S,\mathbb{G}_{m})\]
for \emph{all} $k$-schemes $S$. This implies that the presheaf $S\mapsto \mathrm{H}^1_m(A_\text{dR}\times S,\mathbb{G}_{m})$ already satisfies fppf descent. Therefore, the sheafification in the definition of $E^\natural$ \cite[Def.\ I.4.1.6]{mazur2006universal} is superfluous, leading to the identification $A^\natural\simeq E^\natural$.

Mazur and Messing also define an abelian sheaf $\underline{\operatorname{Ext}}^\natural(A,\mathbb{G}_m)$ which, by Remark~\ref{Grothendieck's remark}, is isomorphic to $A^\natural= \underline{\operatorname{Ext}}^1(A_\text{dR},\mathbb{G}_m)$. In particular, our methods reprove their \cite[Prop.\ I.4.2.1]{mazur2006universal}, which compares $\underline{\operatorname{Ext}}^\natural(A,\mathbb{G}_m)$ and $E^\natural$. Finally, from \cite[Props.\ I.2.6.7 and I.3.2.3]{mazur2006universal}, we conclude that $A^\natural$ is indeed the universal vector extension of $A^\prime$.
\end{remark}

\section{Moduli of character sheaves}
\subsection{Generalized 1-motives and their Laumon duals}\label{section laumon dual}

Let $k$ be a perfect field of any characteristic. In this section, we will systematically apply Proposition~\ref{DK} to represent commutative group stacks over $k$ as two-term complexes (in degrees $-1$ and $0$) of abelian sheaves on the site $(\textsf{Sch}/k)_\text{fppf}$. For the reader's convenience, we recall that the Cartier dual of a commutative formal group is represented by an affine commutative group scheme.

\begin{definition}\label{def generalized 1-motive}
A \emph{generalized 1-motive} is a two-term complex of abelian sheaves $[\mathscr{G}\to G]$ in $\normalfont(\textsf{Sch}/k)_\text{fppf}$, where $G$ is a smooth commutative connected algebraic group, and $\mathscr{G}$ is a commutative formal group whose Cartier dual is smooth and connected.
\end{definition}

A usual 1-motive, as defined in \cite[\S 10.1]{deligne1974theorie}, is a special case of the definition above where $k$ is algebraically closed, $G$ is a semiabelian variety, and $\mathscr{G}$ is a finitely generated free $\mathbb{Z}$-module. Our Definition~\ref{def generalized 1-motive} is inspired by the one in \cite{russell2013albanese} and extends Laumon's \cite{Lau} to base fields that may have positive characteristic.

Let $[\mathscr{G}\to G]$ be a generalized 1-motive. The Barsotti--Chevalley theorem \cite[Thm.\ 8.27]{M} states that $G$ has a smallest subgroup $L$ such that $G/L$ is proper. This subgroup is affine, smooth and connected, and the quotient $G/L$ is an abelian variety $A$. In other words, $G$ can be functorially decomposed as an extension
\[0\to L\to G\to A\to 0,\]
of an abelian variety $A$ by an affine smooth commutative connected algebraic group $L$.

\begin{remark}
    Every algebraic group over a characteristic zero field is smooth. In positive characteristic, a singular connected algebraic group is still an extension of an abelian variety by a linear group, but this decomposition may not be unique. See \cite[Ex.\ 4.3.8]{brion2017some}.
\end{remark}

The motivation for Definition~\ref{def generalized 1-motive} arises from the fact that the Cartier duality of Proposition~\ref{affine cartier duality} naturally restricts to an anti-equivalence of categories
\begin{align*}
    \left\{\begin{array}{c}       \text{Affine smooth}\\ \text{commutative connected}\\       \text{algebraic groups over }k     \end{array}\right\} &\longleftrightarrow\left\{\begin{array}{c}       \text{Commutative formal groups}\\       \text{over }k \text{ whose Cartier dual is}\\ \text{smooth and connected}    \end{array}\right\} \\
    G=\operatorname{Spec}R &\longmapsto G^D\simeq \operatorname{Spf}R^*\\
    \mathscr{G}^D\simeq \operatorname{Spec}R^* &\longmapsfrom \mathscr{G}=\operatorname{Spf}R.
\end{align*}
Using this, we will concoct a dual of $[\mathscr{G}\to G]$ of the form $[L^D\to K]$, for some smooth commutative connected algebraic group $K$ fitting into a short exact sequence
\[0\to \mathscr{G}^D\to K\to A^\prime\to 0,\]
where $A^\prime$ is the dual abelian variety of $A$. A functorial definition of $K$ is given by the following lemma.

\begin{lemma}\label{lemma on K}
Let $[\mathscr{G}\to A]$ be the generalized 1-motive defined by the composition $\mathscr{G}\to G\to A$. The complex $\mathsf{R}\underline{\operatorname{Hom}}([\mathscr{G}\to A],\mathbb{G}_m)$ has no cohomology in degrees $0$ and $2$. Moreover, $\underline{\operatorname{Ext}}^1([\mathscr{G}\to A],\mathbb{G}_m)$ is representable by a smooth commutative connected algebraic group.
\end{lemma}

\begin{proof}
Applying the functor $\mathsf{R}\underline{\operatorname{Hom}}(-,\mathbb{G}_m)$ to the fiber sequence $\mathscr{G}\to A\to [\mathscr{G}\to A]$, we obtain the following long exact sequence:
\[\begin{tikzcd}
	0 & {\underline{\operatorname{Ext}}^0([\mathscr{G}\to A],\mathbb{G}_m)}\ar[draw=none]{d}[name=X, anchor=center]{} & 0 & {\mathscr{G}^D} \\
	& {\underline{\operatorname{Ext}}^1([\mathscr{G}\to A],\mathbb{G}_m)}\ar[draw=none]{d}[name=Y, anchor=center]{} & {A^\prime} & 0 \\
	& {\underline{\operatorname{Ext}}^2([\mathscr{G}\to A],\mathbb{G}_m)} & 0.
	\arrow[from=1-1, to=1-2]
	\arrow[from=1-2, to=1-3]
	\arrow[from=1-3, to=1-4]
	\arrow[from=2-2, to=2-3]
	\arrow[from=2-3, to=2-4]
	\arrow[from=3-2, to=3-3]
	\ar[from=1-4, to=2-2, crossing over, rounded corners,
            to path={ -- ([xshift=5ex]\tikztostart.east)
                      |- (X.center) \tikztonodes
                      -| ([xshift=-5ex]\tikztotarget.west)
                      -- (\tikztotarget)}]{dlll}[at end]{}
    \ar[from=2-4, to=3-2, crossing over, rounded corners,
            to path={ -- ([xshift=5ex]\tikztostart.east)
                      |- (Y.center) \tikztonodes
                      -| ([xshift=-5ex]\tikztotarget.west)
                      -- (\tikztotarget)}]{dlll}[at end]{}
\end{tikzcd}\]
Here, we use Propositions~\ref{cartier dual of abelian schemes and tori}, \ref{ext G^}, and \ref{vanishing ext2} to justify the non-trivial zeros above. The vanishing results follow directly, and the representability of $\underline{\operatorname{Ext}}^1([\mathscr{G}\to A],\mathbb{G}_m)$ by a smooth commutative connected algebraic group follows by descent. (See the following remark for an explanation of how these descent arguments work.)
\end{proof}

\begin{remark}[On properties inherited by extensions]\label{remark on descent}
	Let $S$ be a base scheme, and consider an extension of abelian sheaves on the site $(\textsf{Sch}/S)_\text{fppf}$:
	\[0 \to \mathscr{H}\xrightarrow{\ \varphi\ } \mathscr{E}\xrightarrow{\ \psi\ } \mathscr{G}\to 0.\]
	The composition $\mathscr{H}\times \mathscr{E}\xrightarrow{\ \varphi\times\operatorname{id}\ }\mathscr{E}\times\mathscr{E}\xrightarrow{\ m\ } \mathscr{E}$, where $m$ is the group operation of $\mathscr{E}$, defines an action of $\mathscr{H}$ on $\mathscr{E}$. This action fits into the cartesian diagram
\[\begin{tikzcd}[ampersand replacement=\&]
	{\mathscr{H}\times\mathscr{E}} \&\& {\mathscr{E}\times\mathscr{E}} \& {\mathscr{E}} \\
	{\mathscr{E}} \&\&\& {\mathscr{G},}
	\arrow["{\varphi\times\operatorname{id}}", from=1-1, to=1-3]
	\arrow["{\operatorname{pr}_2}"', from=1-1, to=2-1]
	\arrow["m", from=1-3, to=1-4]
	\arrow["\psi", from=1-4, to=2-4]
	\arrow["\psi", from=2-1, to=2-4]
\end{tikzcd}\]
	giving that $\mathscr{E}$ is an $\mathscr{H}$-torsor over $\mathscr{G}$. In other words, we have an $\mathscr{H}$-equivariant isomorphism $\mathscr{H}\times \mathscr{E}\to \mathscr{E}\times_{\mathscr{G}}\mathscr{E}$.
	
	Assume that $\mathscr{G}$ and $\mathscr{H}$ are represented by group algebraic spaces $G$ and $H$ over $S$, respectively. By \cite[Tag \href{https://stacks.math.columbia.edu/tag/00WN}{00WN}]{Stacks}, there exists a scheme $T$ with a map $T\to \mathscr{E}$ such that the composition
	\[T\to \mathscr{E}\to G\]
	is an fppf cover. This induces a trivialization $H\times T\simeq \mathscr{E}\times_{G}T$. Since the definition of an algebraic space is local in the fppf topology \cite[Tag \href{https://stacks.math.columbia.edu/tag/04SK}{04SK}]{Stacks}, it follows that $\mathscr{E}$ is representable by an algebraic space $E$ over $G$. Therefore, $E$ is also a group algebraic space over $S$.
	
	Let $\mathcal{P}$ be a property of morphisms of algebraic spaces that is stable under base change and fppf-local on the base. The same reasoning implies that if $H\to S$ satisfies $\mathcal{P}$, then so does $E\to G$. Furthermore, if $\mathcal{P}$ is stable under composition and is satisfied by $G\to S$, then $E\to S$ also satisfies $\mathcal{P}$. (See \cite[Tags \href{https://stacks.math.columbia.edu/tag/03H8}{03H8} and \href{https://stacks.math.columbia.edu/tag/03YE}{03YE}]{Stacks} for an extensive list of such properties $\mathcal{P}$.)
	
	Now, suppose that $S$ is the spectrum of a field $k$. If $G$ and $H$ are schemes, then the maps $G\to S$ and $H\to S$ are necessarily separated \cite[Tag \href{https://stacks.math.columbia.edu/tag/047L}{047L}]{Stacks}. The preceding discussion implies that $E\to S$ is also separated, and so \cite[Tag \href{https://stacks.math.columbia.edu/tag/0B8G}{0B8G}]{Stacks} shows that $E$ is a scheme as well.
	
	For the remainder of this discussion, assume that $G$ and $H$ are group algebraic spaces locally of finite type over $k$. Then, we note that $E$ has dimension $\dim G+\dim H$. Indeed, \cite[Tag \href{https://stacks.math.columbia.edu/tag/0AFH}{0AFH}]{Stacks} says that the relative dimension of $E\to G$ is $\dim E-\dim G$. However, this relative dimension is also the dimension of the fiber $H$, completing the proof.
	
	Since $E\to G$ is faithfully flat of finite presentation, the induced morphism $|E|\to |G|$ between the underlying topological spaces is a quotient map \cite[Tag \href{https://stacks.math.columbia.edu/tag/0413}{0413}]{Stacks}. Although the fibers of $|E|\to |G|$ may not be homeomorphic to $|H|$, there exists a continuous surjection from $|H|$ to them, as stated in \cite[Tag \href{https://stacks.math.columbia.edu/tag/03H4}{03H4}]{Stacks}. In particular, if $G$ and $H$ are connected, so is $E$.
\end{remark}

Following Laumon, we will systematically denote by $K$ the algebraic group representing $\underline{\operatorname{Ext}}^1([\mathscr{G}\to A],\mathbb{G}_m)$. The following variant of (stacky) Cartier duality was initially introduced in \cite{deligne1974theorie} and was subsequently generalized in \cite{Lau} and \cite{russell2013albanese}.

\begin{definition}\label{generalized 1-motive}
	Let $[\mathscr{G}\to G]$ be a generalized 1-motive. We define the \emph{Laumon dual} $[\mathscr{G}\to G]^\tri$ to be the generalized 1-motive $[L^D\to K]$, where $L^D\to K$ is the connecting morphism induced by the fiber sequence $L\to [\mathscr{G}\to G]\to [\mathscr{G}\to A]$ via the Cartier duality functor.
\end{definition}

We note that the Laumon dual is functorial on the generalized 1-motive. Indeed, by the functoriality of the Barsotti--Chevalley theorem, a morphism of complexes
\[\begin{tikzcd}[ampersand replacement=\&]
	{\mathscr{G}_1} \& G_1 \\
	{\mathscr{G}_2} \& {G_2}
	\arrow[from=1-1, to=1-2]
	\arrow[from=1-1, to=2-1]
	\arrow[from=1-2, to=2-2]
	\arrow[from=2-1, to=2-2]
\end{tikzcd}\]
between two generalized 1-motives induces maps $L_1\to L_2$ and $A_1\to A_2$ making the diagram
\[\begin{tikzcd}[ampersand replacement=\&]
	0 \& {L_1} \& {G_1} \& {A_1} \& 0 \\
	0 \& {L_2} \& {G_2} \& {A_2} \& 0
	\arrow[from=1-1, to=1-2]
	\arrow[from=1-2, to=1-3]
	\arrow[from=1-2, to=2-2]
	\arrow[from=1-3, to=1-4]
	\arrow[from=1-3, to=2-3]
	\arrow[from=1-4, to=1-5]
	\arrow[from=1-4, to=2-4]
	\arrow[from=2-1, to=2-2]
	\arrow[from=2-2, to=2-3]
	\arrow[from=2-3, to=2-4]
	\arrow[from=2-4, to=2-5]
\end{tikzcd}\]
commute. Consequently, we obtain a commutative diagram in the derived category of abelian sheaves
\[\begin{tikzcd}[ampersand replacement=\&]
	{L_1} \& {[\mathscr{G}_1\to G_1]} \& {[\mathscr{G}_1\to A_1]} \\
	{L_2} \& {[\mathscr{G}_2\to G_2]} \& {[\mathscr{G}_2\to A_2].}
	\arrow[from=1-1, to=1-2]
	\arrow[from=1-1, to=2-1]
	\arrow[from=1-2, to=1-3]
	\arrow[from=1-2, to=2-2]
	\arrow[from=1-3, to=2-3]
	\arrow[from=2-1, to=2-2]
	\arrow[from=2-2, to=2-3]
\end{tikzcd}\]
The Cartier duality functor then induces a morphism $[\mathscr{G}_2\to G_2]^\tri\to [\mathscr{G}_1\to G_1]^\tri$.

Our next proposition gives a comparison between the Laumon dual defined above and the stacky Cartier dual of Definition~\ref{def stacky cartier dual}.

\begin{proposition}\label{comparison dualities}
	Let $[\mathscr{G}\to G]$ be a generalized 1-motive. There exists a derived category map $[\mathscr{G}\to G]^\tri\to [\mathscr{G}\to G]^\vee$, whose cofiber is $\underline{\operatorname{Ext}}^1(L,\mathbb{G}_m)$.
\end{proposition}

\begin{proof}
	The fiber sequence defining the Laumon dual induces the fiber sequence below.
	\[\mathsf{R}\underline{\operatorname{Hom}}([\mathscr{G}\to A],\mathbb{G}_m[1])\to \mathsf{R}\underline{\operatorname{Hom}}([\mathscr{G}\to G],\mathbb{G}_m[1])\to \mathsf{R}\underline{\operatorname{Hom}}(L,\mathbb{G}_m[1])\]
	By Lemma~\ref{lemma on K}, $\underline{\operatorname{Ext}}^2([\mathscr{G}\to A],\mathbb{G}_m)$ vanishes and so $\underline{\operatorname{Ext}}^1([\mathscr{G}\to G],\mathbb{G}_m)\to \underline{\operatorname{Ext}}^1(L,\mathbb{G}_m)$ is an epimorphism. Then, \cite[Lem.\ 3.10]{brochard2021duality} implies that
		\[\tau_{\leq 0}\mathsf{R}\underline{\operatorname{Hom}}([\mathscr{G}\to A],\mathbb{G}_m[1])\to \tau_{\leq 0}\mathsf{R}\underline{\operatorname{Hom}}([\mathscr{G}\to G],\mathbb{G}_m[1])\to \tau_{\leq 0}\mathsf{R}\underline{\operatorname{Hom}}(L,\mathbb{G}_m[1])\]
	is also a fiber sequence. Yet another application of Lemma~\ref{lemma on K} gives that
	\[\tau_{\leq 0}\mathsf{R}\underline{\operatorname{Hom}}([\mathscr{G}\to A],\mathbb{G}_m[1])\simeq K\]
	and so, up to a shift, the fiber sequence just obtained is $L^\vee[-1]\to K\to [\mathscr{G}\to G]^\vee$. Since $L^D\simeq \tau_{\leq 0}(L^\vee[-1])$, there is a natural map $L^D\to L^\vee[-1]$ making the square
\[\begin{tikzcd}
	{L^D} & K & {[L^D\to K]} \\
	{L^\vee[-1]} & K & {[\mathscr{G}\to G]^\vee}
	\arrow[from=1-1, to=1-2]
	\arrow[from=1-2, to=1-3]
	\arrow[color=ocre, from=1-1, to=2-1]
	\arrow[from=2-1, to=2-2]
	\arrow[from=2-2, to=2-3]
	\arrow[equals, from=1-2, to=2-2]
\end{tikzcd}\]
commute and inducing a morphism of fiber sequences. In this way we obtain the desired comparison map.

Now, according to \cite[Tag \href{https://stacks.math.columbia.edu/tag/08J5}{08J5}]{Stacks}, there exists a fiber sequence $L^D\to L^\vee[-1]\to \underline{\operatorname{Ext}}^1(L,\mathbb{G}_m)$. Finally, the octahedral axiom \cite[Tag \href{https://stacks.math.columbia.edu/tag/05R0}{05R0}]{Stacks} gives that the cofiber of $[\mathscr{G}\to G]^\tri\to [\mathscr{G}\to G]^\vee$ is isomorphic to $\underline{\operatorname{Ext}}^1(L,\mathbb{G}_m)$.
\end{proof}

This proposition implies that the comparison map $[\mathscr{G}\to G]^\tri\to [\mathscr{G}\to G]^\vee$ becomes an isomorphism precisely when $\underline{\operatorname{Ext}}^1(L,\mathbb{G}_m)=0$. This condition is satisfied when $G$ is semiabelian, thereby demonstrating that the Cartier dual on 1-motives, as defined by Deligne, agrees with the stacky Cartier dual. Extending this result, we derive the following corollary.

\begin{corollary}\label{both duals coincide often}
	The comparison morphism $[\mathscr{G}\to G]^\tri\to [\mathscr{G}\to G]^\vee$ is always an isomorphism when the base field $k$ has positive characteristic. When $k$ has characteristic zero, the comparison map is an isomorphism if and only if $G$ is a semiabelian variety.
\end{corollary}

\begin{proof}
	In positive characteristic, the abelian sheaf $\underline{\operatorname{Ext}}^1(L,\mathbb{G}_m)$ always vanishes due to \cite[Prop.\ 2.2.17]{rosengarten2023tate}. If $k$ has characteristic zero, $L$ is a product of a torus and a vector group $U$ \cite[Cor.\ 16.15]{M}. In particular, Proposition~\ref{T'} implies that $\underline{\operatorname{Ext}}^1(L,\mathbb{G}_m)\simeq \underline{\operatorname{Ext}}^1(U,\mathbb{G}_m)$. By \cite[Thm.~C]{ribeiro2025}, the latter sheaf vanishes precisely when $U$ does. 
\end{proof}

For the remainder of this section, we assume the base field $k$ has characteristic zero and $\mathscr{G}=\widehat{G}$, so that $[\widehat{G}\to G]\simeq G_\text{dR}$. We now provide a more explicit description of $K=\underline{\operatorname{Ext}}^1([\widehat{G}\to A],\mathbb{G}_m)$. The long exact sequence derived from the extension
\[0\to \widehat{A}\to A\to A_\text{dR}\to 0,\]
via the Cartier duality functor, results in the short exact sequence
\[0\to \Omega_A\to A^\natural\to A^\prime\to 0,\]
as noted in Remark~\ref{remark 3}. Now, the quotient map $\psi\colon G\to A$ induces a pullback map $\psi^*\colon \Omega_A\to \Omega_G$, and we consider the corresponding pushout extension.
\[\begin{tikzcd}
	0 & {\Omega_A} & {A^\natural} & {A^\prime} & 0 \\
	0 & {\Omega_G} & (\Omega_G\times A^\natural)/\Omega_A & {A^\prime} & 0
	\arrow[from=1-1, to=1-2]
	\arrow[from=1-2, to=1-3]
	\arrow[from=1-3, to=1-4]
	\arrow[from=1-4, to=1-5]
	\arrow[from=2-1, to=2-2]
	\arrow[from=2-2, to=2-3]
	\arrow[from=2-3, to=2-4]
	\arrow[from=2-4, to=2-5]
	\arrow["\psi^*", from=1-2, to=2-2]
	\arrow[from=1-3, to=2-3]
	\arrow[equals, from=1-4, to=2-4]
\end{tikzcd}\]
Here, and throughout this work, $\Omega_G$ represents the vector group consisting of the \emph{invariant} differentials of $G$.

\begin{proposition}\label{characterization of K}
	The algebraic group $K=\underline{\operatorname{Ext}}^1([\widehat{G}\to A],\mathbb{G}_m)$ is naturally isomorphic to the quotient $(\Omega_G\times A^\natural)/\Omega_A$.
\end{proposition}

\begin{proof}
Consider the following commutative diagram, whose rows are fiber sequences.	
\[\begin{tikzcd}
	{\widehat{G}} & A & {[\widehat{G}\to A]} \\
	{\widehat{A}} & A & {[\widehat{A}\to A]}
	\arrow[equals, from=1-2, to=2-2]
	\arrow[from=1-3, to=2-3]
	\arrow[from=1-1, to=1-2]
	\arrow[from=1-2, to=1-3]
	\arrow[from=2-1, to=2-2]
	\arrow[from=2-2, to=2-3]
	\arrow[from=1-1, to=2-1]
\end{tikzcd}\]
After applying $\mathsf{R}\underline{\operatorname{Hom}}(-,\mathbb{G}_m)$ and taking long exact sequences in cohomology, we obtain the following commutative diagram with exact rows:
\[\begin{tikzcd}
	0 & {\Omega_A} & {A^\natural} & {A^\prime} & 0 \\
	0 & {\Omega_G} & K & {A^\prime} & 0.
	\arrow[from=1-1, to=1-2]
	\arrow[from=1-2, to=1-3]
	\arrow[from=1-3, to=1-4]
	\arrow[from=1-4, to=1-5]
	\arrow[from=2-1, to=2-2]
	\arrow[from=2-2, to=2-3]
	\arrow[from=2-3, to=2-4]
	\arrow[from=2-4, to=2-5]
	\arrow["\psi^*", from=1-2, to=2-2]
	\arrow[from=1-3, to=2-3]
	\arrow[equals, from=1-4, to=2-4]
\end{tikzcd}\]
We assert that the square on the left is cocartesian. This means that the complex
\[0\to \Omega_A\to \Omega_G \times A^\natural\to K\to 0,\]
with the action of $\Omega_A$ on $\Omega_G \times A^\natural$ as in the typical pushout construction, is exact. This complex is part of the following larger commutative diagram.
\[\begin{tikzcd}
	&& 0 & 0 \\
	& 0 & {\Omega_G} & {\Omega_G} & 0 \\
	0 & {\Omega_A} & {\Omega_G\times A^\natural} & K & 0 \\
	0 & {\Omega_A} & {A^\natural} & {A^\prime} & 0 \\
	& 0 & 0 & 0
	\arrow[from=3-1, to=3-2]
	\arrow[from=3-2, to=3-3]
	\arrow[from=3-3, to=3-4]
	\arrow[from=3-4, to=3-5]
	\arrow[from=1-4, to=2-4]
	\arrow[from=2-4, to=3-4]
	\arrow[from=3-4, to=4-4]
	\arrow[from=1-3, to=2-3]
	\arrow[from=2-2, to=2-3]
	\arrow[equals, from=2-3, to=2-4]
	\arrow[from=2-4, to=2-5]
	\arrow[from=2-3, to=3-3]
	\arrow[from=3-3, to=4-3]
	\arrow[from=2-2, to=3-2]
	\arrow[equals, from=3-2, to=4-2]
	\arrow[from=4-2, to=5-2]
	\arrow[from=4-3, to=5-3]
	\arrow[from=4-4, to=5-4]
	\arrow[from=4-1, to=4-2]
	\arrow[from=4-2, to=4-3]
	\arrow[from=4-3, to=4-4]
	\arrow[from=4-4, to=4-5]
\end{tikzcd}\]
In this diagram, every column is exact, and both the top and bottom rows are exact. Therefore, the middle row must also be exact.
\end{proof}

Let us fix some notation. Given that the algebraic group $K$ is a quotient of $\Omega_G\times A^\natural$, we will denote its points as (equivalence classes of) pairs $(\omega,(\mathscr{L},\nabla))$, where $\omega\in \Omega_G$ and $(\mathscr{L},\nabla)\in A^\natural$.\footnote{This is a slight abuse of notation since the sheafification involved in defining the quotient sheaf may be non-trivial. That being said, by Serre vanishing, $K(S)$ indeed is the quotient of $\Omega_G(S)\times A^\natural(S)$ by $\Omega_A(S)$ for affine $S$.\label{footnote serre vanishing}} We will also systematically denote the natural maps $L\to G$ and $G\to A$ as $\varphi$ and $\psi$, respectively.

\begin{remark}\label{description of the laumon map}
The morphism $L^D\to K$ in the Laumon dual of $[\widehat{G}\to G]$ can also be explicitly described: it maps a character $\chi\in L^D$ to $[\omega,(\mathscr{L},\nabla)]$, where $\omega$ is any element of $\Omega_G$ satisfying $\varphi^*\omega=\mathrm{d}\chi/\chi$, and $(\mathscr{L},\nabla)$ is the unique element of $A^\natural$ satisfying $\psi^*(\mathscr{L},\nabla)\simeq (\mathcal{O}_G,\mathrm{d}-\omega)$. Note that, since the Cartier dual of $G_\text{dR}$ vanishes, the map $L^D\to K$ is a monomorphism. In particular, the Laumon dual of $G_\text{dR}$ is a usual abelian sheaf.
\end{remark}

Finally, we describe the comparison map $[\widehat{G}\to G]^\tri\to [\widehat{G}\to G]^\vee$ from Proposition~\ref{comparison dualities}. The universal property of pushouts permits us to define the morphism of abelian sheaves
\begin{align*}K&\to G^\natural \\ [\omega,(\mathscr{L},\nabla)]&\mapsto (\mathcal{O}_G,\mathrm{d}+\omega)\otimes_{\mathcal{O}_G} \psi^*(\mathscr{L},\nabla).\end{align*}
This map factors through the quotient, resulting in a morphism $\gamma\colon K/L^D\to G^\natural$.

\begin{proposition}\label{computation of ker gamma}
The morphism $\gamma\colon K/L^D\to G^\natural$, sending an equivalence class $[\omega,(\mathscr{L},\nabla)]$ to the line bundle with integrable connection
\[(\mathcal{O}_G,\mathrm{d}+\omega)\otimes_{\mathcal{O}_G} \psi^*(\mathscr{L},\nabla),\]
coincides with the comparison map $G_{\normalfont\text{dR}}^\tri\to G_{\normalfont\text{dR}}^\vee$. In particular, $\gamma$ is a monomorphism and its cokernel is $\underline{\operatorname{Ext}}^1(L,\mathbb{G}_m)$.
\end{proposition}

\begin{proof}
	Consider the morphism $K\to G^\natural$, defined above using the universal property of pushouts. We affirm that the diagram
\[\begin{tikzcd}
	K & {G^\natural} \\
	{\underline{\operatorname{Ext}}^1([\widehat{G}\to A],\mathbb{G}_m)} & {\underline{\operatorname{Ext}}^1([\widehat{G}\to G],\mathbb{G}_m),}
	\arrow["{\rotatebox[origin=c]{90}{$\sim$}}", from=1-1, to=2-1]
	\arrow[equals, from=1-2, to=2-2]
	\arrow[from=1-1, to=1-2]
	\arrow[from=2-1, to=2-2]
\end{tikzcd}\]
on which the map $\underline{\operatorname{Ext}}^1([\widehat{G}\to A],\mathbb{G}_m)\to \underline{\operatorname{Ext}}^1([\widehat{G}\to G],\mathbb{G}_m)$ is induced by the natural morphism of complexes $[\widehat{G}\to G]\to [\widehat{G}\to A]$, commutes. This is the same as showing that the diagram
\[\begin{tikzcd}
	& {\underline{\operatorname{Ext}}^1([\widehat{A}\to A],\mathbb{G}_m)} \\
	{\underline{\operatorname{Hom}}(\widehat{G},\mathbb{G}_m)} & {\underline{\operatorname{Ext}}^1([\widehat{G}\to A],\mathbb{G}_m)} \\
	&& {\underline{\operatorname{Ext}}^1([\widehat{G}\to G],\mathbb{G}_m)}
	\arrow[from=2-2, to=3-3]
	\arrow[from=1-2, to=2-2]
	\arrow[bend left=20, from=1-2, to=3-3]
	\arrow[bend right=15, from=2-1, to=3-3]
	\arrow[from=2-1, to=2-2]
\end{tikzcd}\]
commutes. The upper triangle clearly commutes by functoriality, and the lower triangle can be seen to commute by applying the functor $\mathsf{R}\underline{\operatorname{Hom}}(-,\mathbb{G}_m)$ to the morphism of fiber sequences
\[\begin{tikzcd}
	{\widehat{G}} & G & {[\widehat{G}\to G]} \\
	{\widehat{G}} & A & {[\widehat{G}\to A],}
	\arrow[equals, from=1-1, to=2-1]
	\arrow["\psi", from=1-2, to=2-2]
	\arrow[from=1-3, to=2-3]
	\arrow[from=2-1, to=2-2]
	\arrow[from=1-1, to=1-2]
	\arrow[from=1-2, to=1-3]
	\arrow[from=2-2, to=2-3]
\end{tikzcd}\]
and taking long exact sequences in cohomology. Now, as in the proof of Proposition~\ref{comparison dualities}, there are two dashed morphisms making the diagram
\[\begin{tikzcd}
	K & {[L^D\to K]} & {L^D[1]} \\
	K & {G^\natural} & {L^\vee}
	\arrow[equals, from=1-1, to=2-1]
	\arrow[dashed, color=ocre, from=1-2, to=2-2]
	\arrow[from=1-3, to=2-3]
	\arrow[from=2-1, to=2-2]
	\arrow[from=1-1, to=1-2]
	\arrow[from=1-2, to=1-3]
	\arrow[from=2-2, to=2-3]
\end{tikzcd}\]
commute: the comparison map of Proposition~\ref{comparison dualities} and $\gamma$. The fact that they coincide follows from \cite[Tag \href{https://stacks.math.columbia.edu/tag/0FWZ}{0FWZ}]{Stacks}.
\end{proof}

Recall that the linear part $L$ of $G$ decomposes as a product of a torus $T$, whose Cartier dual is denoted as $X$, and a vector group $U$. Inasmuch as $\underline{\operatorname{Ext}}^1(L,\mathbb{G}_m)\simeq \underline{\operatorname{Ext}}^1(U,\mathbb{G}_m)$ has no $k$-points, due to Propositions~\ref{T'} and \ref{unipotent extension}, this computation is particularly useful for obtaining concrete information about character sheaves. 

\begin{corollary}\label{character sheaves on G}
	The group of isomorphism classes $\mathrm{H}^1_m(G_{\normalfont\text{dR}},\mathbb{G}_m)$ of character sheaves on $G$ fits into the short exact sequence
	\[0\to X\to (\Omega_G\times A^\natural(k))/\Omega_A\to \mathrm{H}^1_m(G_{\normalfont\text{dR}},\mathbb{G}_m)\to 0,\]
where the morphism on the left is described in Remark~\ref{description of the laumon map}, and the morphism on the right is as in Proposition~\ref{computation of ker gamma}. In particular, every character sheaf on $G$ is of the form
\[(\mathcal{O}_G,\mathrm{d}+\omega)\otimes_{\mathcal{O}_G} \psi^*(\mathscr{L},\nabla),\]
for some $\omega\in \Omega_G$ and $(\mathscr{L},\nabla)\in A^\natural(k)$.
\end{corollary}

\begin{proof}
	The preceding proposition gives an isomorphism $(K/L^D)(k)\simeq \mathrm{H}^1_m(G_{\normalfont\text{dR}},\mathbb{G}_m)$. Now, the group of $k$-points $(K/L^D)(k)$ fits into the long exact sequence
	\[0\to X\to K(k)\to (K/L^D)(k)\to \mathrm{H}^1(k,X)\times \mathrm{H}^1(k,\widehat{U^*}),\]
	where the term on the right vanishes due to \cite[Exp.\ VIII, Prop.\ 5.1]{raynaud2006groupes} and \cite[Rem.\ 2.2.18]{bhatt2022prismatic}. Finally, as noted in Footnote~\ref{footnote serre vanishing}, the group $K(k)$ is isomorphic to the quotient $(\Omega_G\times A^\natural(k))/\Omega_A$, completing the proof.
\end{proof}

\subsection{The moduli of character sheaves}

As discussed in the previous section, Laumon defined a dual $[\widehat{G} \to G]^\tri$ that, in a sense, eliminates the enigmatic object $U^\prime = \underline{\operatorname{Ext}}^1(U, \mathbb{G}_m)$ from the stacky Cartier dual $G^\natural = G_\text{dR}^\vee$. Now, according to Remark~\ref{description of the laumon map}, the Laumon dual of $G_\text{dR}$ is isomorphic to the quotient sheaf $K/(X\times \widehat{U^*})$. Since $\widehat{U^*}$ is a formal group, it cannot be representable. This leads us to the following definition.

\begin{definition}[Moduli space of character sheaves]\label{def G flat}
We denote by $G^\flat$ the abelian sheaf $K/X$, where $X\hookrightarrow K$ is the morphism that maps $\chi\in X$ to $[\omega,(\mathscr{L},\nabla)]$, where $\omega$ is any element of $\Omega_G$ satisfying $\varphi^*\omega=\mathrm{d}\chi/\chi$ and $(\mathscr{L},\nabla)$ is the unique element of $A^\natural$ satisfying $\psi^*(\mathscr{L},\nabla)\simeq (\mathcal{O}_G,\mathrm{d}-\omega)$.
\end{definition}

Note that although the map $X \to K$ is a monomorphism, it is not an immersion. Consequently, it is unclear whether the quotient $G^\flat = K/X$ is a scheme, and it may indeed fail to be one. A brief verification shows that the morphism
\begin{align*}
	k &\to \mathrm{H}^1_m(\mathbb{G}_{m,\text{dR}},\mathbb{G}_m)\\
	\alpha &\mapsto (\mathcal{O}_{\mathbb{G}_m},\mathrm{d}-\alpha\:\mathrm{d}x/x)
\end{align*}
is surjective and induces an isomorphism $\mathrm{H}^1_m(\mathbb{G}_{m,\text{dR}},\mathbb{G}_m)\simeq k/\mathbb{Z}$. Accordingly, the moduli space $\mathbb{G}_m^\flat$ is isomorphic to $\mathbb{G}_a/\mathbb{Z}$, an algebraic space\footnote{We emphasize that, as in \cite[Tag \href{https://stacks.math.columbia.edu/tag/025Y}{025Y}]{Stacks} and contrarily to \cite[Déf.\ 1.1]{laumon2018champs}, we \emph{do not} suppose that algebraic spaces are quasi-separated.} that is not a scheme. Nonetheless, the main result of this section, Theorem~\ref{exact sequence of flats}, asserts in Corollary~\ref{G flat is an algebraic space} that $G^\flat$ is as well-behaved as one could hope for.

\begin{remark}
	Let $S$ be a $k$-scheme. According to Propositions~\ref{cartier dual of dR}, \ref{sheafification map}, and \ref{comparison dualities}, there are natural morphisms
	\[G^\flat(S)\xrightarrow{\ \text{(1)}\ } G_\text{dR}^\tri(S) \xrightarrow{\ \text{(2)}\ } G^\natural(S) \xrightarrow{\ \text{(3)}\ } \mathrm{H}^1_m(G_\text{dR}\times S,\mathbb{G}_m).\]
	Proposition~\ref{local-to-global} and Theorem~\ref{global de rham} imply that the maps (1) and (3) are isomorphisms if $S$ is reduced. Furthermore, Proposition~\ref{unipotent extension} states that (2) is an isomorphism for seminormal $k$-schemes $S$. In particular, $G^\flat(k)$ is isomorphic to the group of isomorphism classes of character sheaves on $G$, thereby justifying its name.
\end{remark}

\begin{remark}\label{functoriality of G flat}
	The assignment $G\mapsto G^\flat$ is functorial. Specifically, consider a morphism $G_1\to G_2$ between commutative connected algebraic groups over a characteristic zero field. Since the Barsotti--Chevalley decomposition is functorial, we obtain the following commutative diagram
\[\begin{tikzcd}[ampersand replacement=\&]
	0 \& {L_1} \& {G_1} \& {A_1} \& 0 \\
	0 \& {L_2} \& {G_2} \& {A_2} \& 0,
	\arrow[from=1-1, to=1-2]
	\arrow[from=1-2, to=1-3]
	\arrow[from=1-2, to=2-2]
	\arrow[from=1-3, to=1-4]
	\arrow[from=1-3, to=2-3]
	\arrow[from=1-4, to=1-5]
	\arrow[from=1-4, to=2-4]
	\arrow[from=2-1, to=2-2]
	\arrow[from=2-2, to=2-3]
	\arrow[from=2-3, to=2-4]
	\arrow[from=2-4, to=2-5]
\end{tikzcd}\]
where $A_i$ are abelian varieties and $L_i$ are affine. The linear parts $L_i$ further decompose into a product of tori $T_i$ and unipotent groups $U_i$. By \cite[Cor.\ 14.18]{M}, the morphism $L_1\to L_2$ is a product of morphisms $T_1\to T_2$ and $U_1\to U_2$. In particular, this induces a map $X_2\to X_1$ between the character groups of $T_2$ and $T_1$, respectively. The functoriality of the Laumon dual ensures that the square on the right of the diagram
\[\begin{tikzcd}[ampersand replacement=\&]
	{X_2} \& {X_2\times\widehat{U_2^*}} \& {K_2} \\
	{X_1} \& {X_1\times\widehat{U_1^*}} \& {K_1}
	\arrow[from=1-1, to=1-2]
	\arrow[from=1-1, to=2-1]
	\arrow[from=1-2, to=1-3]
	\arrow[from=1-2, to=2-2]
	\arrow[from=1-3, to=2-3]
	\arrow[from=2-1, to=2-2]
	\arrow[from=2-2, to=2-3]
\end{tikzcd}\]
commutes. Consequently, the universal property of cokernels gives the desired map $G_2^\flat\to G_1^\flat$.
\end{remark}

\begin{example}
	As usual, let $T$ be a torus, $U$ a unipotent group, and $A$ be an abelian variety. We have that
	\[T^\flat\simeq  T^\natural \simeq  \Omega_T/X\simeq \mathfrak{t}^*/X,\quad U^\flat\simeq  \Omega_U \simeq U^*,\quad A^\flat\simeq  A^\natural.\]
	The first one is a group algebraic space, while the other two are algebraic groups.
\end{example}

Before going further, we note that there exists a large diagram relating many of the objects appearing in this section. Since both formal completions and the de Rham functor are exact, we obtain the following commutative diagram:
\[\begin{tikzcd}[ampersand replacement=\&]
	\& 0 \& 0 \& 0 \\
	0 \& {\widehat{T}\times \widehat{U}} \& {\widehat{G}} \& {\widehat{A}} \& 0 \\
	0 \& {T\times U} \& G \& A \& 0 \\
	0 \& {T_\text{dR}\times U_\text{dR}} \& {G_\text{dR}} \& {A_\text{dR}} \& 0 \\
	\& 0 \& 0 \& 0,
	\arrow[from=1-2, to=2-2]
	\arrow[from=1-3, to=2-3]
	\arrow[from=1-4, to=2-4]
	\arrow[from=2-1, to=2-2]
	\arrow[from=2-2, to=2-3]
	\arrow[from=2-2, to=3-2]
	\arrow[from=2-3, to=2-4]
	\arrow[from=2-3, to=3-3]
	\arrow[from=2-4, to=2-5]
	\arrow[from=2-4, to=3-4]
	\arrow[from=3-1, to=3-2]
	\arrow["\varphi", from=3-2, to=3-3]
	\arrow[from=3-2, to=4-2]
	\arrow["\psi", from=3-3, to=3-4]
	\arrow[from=3-3, to=4-3]
	\arrow[from=3-4, to=3-5]
	\arrow[from=3-4, to=4-4]
	\arrow[from=4-1, to=4-2]
	\arrow[from=4-2, to=4-3]
	\arrow[from=4-2, to=5-2]
	\arrow[from=4-3, to=4-4]
	\arrow[from=4-3, to=5-3]
	\arrow[from=4-4, to=4-5]
	\arrow[from=4-4, to=5-4]
\end{tikzcd}\]
in which every column and row is exact. By applying the Cartier duality functor, and passing to the long exact sequences in cohomology, we obtain
\[\begin{tikzcd}
	&& 0 & 0 \\
	& 0 \ar[draw=none]{ddd}[name=X, anchor=center]{} & {G^D} & {X\times \widehat{U^*}} \\
	0 & {\Omega_A} & {\Omega_G} & {\Omega_T\times\Omega_U} & 0 \\
	0 & {A^\natural} & {G^\natural} & {T^\natural\times U^\natural} & 0 \\
	& {A^\prime} & {G^\prime} & {U^\prime} & 0 \\
	& 0 & 0 & 0.
	\arrow[from=3-1, to=3-2]
	\arrow["{\psi^*}", from=3-2, to=3-3]
	\arrow["{\varphi^*}", from=3-3, to=3-4]
	\arrow[from=3-4, to=3-5]
	\arrow[from=4-1, to=4-2]
	\arrow["{\psi^*}", from=4-2, to=4-3]
	\arrow["{\varphi^*}", from=4-3, to=4-4]
	\arrow[from=4-4, to=4-5]
	\arrow[from=5-4, to=5-5]
	\arrow["{\varphi^*}", from=5-3, to=5-4]
	\arrow["{\psi^*}", from=5-2, to=5-3]
	\arrow[from=2-2, to=3-2]
	\arrow[from=3-2, to=4-2]
	\arrow[from=4-2, to=5-2]
	\arrow[from=2-3, to=3-3]
	\arrow[from=3-3, to=4-3]
	\arrow[from=4-3, to=5-3]
	\arrow[from=5-3, to=6-3]
	\arrow[from=5-2, to=6-2]
	\arrow[from=5-4, to=6-4]
	\arrow[from=4-4, to=5-4]
	\arrow[from=3-4, to=4-4]
	\arrow[from=2-4, to=3-4]
	\arrow[from=2-2, to=2-3]
	\arrow["{\varphi^*}", from=2-3, to=2-4]
	\arrow[from=1-3, to=2-3]
	\arrow[from=1-4, to=2-4]
 \ar[from=2-4, to=5-2, crossing over, rounded corners,
            to path={ -- ([xshift=10ex]\tikztostart.east)
                      |- (X.center) \tikztonodes
                      -| ([xshift=-10ex]\tikztotarget.west)
                      -- (\tikztotarget)}]{dddlll}[at end]{}
\end{tikzcd}\label{large diagram}\]
Here, every row (including the zigzag path) and every column is exact. The necessary computations and vanishing results have already been discussed in the previous section. We note that the morphisms in the columns have natural geometric interpretations, as given by
\[\begin{tikzcd}[row sep=tiny]
	0 & {G^D} & {\Omega_G} & {G^\natural} & {G^\prime} & 0 \\
	& \chi & {\mathrm{d}\chi/\chi} & {(\mathscr{L},\nabla)} & {\mathscr{L}} \\
	&& \omega & {(\mathcal{O}_G,\mathrm{d}+\omega).}
	\arrow[maps to, from=2-2, to=2-3]
	\arrow[from=1-1, to=1-2]
	\arrow[from=1-2, to=1-3]
	\arrow[from=1-3, to=1-4]
	\arrow[from=1-4, to=1-5]
	\arrow[from=1-5, to=1-6]
	\arrow[maps to, from=3-3, to=3-4]
	\arrow[maps to, from=2-4, to=2-5]
\end{tikzcd}\]

We are now in position to state the main theorem of this section.

\begin{theorem}\label{exact sequence of flats}
Let $G$ be a commutative connected algebraic group over a characteristic zero field. Write $G$ as an extension of an abelian variety $A$ by a product $T\times U$ of a torus $T$ and a unipotent group $U$. Then the complex $0\to A^\flat\to G^\flat \to T^\flat\times U^\flat\to 0$ is exact.
\end{theorem}

\begin{proof}
	Consider the map $K\to \Omega_T\times \Omega_U$ induced by $\varphi^*\colon\Omega_G\to \Omega_T\times\Omega_U$ and $0\colon A^\natural\to \Omega_T\times\Omega_U$. We claim that the composition $K\to \Omega_T\times \Omega_U\to \Omega_T/X\times \Omega_U$ descends to the quotient $K/X$. According to the universal property of the quotient, we need to verify that the composition
\[X\to X\times \widehat{U^*}\to K\to \Omega_T\times \Omega_U\to \Omega_T/X\times \Omega_U\]
is zero. Applying the functor $\mathsf{R}\underline{\operatorname{Hom}}(-,\mathbb{G}_m)$ to the morphism of fiber sequences
\[\begin{tikzcd}
	L & {[\widehat{L}\to L]} & {\widehat{L}[1]} \\
	L & {[\widehat{G}\to G]} & {[\widehat{G}\to A],}
	\arrow[from=2-1, to=2-2]
	\arrow[from=2-2, to=2-3]
	\arrow[from=1-1, to=1-2]
	\arrow[from=1-2, to=1-3]
	\arrow[from=1-2, to=2-2]
	\arrow[equals, from=1-1, to=2-1]
	\arrow[from=1-3, to=2-3]
\end{tikzcd}\]
and taking long exact sequences in cohomology, we find that the composition $X\times\widehat{U^*}\to K\to \Omega_T\times \Omega_U$ is the familiar map appearing on Page~\pageref{large diagram}. In particular, this composition is the product of $X\to \Omega_T$ and $\widehat{U^*}\to \Omega_U$. It follows that our large composition vanishes, and we obtain a map $K/X\to \Omega_T/X\times \Omega_U$.

Now, we have every morphism needed to consider the following commutative diagram
\[\begin{tikzcd}
	& 0 & 0 & 0 \\
	0 & 0 & X & X & 0 \\
	0 & {A^\natural} & K & {\Omega_T\times\Omega_U} & 0 \\
	0 & {A^\natural} & {K/X} & {\Omega_T/X\times \Omega_U} & 0 \\
	& 0 & 0 & 0,
	\arrow[from=1-3, to=2-3]
	\arrow[from=2-3, to=3-3]
	\arrow[from=3-3, to=4-3]
	\arrow[from=4-3, to=5-3]
	\arrow[from=2-2, to=2-3]
	\arrow[equals, from=2-3, to=2-4]
	\arrow[from=2-4, to=2-5]
	\arrow[from=2-1, to=2-2]
	\arrow[from=1-2, to=2-2]
	\arrow[from=2-2, to=3-2]
	\arrow[equals, from=3-2, to=4-2]
	\arrow[from=4-2, to=5-2]
	\arrow[from=1-4, to=2-4]
	\arrow[from=2-4, to=3-4]
	\arrow[from=3-4, to=4-4]
	\arrow[from=4-4, to=5-4]
	\arrow[from=3-2, to=3-3]
	\arrow[from=4-2, to=4-3]
	\arrow[from=3-3, to=3-4]
	\arrow[from=4-3, to=4-4]
	\arrow[from=3-1, to=3-2]
	\arrow[from=4-1, to=4-2]
	\arrow[from=3-4, to=3-5]
	\arrow[from=4-4, to=4-5]
\end{tikzcd}\]
whose columns are clearly exact. Since the top row is also exact, by the nine-lemma, it suffices to prove that the middle row is exact. This holds by an application of the snake lemma in the pushout extension defining $K$.
\end{proof}

The theorem above implies that $G^\flat$ is a "coarse moduli space" for $G^\natural$, in the sense that $G^\flat$ is represented by an algebraic space with the same $k$-points as $G^\natural$.\footnote{However, it remains unclear whether there is a natural morphism $G^\natural \to G^\flat$, or if such a morphism would satisfy the universal property.}

\begin{corollary}\label{G flat is an algebraic space}
The abelian sheaf $G^\flat$ is representable by a smooth commutative connected group algebraic space. Moreover, it satisfies $\dim G \leq \dim G^\flat \leq 2 \dim G$, with equality on the left if and only if $G$ is affine, and equality on the right if and only if $G$ is proper.
\end{corollary}

\begin{proof}
The discussion in Remark~\ref{remark on descent} indicates that $G^\flat$ is representable by a smooth commutative connected group algebraic space of dimension $\dim T^\flat+\dim U^\flat+\dim A^\flat$. Given that
\begin{align*}
	\dim T^\flat &=\dim \Omega_T/X=\dim \Omega_T=\dim T \\
	\dim U^\flat &= \dim U^* = \dim U\\
	\dim A^\flat &= \dim A^\natural = \dim \Omega_A+\dim A'=2\dim A,
\end{align*}
we find that $\dim G^\flat=\dim T+\dim U+2\dim A=\dim G+\dim A$, thus completing the proof.
\end{proof}

The full subcategory of $\textsf{Ab}((\textsf{Sch}/k)_\text{fppf})$ consisting of commutative connected algebraic groups is closed under extensions. Consequently, it inherits a structure of exact category. One might believe that this makes the functor $G\mapsto G^\flat$ exact. The following results provide some support for this conjecture.

\begin{proposition}\label{exact sequence of flats abelian}
	Let $0\to A_1\to A_2\to A_3\to 0$ be a short exact sequence of abelian varieties over a characteristic zero field. Then the complex $0\to A_3^\flat\to A_2^\flat\to A_1^\flat\to 0$ is exact.
\end{proposition}

\begin{proof}
	By Proposition~\ref{GdR is reasonable}, the complex $0\to A_{1,\text{dR}}\to A_{2,\text{dR}}\to A_{3,\text{dR}}\to 0$ is exact. The Cartier duality functor then induces the exact sequence
\[\underline{\operatorname{Hom}}(A_{1,\text{dR}},\mathbb{G}_m)=0\to A_3^\natural\to A_2^\natural\to A_1^\natural\to 0=\underline{\operatorname{Ext}}^2(A_{3,\text{dR}},\mathbb{G}_m),\]
where the vanishing results are due to Propositions~\ref{cartier dual of dR} and \ref{vanishing ext2}.
\end{proof}

\begin{proposition}\label{extensions of linear groups}
	Let $0\to L_1\to L_2\to L_3\to 0$ be a short exact sequence of linear commutative connected algebraic groups over a characteristic zero field. Then the complex $0\to L_3^\flat\to L_2^\flat\to L_1^\flat\to 0$ is exact.
\end{proposition}

\begin{proof}
	First, we decompose the linear groups $L_i$ as a product of tori $T_i$ and unipotent groups $U_i$. The exactness of the formal completion functor induces the short exact sequence
\[0 \to \widehat{T}_1\times \widehat{U}_1\to \widehat{T}_2\times \widehat{U}_2\to \widehat{T}_3\times \widehat{U}_3\to 0.\]
Then, the Cartier duality functor gives rise to the exact sequence
\[0\to \Omega_{T_3}\times \Omega_{U_3}\to \Omega_{T_2}\times \Omega_{U_2}\to \Omega_{T_1}\times \Omega_{U_1}\to 0 = \underline{\operatorname{Ext}}^1(\widehat{T}_3\times \widehat{U}_3,\mathbb{G}_m).\]
Here, we use Corollary~\ref{cartier dual of formal completion} and Proposition~\ref{ext G^}. As in Remark~\ref{functoriality of G flat}, we have induced maps $X_3\to X_2\to X_1$ between the character groups of the tori $T_i$, fitting into the diagram
\[\begin{tikzcd}[ampersand replacement=\&]
	\& 0 \& 0 \& 0 \\
	0 \& {X_3} \& {X_2} \& {X_1} \& 0 \\
	0 \& {\Omega_{T_3}\times \Omega_{U_3}} \& {\Omega_{T_2}\times \Omega_{U_2}} \& {\Omega_{T_1}\times \Omega_{U_1}} \& 0 \\
	0 \& {\Omega_{T_3}/X_3\times \Omega_{U_3}} \& {\Omega_{T_2}/X_2\times \Omega_{U_2}} \& {\Omega_{T_1}/X_1\times \Omega_{U_1}} \& 0 \\
	\& 0 \& 0 \& 0.
	\arrow[from=1-2, to=2-2]
	\arrow[from=1-3, to=2-3]
	\arrow[from=1-4, to=2-4]
	\arrow[from=2-1, to=2-2]
	\arrow[from=2-2, to=2-3]
	\arrow[from=2-2, to=3-2]
	\arrow[from=2-3, to=2-4]
	\arrow[from=2-3, to=3-3]
	\arrow[from=2-4, to=2-5]
	\arrow[from=2-4, to=3-4]
	\arrow[from=3-1, to=3-2]
	\arrow[from=3-2, to=3-3]
	\arrow[from=3-2, to=4-2]
	\arrow[from=3-3, to=3-4]
	\arrow[from=3-3, to=4-3]
	\arrow[from=3-4, to=3-5]
	\arrow[from=3-4, to=4-4]
	\arrow[from=4-1, to=4-2]
	\arrow[from=4-2, to=4-3]
	\arrow[from=4-2, to=5-2]
	\arrow[from=4-3, to=4-4]
	\arrow[from=4-3, to=5-3]
	\arrow[from=4-4, to=4-5]
	\arrow[from=4-4, to=5-4]
\end{tikzcd}\]
The first row above is exact due to Proposition~\ref{T'}, and the nine lemma implies that the bottom row is also exact.
\end{proof}

\begin{proposition}\label{flat preserve injections}
Let $0\to G_1\to G_2\to G_3\to 0$ be a short exact sequence of commutative connected algebraic groups over a characteristic zero field. Then $G_3^\flat\to G_2^\flat$ is a monomorphism and $G_2^\flat\to G_1^\flat$ is an epimorphism.
\end{proposition}

\begin{proof}
	Consider the following commutative diagram, composed of the short exact sequence of commutative connected algebraic groups $G_i$, along with the induced maps on their Barsotti--Chevalley decompositions.
\[\begin{tikzcd}[ampersand replacement=\&]
	\& 0 \& 0 \& 0 \\
	0 \& {T_1\times U_1} \& {T_2\times U_2} \& {T_3\times U_3} \& 0 \\
	0 \& {G_1} \& {G_2} \& {G_3} \& 0 \\
	0 \& {A_1} \& {A_2} \& {A_3} \& 0 \\
	\& 0 \& 0 \& 0
	\arrow[from=1-2, to=2-2]
	\arrow[from=1-3, to=2-3]
	\arrow[from=1-4, to=2-4]
	\arrow[from=2-1, to=2-2]
	\arrow[from=2-2, to=2-3]
	\arrow[from=2-2, to=3-2]
	\arrow[from=2-3, to=2-4]
	\arrow[from=2-3, to=3-3]
	\arrow[from=2-4, to=2-5]
	\arrow[from=2-4, to=3-4]
	\arrow[from=3-1, to=3-2]
	\arrow[from=3-2, to=3-3]
	\arrow[from=3-2, to=4-2]
	\arrow[from=3-3, to=3-4]
	\arrow[from=3-3, to=4-3]
	\arrow[from=3-4, to=3-5]
	\arrow[from=3-4, to=4-4]
	\arrow[from=4-1, to=4-2]
	\arrow[from=4-2, to=4-3]
	\arrow[from=4-2, to=5-2]
	\arrow[from=4-3, to=4-4]
	\arrow[from=4-3, to=5-3]
	\arrow[from=4-4, to=4-5]
	\arrow[from=4-4, to=5-4]
\end{tikzcd}\]

The exactness of Barsotti--Chevalley decompositions has been thoroughly examined by Brion in \cite{brion2017commutative}. He proved that the complex $0\to U_1\to U_2\to U_3\to 0$ is exact, but 
\[0\to T_1\to T_2\to T_3\to 0 \qquad \text{and}\qquad 0\to A_1\to A_2\to A_3\to 0\]
might only be exact up to isogeny \cite[Thm.\ 2.9, Lems.\ 4.3 and 4.7]{brion2017commutative}. That being said, the map $T_1\to T_2$ is a monomorphism, while $T_2\to T_3$ and $A_2\to A_3$ are epimorphisms.\footnote{Brion’s results apply in arbitrary characteristic, while the characteristic zero case required here is significantly simpler.}

Next, we apply the functor $(-)^\flat$ to the diagram above, obtaining the following commutative diagram
	\[\begin{tikzcd}[ampersand replacement=\&]
	\& 0 \& 0 \& 0 \\
	0 \& {A_3^\flat} \& {A_2^\flat} \& {A_1^\flat} \& 0 \\
	0 \& {G_3^\flat} \& {G_2^\flat} \& {G_1^\flat} \& 0 \\
	0 \& {T_3^\flat\times U_3^\flat} \& {T_2^\flat\times U_2^\flat} \& {T_1^\flat\times U_1^\flat} \& 0 \\
	\& 0 \& 0 \& 0,
	\arrow[from=1-2, to=2-2]
	\arrow[from=1-3, to=2-3]
	\arrow[from=1-4, to=2-4]
	\arrow[from=2-1, to=2-2]
	\arrow[from=2-2, to=2-3]
	\arrow[from=2-2, to=3-2]
	\arrow[from=2-3, to=2-4]
	\arrow[from=2-3, to=3-3]
	\arrow[from=2-4, to=2-5]
	\arrow[from=2-4, to=3-4]
	\arrow[from=3-1, to=3-2]
	\arrow[from=3-2, to=3-3]
	\arrow[from=3-2, to=4-2]
	\arrow[from=3-3, to=3-4]
	\arrow[from=3-3, to=4-3]
	\arrow[from=3-4, to=3-5]
	\arrow[from=3-4, to=4-4]
	\arrow[from=4-1, to=4-2]
	\arrow[from=4-2, to=4-3]
	\arrow[from=4-2, to=5-2]
	\arrow[from=4-3, to=4-4]
	\arrow[from=4-3, to=5-3]
	\arrow[from=4-4, to=4-5]
	\arrow[from=4-4, to=5-4]
\end{tikzcd}\]
whose rows are complexes and columns are exact. According to the preceding propositions, the maps $A_3^\flat\to A_2^\flat$ and $T_3^\flat\times U_3^\flat\to T_2^\flat\times U_2^\flat$ are monomorphisms. The long exact sequence in cohomology associated with the short exact sequence of complexes above implies that $G_3^\flat\to G_2^\flat$ is also a monomorphism.

We claim that $A_2^\flat\to A_1^\flat$ is an epimorphism. Let $F$ be the kernel of $A_1\to A_2$ and $B$ be its cokernel. Note that $F$ is a finite group, while $B$ is an abelian variety. Consider the short exact sequences
\[0\to F\to A_1\to A_1/F\to 0\qquad\text{and}\qquad 0\to A_1/F\to A_2\to B\to 0.\]
By passing to the de Rham spaces and considering the long exact sequences induced by the Cartier duality functor, we obtain that the maps $A_2^\flat\to (A_1/F)^\flat$ and $(A_1/F)^\flat\to A_1^\flat$ are epimorphisms. Hence, so is their composition. Since $T_2^\flat\times U_2^\flat\to T_1^\flat\times U_1^\flat$ is an epimorphism, the same argument as above proves the same for $G_2^\flat\to G_1^\flat$.
\end{proof}

As it turns out, the lack of exactness in the Barsotti--Chevalley decompositions is more than just a nuisance in the proof above. Inspired by \cite[Rem.\ 4.6]{brion2017commutative}, we provide below a counter-example to our conjecture.

\begin{example}\label{flats are not exact}
	Assume the base field $k$ contains a non-trivial $p$-th root of unity, for some prime $p$. Given an elliptic curve $E$ containing a rational point of order $p$, we let $\mu_p$ act diagonally on the product $E\times\mathbb{G}_m$ and consider the quotient $G$. This group fits into the short exact sequence
	\[0\to E\to G\to \mathbb{G}_m\to 0.\]
	
	Taking the Barsotti--Chevalley decomposition of each group, we obtain the following commutative diagram
\[\begin{tikzcd}[ampersand replacement=\&]
	\&\&\& 0 \& 0 \\
	\&\& 0 \& {\mathbb{G}_m} \& {\mathbb{G}_m} \& 0 \\
	\& 0 \& E \& G \& {\mathbb{G}_m} \& 0 \\
	0 \& {\mu_p} \& E \& {E/\mu_p} \& 0 \\
	\& 0 \& 0 \& 0,
	\arrow[from=1-4, to=2-4]
	\arrow[from=1-5, to=2-5]
	\arrow[from=2-3, to=2-4]
	\arrow[from=2-3, to=3-3]
	\arrow[equals, from=2-4, to=2-5]
	\arrow[from=2-4, to=3-4]
	\arrow[from=2-5, to=2-6]
	\arrow[equals, from=2-5, to=3-5]
	\arrow[from=3-2, to=3-3]
	\arrow[from=3-2, to=4-2]
	\arrow[from=3-3, to=3-4]
	\arrow[equals, from=3-3, to=4-3]
	\arrow[from=3-4, to=3-5]
	\arrow[from=3-4, to=4-4]
	\arrow[from=3-5, to=3-6]
	\arrow[from=3-5, to=4-5]
	\arrow[from=4-1, to=4-2]
	\arrow[from=4-2, to=4-3]
	\arrow[from=4-2, to=5-2]
	\arrow[from=4-3, to=4-4]
	\arrow[from=4-3, to=5-3]
	\arrow[from=4-4, to=4-5]
	\arrow[from=4-4, to=5-4]
\end{tikzcd}\]
whose rows and columns are exact. Passing to the de Rham spaces on the bottom row and taking the long exact sequence in cohomology induced by the Cartier dual, we obtain the exact sequence
\[0 \to \mathbb{Z}/p\mathbb{Z}\to (E/\mu_p)^\flat\to E^\flat\to 0.\]

The functor $(-)^\flat$ can be applied to the portion of the preceding diagram consisting of commutative connected algebraic groups. This yields the commutative diagram
\[\begin{tikzcd}[ampersand replacement=\&]
	\&\& 0 \& 0 \\
	\& 0 \& {(E/\mu_p)^\flat} \& {E^\flat} \& 0 \\
	0 \& {\mathbb{G}_m^\flat} \& {G^\flat} \& {E^\flat} \& 0 \\
	0 \& {\mathbb{G}_m^\flat} \& {\mathbb{G}_m^\flat} \& 0 \\
	\& 0 \& 0,
	\arrow[from=1-3, to=2-3]
	\arrow[from=1-4, to=2-4]
	\arrow[from=2-2, to=2-3]
	\arrow[from=2-2, to=3-2]
	\arrow[from=2-3, to=2-4]
	\arrow[from=2-3, to=3-3]
	\arrow[from=2-4, to=2-5]
	\arrow[equals, from=2-4, to=3-4]
	\arrow[from=3-1, to=3-2]
	\arrow[from=3-2, to=3-3]
	\arrow[equals, from=3-2, to=4-2]
	\arrow[from=3-3, to=3-4]
	\arrow[from=3-3, to=4-3]
	\arrow[from=3-4, to=3-5]
	\arrow[from=3-4, to=4-4]
	\arrow[from=4-1, to=4-2]
	\arrow[equals, from=4-2, to=4-3]
	\arrow[from=4-2, to=5-2]
	\arrow[from=4-3, to=4-4]
	\arrow[from=4-3, to=5-3]
\end{tikzcd}\]
where the columns are exact and the rows are complexes. The long exact sequence in cohomology induced from a short exact sequence of complexes then identifies the middle cohomology of 
\[0\to \mathbb{G}_m^\flat\to G^\flat\to E^\flat\to 0\]
with $\mathbb{Z}/p\mathbb{Z}$, thereby proving that the functor $(-)^\flat$ is not exact.
\end{example}

\subsection{Linear and generic subspaces of the moduli space}\label{Subsection on generic subspaces}

For an abelian variety $A$, certain \emph{big} subsets of $A^\flat$ play a central role in non-abelian Hodge theory \cite{simpson1993subspaces} and in the formulation of generic vanishing theorems for de Rham cohomology \cite{schnell2015holonomic}. For a torus $T$, similar \emph{big} subsets of $T^\flat$ were also studied in \cite{sabbah1992lieu}. In this section, we extend the definition of these subsets---see Definition~\ref{def generic subspace}---to arbitrary commutative connected algebraic groups.

\begin{definition}[Linear subspace]\label{def linear subspace}
Let $G$ be a commutative connected algebraic group over a characteristic zero field. For an epimorphism $\rho\colon G\twoheadrightarrow \widetilde{G}$ with connected kernel, the image of $\rho^\flat\colon \widetilde{G}^\flat\hookrightarrow G^\flat$ is said to be a \emph{linear subspace} of $G^\flat$.
\end{definition}

The following remark concretely characterizes linear subspaces of $G^\flat$ when $G$ is affine. Furthermore, when $G$ is an abelian variety, this notion is related to existing concepts in the literature.

\begin{remark}\label{examples of linear subvarieties}
Let $G\twoheadrightarrow \widetilde{G}$ be an epimorphism between commutative connected algebraic groups with connected kernel $N$. If $G$ is linear, unipotent, a torus, or an abelian variety, then $\widetilde{G}$ and $N$ inherit the same properties. The following characterizations follow from this observation.

Consider a torus $T$ with character group $X$. A linear subspace of $T^\flat\simeq \Omega_T/X$ is of the form $V/Y$, where $Y$ is a subgroup of $X$ and $V$ is the linear subspace of $\Omega_T$ generated by $Y$, in the sense of linear algebra. For a unipotent group $U$, a linear subspace of $U^\flat = U^*$ corresponds to a linear subspace of the underlying vector space of $U^*$, in the sense of linear algebra.

Let $L$ be a linear commutative connected algebraic group, and decompose it as a product of a torus $T$ and a unipotent group $U$. If $\rho\colon L\twoheadrightarrow \widetilde{L}$ is an epimorphism with connected kernel, then $\widetilde{L}$ is also linear and decomposes as $\widetilde{L}\simeq \widetilde{T}\times \widetilde{U}$. Moreover, $\rho$ is a product of epimorphisms $T\twoheadrightarrow \widetilde{T}$ and $U\twoheadrightarrow \widetilde{U}$. Thus, a linear subspace of $L^\flat$ is the product of linear subspaces of $T^\flat$ and $U^\flat$.

For an abelian variety $A$, linear subspaces of $A^\flat\simeq A^\natural$ were first studied by Simpson, who termed them \emph{triple tori} \cite[p.\ 365]{simpson1993subspaces}. Schnell refers to translates of linear subspaces of $A^\flat$ as \emph{linear subvarieties} \cite[Def.\ 2.3]{schnell2015holonomic}.
\end{remark}

\begin{proposition}\label{inverse image of linear subspaces}
Let $G$ be a commutative connected algebraic group over a characteristic zero field, and express $G$ as an extension of an abelian variety $A$ by a linear group $L$. If $V\subset L^\flat$ is a (translate of a) linear subspace, so is its inverse image by $G^\flat\to L^\flat$.
\end{proposition}

\begin{proof}
	Let $L\twoheadrightarrow \widetilde{L}$ be the epimorphism defining the linear subspace $V$, and let $N$ be its kernel. Denote by $\widetilde{G}$ the quotient of $G$ by $N$; we have a commutative diagram
\[\begin{tikzcd}[ampersand replacement=\&]
	\& 0 \& 0 \& 0 \\
	0 \& N \& L \& {\widetilde{L}} \& 0 \\
	0 \& N \& G \& {\widetilde{G}} \& 0 \\
	{} \& 0 \& A \& A \& 0 \\
	\&\& 0 \& 0,
	\arrow[from=1-2, to=2-2]
	\arrow[from=1-3, to=2-3]
	\arrow[from=1-4, to=2-4]
	\arrow[from=2-1, to=2-2]
	\arrow[from=2-2, to=2-3]
	\arrow[equals, from=2-2, to=3-2]
	\arrow[from=2-3, to=2-4]
	\arrow[from=2-3, to=3-3]
	\arrow[from=2-4, to=2-5]
	\arrow[from=2-4, to=3-4]
	\arrow[from=3-1, to=3-2]
	\arrow[from=3-2, to=3-3]
	\arrow[from=3-2, to=4-2]
	\arrow[from=3-3, to=3-4]
	\arrow[from=3-3, to=4-3]
	\arrow[from=3-4, to=3-5]
	\arrow[from=3-4, to=4-4]
	\arrow[from=4-2, to=4-3]
	\arrow[equals, from=4-3, to=4-4]
	\arrow[from=4-3, to=5-3]
	\arrow[from=4-4, to=4-5]
	\arrow[from=4-4, to=5-4]
\end{tikzcd}\]
	whose rows and columns are exact. Applying the functor $(-)^\flat$, we obtain the following commutative diagram
\[\begin{tikzcd}[ampersand replacement=\&]
	\& 0 \& 0 \\
	0 \& {A^\flat} \& {A^\flat} \& 0 \\
	0 \& {\widetilde{G}^\flat} \& {G^\flat} \& {N^\flat} \& 0 \\
	0 \& {\widetilde{L}^\flat} \& {L^\flat} \& {N^\flat} \& 0 \\
	\& 0 \& 0 \& 0.
	\arrow[from=1-2, to=2-2]
	\arrow[from=1-3, to=2-3]
	\arrow[from=2-1, to=2-2]
	\arrow[equals, from=2-2, to=2-3]
	\arrow[from=2-2, to=3-2]
	\arrow[from=2-3, to=2-4]
	\arrow[from=2-3, to=3-3]
	\arrow[from=2-4, to=3-4]
	\arrow[from=3-1, to=3-2]
	\arrow[from=3-2, to=3-3]
	\arrow[from=3-2, to=4-2]
	\arrow[from=3-3, to=3-4]
	\arrow[from=3-3, to=4-3]
	\arrow[from=3-4, to=3-5]
	\arrow[from=4-1, to=4-2]
	\arrow[from=4-2, to=4-3]
	\arrow[from=4-2, to=5-2]
	\arrow[from=4-3, to=4-4]
	\arrow[from=4-3, to=5-3]
	\arrow[equals, from=4-4, to=3-4]
	\arrow[from=4-4, to=4-5]
	\arrow[from=4-4, to=5-4]
\end{tikzcd}\]
According to Theorem~\ref{exact sequence of flats} and Proposition~\ref{extensions of linear groups}, its columns are exact, as well as the top and bottom rows. Consequently, the middle row is exact as well. A diagram chase shows that the square
\[\begin{tikzcd}[ampersand replacement=\&]
	{\widetilde{G}^\flat} \& {G^\flat} \\
	{\widetilde{L}^\flat} \& {L^\flat}
	\arrow[from=1-1, to=1-2]
	\arrow[from=1-1, to=2-1]
	\arrow[from=1-2, to=2-2]
	\arrow[from=2-1, to=2-2]
\end{tikzcd}\]
is cartesian, thereby concluding the proof.
\end{proof}

\begin{definition}[Generic subspace]\label{def generic subspace}
	Let $G$ be a commutative connected algebraic group over a characteristic zero field. A \emph{generic subspace} of $G^\flat$ is the complement of a finite union of translates of linear subspaces of $G^\flat$ with positive codimension.
\end{definition}

In algebraic geometry, a property is often said to hold generically if it holds on an open dense subset. For extensions of abelian varieties by unipotent groups, this can be related with our definition above.

\begin{proposition}\label{generic implies open and dense}
Let $G$ be a commutative algebraic group over a characteristic zero field that is an extension of an abelian variety by a unipotent group. If $V$ is a generic subspace of $G^\flat$, then $V$ is open and dense in $G^\flat$.
\end{proposition}

\begin{proof}
	Since the intersection of a finite number of open dense subsets is also open and dense, it suffices to prove that the complement of a linear subspace with positive codimension is open and dense. Let $G\twoheadrightarrow \widetilde{G}$ be an epimorphism with connected kernel $N$ defining the linear subspace $\widetilde{G}^\flat$ of $G^\flat$.
	
	First, we claim that $\widetilde{G}$ is also an extension of an abelian variety by a unipotent group. Denote by $U$ the maximal unipotent subgroup of $G$ and by $A$ the quotient $G/U$. The universal property of quotients yields a morphism $N/(U\cap N)\to A$ making the diagram
\[\begin{tikzcd}[ampersand replacement=\&]
	0 \& {U\cap N} \& N \& {N/(U\cap N)} \& 0 \\
	0 \& U \& G \& A \& 0
	\arrow[from=1-1, to=1-2]
	\arrow[from=1-2, to=1-3]
	\arrow[from=1-2, to=2-2]
	\arrow[from=1-3, to=1-4]
	\arrow[from=1-3, to=2-3]
	\arrow[from=1-4, to=1-5]
	\arrow[color=ocre, from=1-4, to=2-4]
	\arrow[from=2-1, to=2-2]
	\arrow[from=2-2, to=2-3]
	\arrow[from=2-3, to=2-4]
	\arrow[from=2-4, to=2-5]
\end{tikzcd}\]
commute. A quick diagram chase shows that $N/(U\cap N)\to A$ is a monomorphism. The usual isomorphism theorems then imply that $(G/N)/(U/(U\cap N))$ is isomorphic to $(G/U)/(N/(U\cap N))$. Denoting by $B$ this common group, we may complete the diagram above to
\[\begin{tikzcd}[ampersand replacement=\&]
	\& 0 \& 0 \& 0 \\
	0 \& {U\cap N} \& N \& {N/(U\cap N)} \& 0 \\
	0 \& U \& G \& A \& 0 \\
	0 \& {U/(U\cap N)} \& {\widetilde{G}} \& B \& 0 \\
	\& 0 \& 0 \& {0.}
	\arrow[from=1-2, to=2-2]
	\arrow[from=1-3, to=2-3]
	\arrow[from=1-4, to=2-4]
	\arrow[from=2-1, to=2-2]
	\arrow[from=2-2, to=2-3]
	\arrow[from=2-2, to=3-2]
	\arrow[from=2-3, to=2-4]
	\arrow[from=2-3, to=3-3]
	\arrow[from=2-4, to=2-5]
	\arrow[from=2-4, to=3-4]
	\arrow[from=3-1, to=3-2]
	\arrow[from=3-2, to=3-3]
	\arrow[from=3-2, to=4-2]
	\arrow[from=3-3, to=3-4]
	\arrow[from=3-3, to=4-3]
	\arrow[from=3-4, to=3-5]
	\arrow[from=3-4, to=4-4]
	\arrow[from=4-1, to=4-2]
	\arrow[from=4-2, to=4-3]
	\arrow[from=4-2, to=5-2]
	\arrow[from=4-3, to=4-4]
	\arrow[from=4-3, to=5-3]
	\arrow[from=4-4, to=4-5]
	\arrow[from=4-4, to=5-4]
\end{tikzcd}\]
Here every column and every row is exact. Since $B$ is a quotient of an abelian variety, it is also an abelian variety. Similarly, $U/(U\cap N)$ is a unipotent group.

Recall from Corollary~\ref{G flat is an algebraic space} that $G^\flat$ and $\widetilde{G}^\flat$ are commutative connected algebraic groups. Since algebraic groups in characteristic zero are smooth, $G^\flat$ is irreducible. According to \cite[Thm.\ 5.34]{M}, the map $\widetilde{G}^\flat\to G^\flat$ is a closed immersion and so its complement is open. Since $\widetilde{G}^\flat$ is assumed to be of positive codimension, its complement is non-empty. The irreducibility of $G^\flat$ then implies that it is dense, thus completing the proof.
\end{proof}

The preceding result does not hold for general commutative connected algebraic groups. The Zariski topology of $\mathbb{G}_m^\flat\simeq \mathbb{A}^1/\mathbb{Z}$ is too coarse. Indeed, the only open dense subset of $\mathbb{G}_m^\flat$ is the entire space itself. This suggests the following definition.

\begin{definition}[Analytic moduli space]
	Let $G$ be a commutative connected algebraic group over $\mathbb{C}$. Taking the $\mathbb{C}$-points of the morphism $X\hookrightarrow K$ defining the moduli space $G^\flat$, we obtain a monomorphism of abelian groups $X\hookrightarrow K(\mathbb{C})$. Equip $X$ with the discrete topology and $K(\mathbb{C})$ with the analytic topology. The quotient $K(\mathbb{C})/X$ is said to be the \emph{analytic moduli space} and is denoted by $G^\flat_\text{an}$.
\end{definition}

Interestingly, although $G^\flat$ is often not representable by a scheme, $G^\flat_\text{an}$ is a complex manifold. Furthermore, for an abelian variety $A$, the complex manifold $A^\flat_\text{an}$ is Stein, even though $A^\flat$ is not affine. The moduli space $A^\flat$ is even \emph{anti-affine}, meaning that every morphism of algebraic varieties $A^\flat\to\mathbb{A}^1$ is constant \cite[Prop.\ 5.5.8]{brion2017some}.

For a generic subspace $V$ of $G^\flat$, we denote by $V_\text{an}$ the set $V(\mathbb{C})$ with the subspace topology inherited from $G^\flat_\text{an}$. An analytic analogue of Proposition~\ref{generic implies open and dense} holds in all generality.

\begin{proposition}\label{generic implies analytic open dense}
Let $G$ be a commutative connected algebraic group over $\mathbb{C}$. If $V$ is a generic subspace of $G^\flat$, then $V_{\normalfont\text{an}}$ is an open subset of $G^\flat_{\normalfont\text{an}}$ whose complement has measure zero. In particular, it is an open dense subset.
\end{proposition}

\begin{proof}
	Let $G\twoheadrightarrow \widetilde{G}$ be an epimorphism with connected kernel $N$ defining a linear subspace $\widetilde{G}^\flat$ of $G^\flat$ with positive codimension. As in the proof of Proposition~\ref{generic implies open and dense}, it suffices to prove that $\widetilde{G}^\flat_\text{an}$ is a closed subset of $G^\flat_\text{an}$ with measure zero.
	
According to the functoriality of $(-)^\flat$, explained in Remark~\ref{functoriality of G flat}, the natural map $\widetilde{G}^\flat_\text{an}\hookrightarrow G^\flat_\text{an}$ is induced by a morphism of schemes $\widetilde{K}\to K$. Consequently, $\widetilde{G}^\flat_\text{an}$ is a complex Lie subgroup of $G^\flat_\text{an}$. It follows that $\widetilde{G}^\flat_\text{an}\hookrightarrow G^\flat_\text{an}$ is a closed immersion \cite[Chap.\ III, \S 1.3.5]{bourbaki2006groupes}. Since we assumed that $\dim \widetilde{G}^\flat < \dim G^\flat$, Sard's theorem implies that $\widetilde{G}^\flat_\text{an}$ has measure zero in $G^\flat_\text{an}$.
\end{proof}

\begin{definition}[Character variety]
	Let $G$ be a commutative connected algebraic group over $\mathbb{C}$. The \emph{character variety} of $G$, denoted $\operatorname{Char}(G)$, is the spectrum of the group algebra $\mathbb{C}[\pi_1(G_\text{an})]$.
\end{definition}

Recall that the group algebra functor $\textsf{Grp}\to \textsf{Alg}(\mathbb{C})$ is left adjoint to the functor sending a $\mathbb{C}$-algebra $R$ to its group of units $R^\times$. Thus, the $\mathbb{C}$-points of $\operatorname{Char}(G)$ correspond to characters $\pi_1(G_\text{an})\to\mathbb{C}^\times$ of the fundamental group $\pi_1(G_\text{an})$. In other words, $\operatorname{Char}(G)$ is a moduli space for rank one local systems on $G_\text{an}$. The Hopf algebra structure in $\mathbb{C}[\pi_1(G_\text{an})]$ induces a structure of algebraic group in $\operatorname{Char}(G)$, encoding the tensor product of local systems.

The Riemann--Hilbert equivalence states that local systems on $G_\text{an}$ correspond to \emph{regular} connections on $G$. Since all non-trivial character sheaves on unipotent groups are irregular, $\operatorname{Char}(G)$ cannot parametrize character sheaves in general. On the other hand, character sheaves on semiabelian varieties are regular, allowing such a comparison.

\begin{proposition}\label{character sheaves and character variety}
Let $G$ be a semiabelian variety over $\mathbb{C}$. Then the complex Lie groups $G^\flat_{\normalfont\text{an}}$ and $\operatorname{Char}(G)_{\normalfont\text{an}}$ are isomorphic.
\end{proposition}

\begin{proof}
Write $G$ as an extension of an abelian variety $A$ by a torus $T$. By analytification, we obtain a short exact sequence
\[0\to T_\text{an}\to G_\text{an}\to A_\text{an}\to 0\]
of complex Lie groups. In particular, $G_\text{an}\to A_\text{an}$ is a fibration with fiber $T_\text{an}$. Considering the associated long exact sequence in homotopy, we have
\[0=\pi_2(A_\text{an})\to \pi_1(T_\text{an})\to \pi_1(G_\text{an})\to \pi_1(A_\text{an})\to \pi_0(T_\text{an})=0.\] 
Here, $\pi_0(T_\text{an})$ vanishes since $T_\text{an}$ is connected, and $\pi_2(A_\text{an})$ vanishes since $A_\text{an}$ is a Lie group. This induces a short exact sequence of Hopf algebras
\[\mathbb{C}\to \mathbb{C}[\pi_1(T_\text{an})]\to \mathbb{C}[\pi_1(G_\text{an})]\to \mathbb{C}[\pi_1(A_\text{an})]\to \mathbb{C},\]
followed by a short exact sequence of affine algebraic groups
\[0\to \operatorname{Char}(A)\to \operatorname{Char}(G)\to \operatorname{Char}(T)\to 0.\]

There is a natural holomorphic morphism $G^\flat_\text{an}\to \operatorname{Char}(G)_\text{an}$ sending a character sheaf $(\mathscr{L},\nabla)$ to the local system $\ker \nabla_\text{an}$. This is compatible with inverse images and tensor products, giving rise to the following commutative diagram of complex Lie groups.
\[\begin{tikzcd}[ampersand replacement=\&]
	0 \& {A^\flat_\text{an}} \& {G^\flat_\text{an}} \& {T^\flat_\text{an}} \& 0 \\
	0 \& {\operatorname{Char}(A)_\text{an}} \& {\operatorname{Char}(G)_\text{an}} \& {\operatorname{Char}(T)_\text{an}} \& 0
	\arrow[from=1-1, to=1-2]
	\arrow[from=1-2, to=1-3]
	\arrow[from=1-2, to=2-2]
	\arrow[from=1-3, to=1-4]
	\arrow[from=1-3, to=2-3]
	\arrow[from=1-4, to=1-5]
	\arrow[from=1-4, to=2-4]
	\arrow[from=2-1, to=2-2]
	\arrow[from=2-2, to=2-3]
	\arrow[from=2-3, to=2-4]
	\arrow[from=2-4, to=2-5]
\end{tikzcd}\]
Since character sheaves on $A$ and $T$ are regular, the Riemann--Hilbert correspondence implies that the morphisms $A^\flat_\text{an}\to \operatorname{Char}(A)_\text{an}$ and $T^\flat_\text{an}\to \operatorname{Char}(T)_\text{an}$ are injective. The fact that a line bundle with integrable connection on $A$ is automatically a character sheaf further says that $A^\flat_\text{an}\to \operatorname{Char}(A)_\text{an}$ is surjective.

To prove that $T^\flat_\text{an}\to \operatorname{Char}(T)_\text{an}$ is surjective, we may suppose that $T=\mathbb{G}_m$. A line bundle with integrable connection on $\mathbb{G}_m$ is of the form $(\mathcal{O}_{\mathbb{G}_m},\mathrm{d}+f(x)\,\mathrm{d}x/x)$, for some $f\in k[x,x^{-1}]$. This connection is regular precisely if $f$ is constant. In other words, regular line bundles with integrable connection on $\mathbb{G}_m$ are the same as character sheaves. This proves that $(\mathbb{G}_m)^\flat_\text{an}\to \operatorname{Char}(\mathbb{G}_m)_\text{an}$ is surjective.
\end{proof}

\begin{corollary}
	Let $G$ be a semiabelian variety over $\mathbb{C}$. Then every character sheaf on $G$ is regular.
\end{corollary}

\appendix
\section{Connections and de Rham spaces}
As it was first observed by Simpson \cite{Simpson}, given a choice of cohomology theory $\mathrm{H}$ and a "space" $X$, there is often a stack $X_\mathrm{H}$ whose category of quasi-coherent sheaves coincides with the category of coefficients for $\mathrm{H}$. Moreover, the association $X\mapsto X_\mathrm{H}$ preserves the functoriality of the given cohomology theory.

In this appendix, we study the de Rham side of this story. Namely, given a variety $X$, the \emph{de Rham space} $X_\text{dR}$ has the marvellous property that $\mathbb{G}_m$-torsors over it are the same as line bundles on $X$ \emph{endowed with a flat connection}. Moreover, formal completions can also be understood in function of the de Rham spaces.

The author claims no originality for any result in this appendix: all the results in it are either available in the literature or are folklore. (See \cite{gaitsgory2017study} and \cite{hennion2020higher} for more on this.) However, even the results that have published proofs are usually studied in the context of derived algebraic geometry, so we thought that this appendix could be helpful to some readers.

\subsection{Basic properties of the de Rham space}

Let $k$ be a field and consider the category $\textsf{Aff}/k$ of affine schemes over $k$. In order to simplify notation, we will often denote an object $\operatorname{Spec}R$ of $\textsf{Aff}/k$ as $R$.

\begin{definition}[de Rham space]\label{def de rham space}
Given a presheaf of sets $X$ on $\textsf{Aff}/k$, its \emph{de Rham space} $X_\text{dR}$ is the presheaf defined by
\[X_\text{dR}(R)\colonequals \operatornamewithlimits{colim}_{I\subset R}X(R/I),\]
where the colimit runs over the filtered poset of nilpotent ideals of $R$. This presheaf comes equipped with a morphism $X\to X_\text{dR}$ induced by the trivial ideal $I=0$.
\end{definition}

This assignment is functorial: given a morphism of presheaves $f\colon X\to Y$, there is an induced map $f_\text{dR}\colon X_\text{dR}\to Y_\text{dR}$, making the diagram
\[\begin{tikzcd}
	X & {X_\text{dR}} \\
	Y & {Y_{\text{dR}}}
	\arrow[from=1-1, to=1-2]
	\arrow["f"', from=1-1, to=2-1]
	\arrow["{f_\text{dR}}", from=1-2, to=2-2]
	\arrow[from=2-1, to=2-2]
\end{tikzcd}\]
commute. As it will be formalized in Corollary~\ref{de rham space is a coequalizer}, the geometric interpretation of $X_\text{dR}$, at least for smooth $k$-schemes $X$, is that it is a quotient of $X$ where we identify infinitesimally close points.

We begin our study of the de Rham space with the following simple observation, which will prove useful later.

\begin{proposition}\label{de rham functor preserves limits on presheaves}
The functor $(-)_{\normalfont\text{dR}}\colon \normalfont\textsf{PSh}(\textsf{Aff}/k)\to \textsf{PSh}(\textsf{Aff}/k)$ preserves finite limits and arbitrary colimits.
\end{proposition}

\begin{proof}
Since limits and colimits of presheaves are computed pointwise, this follows from the fact that filtered colimits in the category of sets commute with arbitrary colimits and finite limits.
\end{proof}

For a finite type $k$-algebra $R$ (or more generally, a noetherian $k$-algebra), the nilradical $\operatorname{Nil}(R)$ is nilpotent as it is generated by finitely many nilpotent elements. Consequently, $X_\text{dR}(R) \simeq X(R/\operatorname{Nil}(R)) = X(R_\text{red})$. This property holds for any $k$-algebra provided that $X$ is a locally of finite type $k$-scheme.

\begin{proposition}\label{Xdr(R)=X(Rred)}
Let $X$ be a locally of finite type scheme over $k$. Then $X_{\normalfont\text{dR}}(R)\simeq X(R_{\normalfont\text{red}})$ for every $k$-algebra $R$.
\end{proposition}

\begin{proof}
	Define $S\colonequals \operatorname{colim}_{I\subset R}R/I$, where the colimit runs through the nilpotent ideals of $R$. As usual, elements of $S$ are denoted as equivalence classes of the form $[I, x]$, where $I$ is a nilpotent ideal in $R$ and $x$ is an element of $R$. Here, $[I, x]=[I^\prime,x^\prime]$ if there exists a nilpotent ideal $J$ containing $I$ and $I^\prime$ such that $x\equiv x^\prime\bmod{J}$.
	
	The natural map $R\to S$, corresponding to the ideal $I=0$, sends every nilpotent in $R$ to zero. In other words, it factors through the nilradical yielding a map $R_\text{red}\to S$. We affirm that this morphism is injective. Indeed, $[0, x]=0$ means that there exists a nilpotent ideal $J$ containing $x$. It follows that $x$ is nilpotent and so vanishes on $R_\text{red}$. Since $[I, x]\in S$ is the image of $x\in R$, we have that $R_\text{red}\to S$ is an isomorphism.
	
	As $X$ is locally of finite type, \cite[Tag \href{https://stacks.math.columbia.edu/tag/01ZC}{01ZC}]{Stacks} gives $X_\text{dR}(R)\simeq X(S)\simeq X(R_\text{red})$, concluding the proof.
\end{proof}

Perhaps not surprisingly, given the aforementioned geometric interpretation of $X_\text{dR}$, formal completions of schemes can be written in terms of de Rham spaces.

\begin{proposition}\label{formal completions as de rham spaces}
Let $X$ be a $k$-scheme and let $Z$ be a closed subscheme of $X$. The formal completion $\widehat{X}_Z$ of $X$ along $Z$ is isomorphic to $X \times_{X_{\normalfont\text{dR}}}Z_{\normalfont\text{dR}}$.
\end{proposition}

\begin{proof}
	Let $\mathcal{I}\subset\mathcal{O}_X$ be the ideal sheaf defined by $Z$, and let $R$ be a $k$-algebra. We aim to obtain a functorial isomorphism
	\[\operatorname*{colim}_{I\subset R} X(R) \times_{X(R/I)}Z(R/I)\simeq \operatorname*{colim}_{n\geq 0} \operatorname{Spec}_X(\mathcal{O}_X/\mathcal{I}^{n+1})(R),\]
	where the colimit on the left runs through the nilpotent ideals of $R$. For an ideal $I\subset R$, let $i_I$ denote the closed immersion $\operatorname{Spec}R/I\to \operatorname{Spec}R$. Then, we have that
	\begin{align*}
	\operatorname*{colim}_{I\subset R} X(R)\times_{X(R/I)}Z(R/I) &\simeq\operatorname*{colim}_{I\subset R}\:\{x\in X(R)\:|\: i_I^* x^*\mathcal{I}=0\} \\ 
	&\simeq\operatorname*{colim}_{n\geq 0}\operatorname*{colim}_{I^{n+1}=0}\:\{x\in X(R)\:|\: i_I^* x^*\mathcal{I}=0\}\\ 
	&\simeq\operatorname*{colim}_{n\geq 0}\:\{x\in X(R)\:|\: x^*\mathcal{I}^{n+1}=0\}\simeq\widehat{X}_Z(R), 
	\end{align*}
	where the last isomorphism is the universal property of the relative spectrum.
\end{proof}

Note that the expression $X \times_{X_\text{dR}}Z_\text{dR}$ is meaningful even if $Z\to X$ is not a closed immersion. In such cases, it can be used to \emph{define} formal completions. Moreover, this characterization makes it clear that the projection $\widehat{X}_Z \to X$ is a monomorphism whenever $Z\to X$ is.

\begin{corollary}\label{inclusion of formal completion}
If $Z$ is a closed subscheme of a $k$-scheme $X$, the projection $\widehat{X}_Z \to X$ is a monomorphism of presheaves.
\end{corollary}

\begin{proof}
According to \cite[Tag \href{https://stacks.math.columbia.edu/tag/01L7}{01L7}]{Stacks}, the closed immersion $i\colon Z\to X$ is a monomorphism in the category of schemes. Now, the Yoneda embedding preserves limits and any functor that preserves limits preserves monomorphisms \cite[Tag \href{https://stacks.math.columbia.edu/tag/01L3}{01L3}]{Stacks}. In other words, $i$ is a monomorphism of presheaves. Since the de Rham functor preserves finite limits, so is $Z_\text{dR}\to X_\text{dR}$. Finally, fibered products preserve monomorphisms and this finishes the proof.
\end{proof}

For a morphism of $k$-schemes $p\colon X\to S$, the universal property of fibered products induces a map $\iota_p\colon X\to X_\text{dR}\times_{S_\text{dR}}S$. As the following proposition shows, this map faithfully encodes the differential information contained in $p$.\footnote{This result is somewhat akin to the slogan "$f$ is smooth (resp.\ unramified) if and only if $\mathrm{d}f$ is surjective (resp.\ injective)".}

\begin{proposition}\label{de rham is an epimorphism of presheaves}
Let $p\colon X\to S$ be a morphism of $k$-schemes. Then $p$ is formally smooth (resp.\ formally unramified) if and only if $\iota_p\colon X\to X_{\normalfont\text{dR}}\times_{S_{\normalfont\text{dR}}}S$ is an epimorphism (resp.\ monomorphism) of presheaves.
\end{proposition}

\begin{proof}
Recall that $p$ is formally smooth (resp.\ formally unramified) if, for every $k$-algebra $R$ with a map $\operatorname{Spec}R\to S$ and for every nilpotent ideal $I\subset R$, the induced map
	\[\operatorname{Hom}_S(\operatorname{Spec}R,X)\to \operatorname{Hom}_S(\operatorname{Spec}R/I,X)\]
	is surjective (resp.\ injective). We will translate this condition into the surjectivity (resp. injectivity) of $X(R)\to (X_\text{dR}\times_{S_\text{dR}}S)(R)$.
	
	Suppose that $p$ is formally smooth, let $R$ be any $k$-algebra, and let $[x,s]$ be an element of
	\[(X_\text{dR}\times_{S_\text{dR}} S)(R)\simeq \operatorname{colim}_{I\subset R}X(R/I)\times_{S(R/I)}S(R).\]
	That is, there exists a nilpotent ideal $I\subset R$ such that $x\in X(R/I)$ and $s\in S(R)$ agree on $S(R/I)$. By formal smoothness, we have a morphism $\overline{x}\colon \operatorname{Spec}R\to X$ making the diagram
\[\begin{tikzcd}
	{\operatorname{Spec}R/I} & X \\
	{\operatorname{Spec}R} & S
	\arrow["p", from=1-2, to=2-2]
	\arrow[from=1-1, to=2-1]
	\arrow["s", from=2-1, to=2-2]
	\arrow["x", from=1-1, to=1-2]
	\arrow["{\overline{x}}", color=ocre, from=2-1, to=1-2]
\end{tikzcd}\]
commute. (The lower triangle commutes since $\overline{x}$ is a morphism over $S$, and the upper triangle commutes because $\overline{x}$ maps to $x$.) This is an element of $X(R)$ mapping to $(X_\text{dR}\times_{S_\text{dR}} S)(R)$. Conversely, suppose that $R$ is a $k$-algebra with a map $s\colon\operatorname{Spec}R\to S$, $I\subset R$ is a nilpotent ideal, and $x$ is an element of $\operatorname{Hom}_S(\operatorname{Spec}R/I,X)$. This data defines an element of $(X_\text{dR}\times_{S_\text{dR}} S)(R)$ and so there exists $\overline{x}\in X(R)$ mapping to it. In other words, $\iota_p\colon X\to X_\text{dR}\times_{S_\text{dR}} S$ is an epimorphism if and only if $p$ is formally smooth.

Now, if $p$ is formally unramified, consider a $k$-algebra $R$ and let $x,y\in X(R)$ be two elements whose images in $X_\text{dR}\times_{S_\text{dR}} S(R)$ coincide. That is, there exists a nilpotent ideal $I\subset R$ such that the diagram
\[\operatorname{Spec}R/I\to \operatorname{Spec}R \rightrightarrows X\xrightarrow{\ p\ } S\]
commutes. In particular, $x$ and $y$ define elements of $\operatorname{Hom}_S(\operatorname{Spec}R,X)$ that coincide on $\operatorname{Hom}_S(\operatorname{Spec}R/I,X)$. Since $p$ is formally unramified, we have that $x=y$.

Suppose that $\iota_p\colon X\to X_\text{dR}\times_{S_\text{dR}} S$ is a monomorphism and let $x,y\in \operatorname{Hom}_S(\operatorname{Spec}R,X)$ be two morphisms that coincide on $\operatorname{Hom}_S(\operatorname{Spec}R/I,X)$, for some $k$-algebra $R$ with a map $\operatorname{Spec}R\to S$ and a nilpotent ideal $I\subset R$. In particular, $x$ and $y$ are elements of $X(R)$ that coincide on $(X_\text{dR}\times_{S_\text{dR}} S)(R)$. It follows that $x=y$ and so $p$ is formally unramified.
\end{proof}

Let $Y\to X$ be an immersion of $k$-schemes that factors as $Y\to U\to X$, where $Y\to U$ is a closed immersion with ideal $\mathscr{I}$ and $U\to X$ is an open immersion. The formal completion of $X$ along $Y$ is usually defined as the colimit of $\operatorname{Spec}_U(\mathcal{O}_U/\mathcal{I}^{n+1})$, for $n\geq 0$. The previous proposition implies that $U\simeq U_\text{dR}\times_{X_\text{dR}}X$, and so
\[U\times_{U_\text{dR}}Z_\text{dR}\simeq X\times_{X_\text{dR}}U_\text{dR}\times_{U_\text{dR}}Z_\text{dR}\simeq X\times_{X_\text{dR}}Z_\text{dR},\]
proving that Proposition~\ref{formal completions as de rham spaces} also applies to locally closed immersions.

\begin{corollary}\label{de rham space is a coequalizer}
Let $X\to S$ be a formally smooth morphism of $k$-schemes. Then $X_{\normalfont\text{dR}} \times_{S_{\normalfont\text{dR}}}S$ is the coequalizer of
\[\widehat{(X\times_S X)}_{\Delta}\rightrightarrows X,\]
where $\widehat{(X\times_S X)}_{\Delta}$ is the formal completion of $X\times_S X$ along the diagonal.
\end{corollary}

\begin{proof}
In order to simplify notation, let $Y=X_\text{dR}\times_{S_\text{dR}} S$. Since every epimorphism is effective in a topos, $X\to Y$ is the coequalizer of $X\times_{Y}X\rightrightarrows X$. Now, the result follows from general category theory: the pullback of
\[\begin{tikzcd}
	& {X\times_S X} \\
	{X_\text{dR}} & {(X\times_S X)_\text{dR}}
	\arrow["{\Delta_\text{dR}}", from=2-1, to=2-2]
	\arrow[from=1-2, to=2-2]
\end{tikzcd}\]
is $X\times_{Y}X$.
\end{proof}

Recall that the functor of points of a scheme $X$ is a sheaf for the étale and fppf topologies. We will now study the descent properties of $X_\text{dR}$, and we begin with a lemma.

\begin{lemma}\label{lemma étale topology}
Let $R$ be a $k$-algebra and let $\{R\to R_i\}_{i\in I}$ be an étale covering. Then the reduction $\{R_{\normalfont\text{red}}\to R_{i,\normalfont\text{red}}\}_{i\in I}$ is also an étale covering. Moreover, any étale covering of $R_{\normalfont\text{red}}$ arises in this way.
\end{lemma}

\begin{proof}
Since $R\to R_i$ is étale, so is its base change $R_\text{red}\to R_\text{red}\otimes_R R_i$. By \cite[Tag \href{https://stacks.math.columbia.edu/tag/033B}{033B}]{Stacks}, we have that $R_\text{red}\otimes_R R_i$ is reduced and then \cite[Cor.\ 5.1.8]{EGAI} gives that
\[R_\text{red}\otimes_R R_i=(R_\text{red}\otimes_R R_i)_\text{red}\simeq (R_\text{red}\otimes_{R_\text{red}} R_{i,\text{red}})_\text{red}\simeq R_{i,\text{red}}.\]
It follows that $\{R_{\text{red}}\to R_{i,\text{red}}\}_{i\in I}$ is an étale covering of $R_\text{red}$.

Now, consider an étale covering $\{R_\text{red}\to S_i\}_{i\in I}$ of $R_\text{red}$. By the topological invariance of the étale site, there exists a covering $\{R\to R_i\}_{i\in I}$ along with isomorphisms $R_\text{red}\otimes_R R_i\simeq S_i$ for all $i\in I$ \cite[Tag \href{https://stacks.math.columbia.edu/tag/04DZ}{04DZ}]{Stacks}. The same argument as above shows that $S_i$ is reduced, and then $S_i\simeq R_{i,\text{red}}$.
\end{proof}

\begin{proposition}\label{etale sheaf on aff}
Let $X$ be a locally of finite type scheme over $k$. The de Rham space $X_{\normalfont\text{dR}}$ is an étale sheaf on $\normalfont\textsf{Aff}/k$.
\end{proposition}

\begin{proof}
Let $R$ be a $k$-algebra and let $\{R\to R_i\}_{i\in I}$ be an étale covering of $R$. We want to prove that the diagram
\[X(R_\text{red})\to \prod_i X(R_{i,\text{red}})\rightrightarrows \prod_{i,j}X((R_i\otimes_R R_j)_\text{red}) \]
is an equalizer. The lemma above says that $\{R_\text{red}\to R_{i,\text{red}}\}_{i\in I}$ is also an étale cover and then the fact that $X$ is an étale sheaf implies that the diagram
\[X(R_\text{red})\to \prod_i X(R_{i,\text{red}})\rightrightarrows \prod_{i,j}X(R_{i,\text{red}}\otimes_{R_\text{red}} R_{j,\text{red}})\]
is an equalizer. The same argument as in the proof of the previous lemma shows that $R_{i,\text{red}}\otimes_{R_\text{red}} R_{j,\text{red}}$ is reduced. Then \cite[Cor.\ 5.1.8]{EGAI} gives isomorphisms $R_{i,\text{red}}\otimes_{R_\text{red}} R_{j,\text{red}}\simeq (R_i\otimes_R R_j)_\text{red}$, finishing the proof.
\end{proof}

In particular, the preceding proposition implies that the de Rham space of a commutative algebraic group over $k$ is an abelian étale sheaf on $\textsf{Aff}/k$.

\begin{proposition}\label{de rham functor is exact}
Suppose that $k$ has characteristic zero. Then the functor $(-)_{\normalfont\text{dR}}$ from commutative algebraic groups over $k$ to abelian étale sheaves on $\normalfont\textsf{Aff}/k$ is exact.
\end{proposition}

\begin{proof}
Let $0\to A\to B \to C\to 0$ be an exact sequence of commutative algebraic groups over $k$. In particular, it is left-exact in the category of abelian \emph{presheaves} on $\textsf{Aff}/k$. By Proposition~\ref{de rham functor preserves limits on presheaves}, the induced exact sequence $0\to A_\text{dR}\to B_\text{dR}\to C_\text{dR}\to 0$ is also left-exact in abelian presheaves. Since sheafification is exact, this sequence is left-exact in the category of abelian étale sheaves.

Let us verify that $B_\text{dR}\to C_\text{dR}$ is an epimorphism of abelian sheaves. Given a $k$-algebra $R$ and an element $c\in C_\text{dR}(R)=C(R_\text{red})$, the fact that $B\to C$ is an epimorphism of étale sheaves implies that there exists a covering $\{R_\text{red}\to S_i\}_{i\in I}$ such that $c|_{S_i}$ is in the image of $B(S_i)\to C(S_i)$ for all $i\in I$ \cite[Tag \href{https://stacks.math.columbia.edu/tag/00WN}{00WN}]{Stacks}. Lemma~\ref{lemma étale topology} then gives a covering $\{R\to R_i\}_{i\in I}$ whose reduction is $\{R_\text{red}\to S_i\}_{i\in I}$. It follows that $c|_{R_i}=c|_{S_i}$ is in the image of $B_\text{dR}(R_i)\to C_\text{dR}(R_i)$ for all $i\in I$, concluding the proof.
\end{proof}

We proved in Corollary~\ref{de rham space is a coequalizer} that $X_\text{dR}$ is a quotient of $X$ in which we identify infinitesimally close points. When $X$ is a commutative algebraic group $G$, the difference of two such points has to live in an infinitesimal neighborhood of the identity. This heuristic leads to the result below.

\begin{proposition}\label{G_dr = G/G^ et}
Let $G$ be a smooth commutative algebraic group over $k$. Then $G_{\normalfont\text{dR}}$ is isomorphic to the presheaf quotient $G/\widehat{G}$, where $\widehat{G}$ is the formal completion of $G$ along the identity. In particular, $G_{\normalfont\text{dR}}$ is also isomorphic to the sheaf quotient $G/\widehat{G}$.
\end{proposition}

\begin{proof}
In this proof, let us consider every (co)limit to be taken inside the category of abelian presheaves on $\textsf{Aff}/k$. As the cokernel of the identity section $e\colon \operatorname{Spec}k\to G$ is $G$ itself, a variant of Proposition~\ref{de rham functor preserves limits on presheaves} for abelian presheaves shows that the cokernel of $e_\text{dR}\colon \operatorname{Spec}k\to G_\text{dR}$ is $G_\text{dR}$. The universal property of cokernels then induces the dashed map below.
\[\begin{tikzcd}
	{\widehat{G}} & G & {G/\widehat{G}} \\
	{\operatorname{Spec}k} & {G_\text{dR}} & {G_\text{dR}}
	\arrow["{e_\text{dR}}", from=2-1, to=2-2]
	\arrow[from=1-1, to=2-1]
	\arrow[from=1-1, to=1-2]
	\arrow[from=1-2, to=2-2]
	\arrow[equals, from=2-2, to=2-3]
	\arrow[from=1-2, to=1-3]
	\arrow[dashed, color=ocre, from=1-3, to=2-3]
\end{tikzcd}\]
The square on the left is cartesian due to Proposition~\ref{formal completions as de rham spaces}, and $G\to G_\text{dR}$ is an epimorphism since $G$ is smooth. Then, \cite[Tag \href{https://stacks.math.columbia.edu/tag/08N4}{08N4}]{Stacks} implies that the square on the left is also cocartesian, and \cite[Tag \href{https://stacks.math.columbia.edu/tag/08N3}{08N3}]{Stacks} gives that $G/\widehat{G}\to G_\text{dR}$ is an isomorphism. Since $G_\text{dR}$ is already an étale sheaf, the presheaf and the sheaf quotients $G/\widehat{G}$ coincide.
\end{proof}

We now study the descent properties of de Rham spaces with respect to the finer fppf topology. We remark that Proposition~\ref{de rham functor is exact} as well as the following proposition and its corollary, are the unique results in this section that need the base field $k$ to have characteristic zero. 

\begin{proposition}\label{G_dr = G/G^ fppf}
	Let $G$ be a commutative algebraic group over a characteristic zero field $k$. Then $G_{\normalfont\text{dR}}$ is an fppf sheaf isomorphic to $G/\widehat{G}$ and the functor $(-)_{\normalfont\text{dR}}$ from commutative algebraic groups over $k$ to abelian fppf sheaves is exact.
\end{proposition}

\begin{proof}
	According to Proposition~\ref{Cartier lemma}, the formal completion $\widehat{G}$ is a direct sum of copies of $\widehat{\mathbb{G}}_a$. Then, given a $k$-algebra $R$, \cite[Rem.\ 2.2.18]{bhatt2022prismatic} says that $\mathrm{H}^1_\text{fppf}(R,\widehat{\mathbb{G}}_a)=0$ and so $(G/\widehat{G})(R)\simeq G(R)/\widehat{G}(R)\simeq G_\text{dR}(R)$, where the quotient on the left is taken on the fppf topology. In other words $G_\text{dR}$ is an fppf sheaf isomorphic to $G/\widehat{G}$. The exactness of $(-)_\text{dR}$ here is a particular case of Proposition~\ref{de rham functor is exact}.
\end{proof}

\begin{corollary}
	Let $X$ be a locally of finite type scheme over a characteristic zero field $k$. Then $X_{\normalfont\text{dR}}$ is an fppf sheaf.
\end{corollary}

\begin{proof}
	Let $R$ be a $k$-algebra and let $\{R\to R_i\}_{i\in I}$ be an fppf covering of $R$. By the previous proposition, the diagram
\[R_\text{red}\to \prod_i R_{i,\text{red}}\rightrightarrows \prod_{i,j}(R_i\otimes_R R_j)_\text{red}\]
is an equalizer in the category of $k$-algebras. The functor of points $X(-)\colon (\textsf{Aff}/k)^\text{op}\to \textsf{Set}$ sends this diagram to an equalizer in the category of sets, finishing the proof.
\end{proof}

\begin{remark}[de Rham spaces in positive characteristic]\label{de rham positive characteristic}
	Let $k$ be a field of characteristic $p>0$. Given a $k$-algebra $R$, the colimit $R_\text{perf}$ of the tower
	\[R\xrightarrow{\ x\mapsto x^p\ }R\xrightarrow{\ x\mapsto x^p\ }R\xrightarrow{\ x\mapsto x^p\ }\cdots\]
	is the so-called \emph{colimit perfection} of $R$. It is always a perfect $k$-algebra and the natural map $R\to R_\text{perf}$ is universal among morphisms from $R$ to a perfect algebra. We define a presheaf $\mathbb{G}_{a,\text{perf}}$ on $\textsf{Aff}/k$ by $\mathbb{G}_{a,\text{perf}}(R) \colonequals R_\text{perf}$. As \cite[Rem.\ 2.2.18]{bhatt2022prismatic} shows, we have an exact sequence of abelian fppf sheaves
	\[0\to \widehat{\mathbb{G}}_a\to \mathbb{G}_a\to \mathbb{G}_{a,\text{perf}}\to 0.\]
	It follows that the natural map of abelian étale sheaves $\mathbb{G}_{a,\text{dR}}\to \mathbb{G}_{a,\text{perf}}$ identifies $\mathbb{G}_{a,\text{perf}}$ with the fppf sheafification of $\mathbb{G}_{a,\text{dR}}$.
\end{remark}

We end this section by extending the definition of de Rham spaces from functors on $\textsf{Aff}/k$ to functors on $\textsf{Sch}/k$.

\begin{definition}
	Let $X$ be a scheme over $k$, seen as its functor of points $(\textsf{Sch}/k)^\text{op}\to \textsf{Set}$. We define its \emph{de Rham space} by taking the de Rham space of the restriction $(\textsf{Aff}/k)^\text{op}\to (\textsf{Sch}/k)^\text{op}\to \textsf{Set}$ and then right Kan extending to $(\textsf{Sch}/k)^\text{op}$.
\end{definition}

This definition actually coincides with the naive one when $X$ is locally of finite type over $k$, but it will be more convenient. Indeed, given a $k$-scheme $S$, Proposition~\ref{Xdr(R)=X(Rred)} implies that
\[X_\text{dR}(S) \simeq \operatorname*{lim}_{\operatorname{Spec}R\to S} X(\operatorname{Spec}R_\text{red}),\]
where the limit runs through the affine $k$-schemes with a map to $S$. Since the functor $X(-)=\operatorname{Mor}_k(-,X)$ commutes with limits, this is also
\[X\left(\operatorname*{colim}_{\operatorname{Spec}R\to S}\operatorname{Spec} R_\text{red}\right)\simeq X(S_\text{red}).\]

The usefulness of this definition comes from the fact that we have an equivalence of topoi $\textsf{Sh}((\textsf{Aff}/k)_\text{ét})\simeq \textsf{Sh}((\textsf{Sch}/k)_\text{ét})$. Here, the functor $\textsf{Sh}((\textsf{Sch}/k)_\text{ét})\to \textsf{Sh}((\textsf{Aff}/k)_\text{ét})$ is given by restriction and the functor $\textsf{Sh}((\textsf{Aff}/k)_\text{ét})\to \textsf{Sh}((\textsf{Sch}/k)_\text{ét})$ is a right Kan extension \cite[Tag \href{https://stacks.math.columbia.edu/tag/021E}{021E}]{Stacks}.

Every result in this section that was true for sheaves generalizes to this setting. Take Proposition~\ref{G_dr = G/G^ et} as an example: it is no longer true that $G_\text{dR}\simeq G/\widehat{G}$ as presheaves on $\textsf{Sch}/k$, but $G_\text{dR}$ is isomorphic to the quotient $G/\widehat{G}$ taken in $\textsf{Ab}((\textsf{Sch}/k)_\text{ét})$. For the sake of completeness, we give precise statements below.

\begin{proposition}
	Let $Z$ be a closed subscheme of a $k$-scheme $X$. Then the formal completion $\widehat{X}_Z$ of $X$ along $Z$ is isomorphic to $X\times_{X_{\normalfont\text{dR}}}Z_{\normalfont\text{dR}}$ and the projection $\widehat{X}_Z\to X$ is a monomorphism of étale sheaves on $\normalfont\textsf{Sch}/k$.
\end{proposition}

\begin{proposition}\label{epimorphism of etale sheaves}
	Let $f\colon X\to S$ be a morphism of $k$-schemes. Then $f$ is formally unramified if and only if $X\to X_{\normalfont\text{dR}}\times_{S_{\normalfont\text{dR}}}S$ is a monomorphism of étale sheaves on $\normalfont\textsf{Sch}/k$. Moreover, if $f$ is formally smooth, then $X\to X_{\normalfont\text{dR}}\times_{S_{\normalfont\text{dR}}}S$ is an epimorphism of étale sheaves on $\normalfont\textsf{Sch}/k$.
\end{proposition}

\begin{proposition}\label{GdR is reasonable}
Let $G$ be a commutative algebraic group over $k$. Then $G_{\normalfont\text{dR}}$ is isomorphic to the quotient $G/\widehat{G}$ taken in $\normalfont\textsf{Ab}((\textsf{Sch}/k)_{\normalfont\text{ét}})$. Moreover, the functor $(-)_{\normalfont\text{dR}}$ from commutative algebraic groups over $k$ to abelian étale sheaves on $\normalfont\textsf{Sch}/k$ is exact.
\end{proposition}

Due to \cite[Tag \href{https://stacks.math.columbia.edu/tag/021V}{021V}]{Stacks}, all results in the last 3 propositions also hold for the fppf topology as long as $k$ has characteristic zero. 

\subsection{Connections as torsors on the de Rham space}

Let $k$ be a field of characteristic zero, $S$ be a $k$-scheme, and $X$ be a smooth $S$-scheme. For a vector bundle $V$ on $X$, recall that a connection on $V$ relative to $S$ is an $\mathcal{O}_S$-linear map $\nabla\colon V\to \Omega^1_{X/S}\otimes_{\mathcal{O}_X}V$ satisfying the Leibniz rule
\[\nabla(fx)=\mathrm{d}f\otimes x+f\nabla(x),\]
for local sections $f$ of $\mathcal{O}_X$ and $x$ of $V$. The connection $\nabla$ is said to be \emph{flat} if $\nabla_1\circ \nabla=0$, where $\nabla_1\colon \Omega^1_{X/S}\otimes_{\mathcal{O}_X}V\to \Omega^2_{X/S}\otimes_{\mathcal{O}_X}V$ is defined by 
\[\nabla_1(\omega\otimes x)=\mathrm{d}\omega\otimes x-\omega\wedge\nabla(x).\]

The following result allows us to study those objects as torsors on de Rham spaces.

\begin{proposition}
    Let $p\colon X\to S$ be a smooth morphism of schemes, and let $\iota_p\colon X\to \normalfont X_\text{dR}\times_{S_\text{dR}} S$ be the induced morphism as in Proposition~\ref{de rham is an epimorphism of presheaves}. For an integer $n$, there exists an equivalence of groupoids (indicated by the dashed arrow below) making the diagram
\[\begin{tikzcd}[column sep=tiny]
	\begin{array}{c} \left\{\begin{array}{c}       \mathrm{GL}_n\text{-torsors}\\ \text{over }\normalfont X_\text{dR}\times_{S_\text{dR}}S  \end{array}\right\} \end{array} &&& \begin{array}{c} \left\{\begin{array}{c}       \text{Rank }n \text{ vector bundles on }X\\ \text{ with flat connection relative to }S \end{array}\right\} \end{array} & {(V,\nabla)} \\
	\begin{array}{c} \left\{\begin{array}{c}       \mathrm{GL}_n\text{-torsors}\\ \text{over }X  \end{array}\right\} \end{array} &&& \begin{array}{c} \left\{\begin{array}{c}       \text{Rank }n \text{ vector}\\ \text{bundles on }X  \end{array}\right\} \end{array} & V
	\arrow["\sim", dashed, from=1-1, to=1-4]
	\arrow["{\iota_p^*}"', from=1-1, to=2-1]
	\arrow[from=1-4, to=2-4]
	\arrow[maps to, from=1-5, to=2-5]
	\arrow["\sim", from=2-1, to=2-4]
\end{tikzcd}\]
commute up to isomorphism. Moreover, this equivalence is functorial in $p$. For $n=1$, the equivalence is symmetric monoidal with respect to the contracted product of torsors and the tensor product of connections.
\end{proposition}

\begin{proof}
Equip the category $\mathsf{Sh}((\mathsf{Sch}/S)_\text{fppf})$ with its canonical topology, in which coverings are given by jointly epimorphic families. By Proposition~\ref{epimorphism of etale sheaves}, the morphism $\iota_p$ is an epimorphism, and since torsors form a stack on $\mathsf{Sh}((\mathsf{Sch}/S)_\text{fppf})$, the groupoid of $\mathrm{GL}_n$-torsors over $X_\text{dR}\times_{S_\text{dR}}S$ is equivalent to the groupoid of descent data $\operatorname{Desc}(\iota_p)$.

An object of $\operatorname{Desc}(\iota_p)$ consists of a $\mathrm{GL}_n$-torsor $V$ over $X$ together with an isomorphism
\[\varepsilon\colon \operatorname{pr}_1^*V \to \operatorname{pr}_2^*V,\]
satisfying the cocycle condition $\operatorname{pr}_{12}^*(\varepsilon)\circ \operatorname{pr}_{23}^*(\varepsilon)=\operatorname{pr}_{13}^*(\varepsilon)$. Here we set $Y\colonequals X_{\text{dR}}\times_{S_{\text{dR}}} S$, and we write $\operatorname{pr}_i$ (resp. $\operatorname{pr}_{ij}$) for the canonical projection to the $i$-th factor (resp. to the $(i,j)$-factors) in the simplicial diagram
\[X\times_Y X\times_Y X\rightrightrightarrows X\times_Y X\rightrightarrows X.\]
As in the proof of Corollary~\ref{de rham space is a coequalizer}, the projection maps appearing above can be identified with the canonical projections
\[(\widehat{X\times_S X\times_S X})_\Delta
\rightrightrightarrows
(\widehat{X\times_S X})_\Delta
\rightrightarrows
X,\]
where $(\widehat{-})_\Delta$ denotes the formal completion along the diagonal.

A morphism $\varphi\colon (V,\varepsilon)\to (V^\prime,\varepsilon^\prime)$ in $\operatorname{Desc}(\iota_p)$ consists of a morphism of $\mathrm{GL}_n$-torsors $\varphi\colon V\to V^\prime$ making the diagram
\[\begin{tikzcd}
	{\operatorname{pr}_1^*V} && {\operatorname{pr}_1^*V^\prime} \\
	{\operatorname{pr}_2^*V} && {\operatorname{pr}_2^*V^\prime}
	\arrow["{\operatorname{pr}_1^*(\varphi)}", from=1-1, to=1-3]
	\arrow["\varepsilon"', from=1-1, to=2-1]
	\arrow["{\varepsilon^\prime}", from=1-3, to=2-3]
	\arrow["{\operatorname{pr}_2^*(\varphi)}", from=2-1, to=2-3]
\end{tikzcd}\]
commute. Using the equivalence between $\mathrm{GL}_n$-torsors and rank $n$ vector bundles, we conclude that $\operatorname{Desc}(\iota_p)$ is equivalent to the groupoid of rank $n$ vector bundles on $X$ endowed with a stratification relative to $S$ in the sense of \cite[Def.~2.10]{berthelot2015notes}. The result then follows from \cite[Thm.~2.15]{berthelot2015notes}.
\end{proof}

In this paper, the preceding result will be used primarily through the following immediate corollary. For the reader's convenience, we note that $X_\text{dR}\times S\simeq (X\times S)_\text{dR}\times_{S_\text{dR}} S$.

\begin{corollary}\label{fundamentalresultonderhamspaces}
Let $X$ be a smooth scheme over a field $k$ of characteristic zero, and let $S$ be a $k$-scheme. Then the fppf cohomology group
\[\mathrm{H}^1(X_{\normalfont\text{dR}}\times S,\mathbb{G}_m)\]
classifies isomorphism classes of line bundles on $X$ endowed with an integrable connection relative to $S$. The group law corresponds to the tensor product of connections.
\end{corollary}

\nocite{*}
\printbibliography

\textsc{Gabriel Ribeiro, Department of Mathematics, ETH Zurich, 8092 Zurich, Switzerland}

\textit{Email address:} \href{mailto:gabriel.ribeiro@math.ethz.ch}{gabriel.ribeiro@math.ethz.ch}

\end{document}